\documentclass[a4paper]{amsart} % ,9pt
\usepackage{amsmath}
\usepackage{amssymb}%This gives us symbols like \square
\usepackage{verbatim}%This gives us the \begin{comment} environment.
\usepackage{amsthm}%This gives us Definition style of theorem.
\usepackage{mathrsfs}%This gives us \mathscr#
\usepackage{graphicx}
\usepackage{mathtools}
\usepackage{caption}
\usepackage{subcaption}
\usepackage{url}
\usepackage{epstopdf}
\usepackage{pgf,tikz}
\usepackage{todonotes}
\usepackage{bbm}
\usepackage[notref,notcite,color]{showkeys}
\usepackage{extarrows}
\usetikzlibrary{arrows}
\usepackage{algorithm}
 \usepackage{algorithmicx}
\usepackage{amssymb,amsthm}    % amssymb package contains more mathematical symbols
\usepackage{graphicx}          % graphicx package enables you to paste in graphics
\usepackage{amsmath}         % best maths package on offer
\usepackage{caption}
\usepackage{subcaption}
\usepackage{dsfont}
\usepackage{listings}
\usepackage{bbm}
\usepackage{color}
\usepackage{ dsfont }
\usepackage{enumerate}
 \usepackage{relsize}
\usepackage{xfrac}

\renewcommand{\Delta}{\triangle}
\newcommand{\bbZ}{\mathbb Z}
\newcommand{\bbE}{\mathbb E}

\newcommand{\bbR}{\mathbb R}
\newcommand{\cA}{\mathcal{A}}

\DeclarePairedDelimiter\floor{\lfloor}{\rfloor}

\definecolor{darkblue}{rgb}{0,0,0.7}

\definecolor{darkgreen}{rgb}{0.01,0.75,0.24}

\def \Ee[#1]{\mathcal{E}^{\text{{#1}}}}

\def\pa[#1,#2]{\frac{\partial {#1}}{\partial {#2}} }

\def\idom[#1,#2,#3]{\int_{#1}\hspace{1pt} {#2} \hspace{1pt} \text{d}{#3}}
\def\res[#1,#2]{\left.{#1}\right|_{#2}}

\def\var[#1,#2]{\langle \delta \mathcal{E}^{\text{{#1}}}({#2}),v\rangle}
\def\vars[#1,#2,#3]{\langle \delta^2\mathcal{E}^{\text{{#1}}}({#2})v,{#3}\rangle}
\def\vard[#1,#2,#3,#4]{\langle \delta\mathcal{E}^{\text{{#1}}}({#2})-\delta\mathcal{E}^{\text{{#3}}}({#4}),v\rangle}

\newcommand{\calI}{\mathcal{I}}

\newcommand{\calB}{\mathcal{B}}

\newcommand{\bI}{\mathbf{I}}

\newcommand{\calA}{\mathcal{A}}
\newcommand{\fraku}{\mathfrak{u}}

\DeclareMathOperator*{\argmin}{arg\,min}

\newcommand{\be}{\begin{equation}}
\newcommand{\en}{\end{equation}}
\newcommand{\ben}{\begin{equation*}}
\newcommand{\enn}{\end{equation*}}
\newcommand{\bea}{\begin{aligned}}
\newcommand{\ena}{\end{aligned}}

\def\ba#1\ena{\begin{align}#1\end{align}}
\def\ban#1\enan{\begin{align*}#1\end{align*}}

\theoremstyle{plain}
\newtheorem{thm}{Theorem}[section]

\newtheorem{lem}[thm]{Lemma}

\newtheorem{cor}[thm]{Corollary}
\newtheorem{prop}[thm]{Proposition}
\newtheorem{assumption}[thm]{Assumption}

\newtheorem{proposition}[thm]{Proposition}

\newtheorem{remark}[thm]{Remark}

\numberwithin{equation}{section}

\begin{document}

\title[Tikhonov Regularization Within Ensemble Kalman Inversion]{Tikhonov 
Regularization Within\\ Ensemble Kalman Inversion}
\author[N. K. Chada] {Neil K. Chada}
\address{Department of Statistics and Applied Probability, National University of Singapore, 119077, Singapore}
\email{neil.chada@nus.edu.sg}

\author[A. M. Stuart] {Andrew M. Stuart}
\address{Computing and Mathematical Sciences, California Institute of Technology, Pasadena, CA99125, USA}
\email{astuart@caltech.edu}

\author[X. T. Tong] {Xin T. Tong}
\address{Department of Mathematics, National University of Singapore, 119077, Singapore}
\email{mattxin@nus.edu.sg}

\begin{abstract}
{Ensemble Kalman inversion is a parallelizable methodology
for solving inverse or parameter estimation problems. Although it is based
on ideas from Kalman filtering, it may be viewed as
a derivative-free optimization method. In its most basic form it regularizes
ill-posed inverse problems through the subspace property: the solution
found is in the linear span of the initial ensemble employed. In this work we 
demonstrate how further regularization can be imposed, incorporating prior 
information about the underlying unknown. In particular we study how to
impose  Tikhonov-like Sobolev penalties. As well as introducing 
this modified ensemble Kalman inversion methodology, we also 
study its continuous-time limit, proving ensemble collapse; in the
language of multi-agent optimization this may be viewed as reaching consensus.
We also conduct a suite of numerical experiments to highlight the benefits of 
Tikhonov regularization in the ensemble inversion context.}
\end{abstract}

\maketitle

regularizartion, long-term behaviour \\

\section{introduction}
\label{sec:intro}

Inverse problems are ubiquitous in science and engineering. They occur in numerous applications, 
such as recovering permeability from measurement of flow in a porous medium \cite{LR09, ORL08}, 
or locating pathologies via medial imaging \cite{KS04}. Mathematically speaking, an inverse 
problem may be formulated as the recovery of parameter $u \in X$ from noisy data $y \in Y$ 
where the parameter $u$ and data $y$ are related by 
\begin{equation}
\label{eq:inv}
y=G(u) + \eta,
\end{equation}
$G$ is an operator from the space of parameters to observations, and $\eta$ represents noise; in
this paper we will restrict to $X,Y$ being separable Hilbert spaces. 
Inverse problems are typically solved through two 
competing methodologies: the  deterministic optimization approach 
\cite{EHNR96} and the probabilistic  Bayesian  approach \cite{KS04}. 
The former is based on defining a loss function $\ell(G(u),y)$ which one aims to minimize; a regularizer $R(u)$ that incorporates  prior information about $u$  is commonly added to improve the inversion  \cite{BB18}. The Bayesian approach 
instead views $u,y$ and $\eta$ as random variables and focusses on the
conditional distribution of $u|y$ via Bayes' Theorem as the solution; this  approach has received recent attention since it provides representation of the underlying uncertainty; it may be formulated even in the infinite-dimensional setting \cite{AMS10}. The optimization and Bayesian approaches are linked via the notion
of the maximum a posteriori (MAP) estimator through which the mode of
the conditional distribution on $u|y$ is shown to correspond to optimization of a 
regularized loss function \cite{ABDH18,DLSV13,BH15,KS04,LS18}. 

Ensemble Kalman inversion (EKI) is a recently proposed inversion
methodology that lies at the interface between  the deterministic and 
probabilistic approaches \cite{CIRS18,ILS13,ORL08}. 
It is based on the ensemble Kalman filter (EnKF) \cite{GE09,GE03,LSZ15,SR19}, which is an algorithm originally designed for high dimensional state estimation, derived by combining sequential Bayesian methods  with an approximate Gaussian ansatz. EKI applies EnKF to the inverse problem setting by introducing a trivial 
dynamics for the unknown. The algorithm works by iteratively updating an ensemble 
of candidate solutions $\{u^{(j)}_n\}_{j=1}^J$ from iteration index $n$ to $n+1$; 
here $j$ indexes the ensemble and $J$ denotes the size of the ensemble. The basic form of the
algorithm is as follows. Define the empirical means 
\[
\bar{u}_n = \frac{1}{J}\sum^{J}_{j=1}u^{(j)}_n, \quad \bar{G}_n = \frac{1}{J}\sum^{J}_{j=1}G(u^{(j)}_n),
\]
and covariances
\begin{subequations}
%\label{eq:covs}
\begin{align*}
C^{uu}_{n} & = \frac{1}{J}\sum^{J}_{j=1} \bigl(u_n^{(j)} - \bar{u}_n\bigr)\otimes 
 \bigl(u_n^{(j)} - \bar{u}_n\bigr), \quad 
C^{up}_{n} = \frac{1}{J}\sum^{J}_{j=1} \bigl(u_n^{(j)} - \bar{u}_n\bigr)\otimes ({G}
\bigl(u_n^{(j)}) - \bar{{G}}_n\bigr),\\ 
C^{pp}_{n} & = \frac{1}{J}\sum^{J}_{j=1}  
\bigl({G}(u_n^{(j)}) - \bar{{G}}_n\bigr)
\otimes   \bigl({G}(u^{(j)}_n) - \bar{{G}}_n\bigr).
\end{align*}
\end{subequations}
Then the EKI update formulae are 
\begin{equation}
\label{eq:update}
u^{(j)}_{n+1} = u^{(j)}_n + C^{up}_{n} \big(C^{pp}_n +  \Gamma\big)^{-1}\big(y^{(j)}_{n+1} - {G}(u^{(j)}_n)\big),
\end{equation}
where the artificial observations are given by
\begin{equation}
\label{eq:data}
y^{(j)}_{n+1} = y+ \xi^{(j)}_{n+1}, \quad \xi^{(j)}_{n+1} \sim \mathcal{N}(0,\Gamma')
\quad \rm{i.i.d.}.
\end{equation}
Here an implicit assumption is that $\eta$ is additive centred Gaussian noise
with covariance $\Gamma.$
Typical choices for $\Gamma'$ include $\mathbf{0}$ and $\Gamma$. 
The history of the development of the method, which occurred primarily
within the oil industry, may be found in \cite{ORL08}; the general and
application-neutral formulation of the method as presented here 
may be found in \cite{ILS13}.

For linear $G$ the method provably optimizes the standard least
squares loss function over a finite dimensional subspace \cite{SS17};
for nonlinear $G$ similar behaviour is observed empirically in \cite{ILS13}. 
However the ensemble does not, in general, accurately capture posterior 
variability; this is demonstrated theoretically in \cite{ESS15} and
numerically in \cite{ILS13,LS12}. For this reason we focus on the
perspective of EKI as a derivative-free optimization method, somewhat
similar in spirit to the paper \cite{ZNZ08} concerning
the EnKF for state estimation. Viewed in this way EKI may be seen
as part of a wider class of tools based around multi-agent
interacting systems which aim to optimize via consensus
\cite{PTTM17}. Within
this context of EKI as an optimization tool for inversion,
a potential drawback is the issue of how to incorporate 
regularization. It is demonstrated in \cite{LR09,ILS13} that
the updated ensemble lies within the linear span of the initial ensemble
and this is a form of regularization since it restricts the solution
to a finite dimensional space. However the numerical evidence in
\cite{ILS13} demonstrates that overfitting may still occur, and this
led to the imposition of iterative regularization by
analogy with the Levenburg-Marquardt approach, a method pioneered
in \cite{MAI16}; see \cite{CIRS18} for an application of this approach.

There are a number of approaches to regularization of ill-posed
inverse problems which are applied in the deterministic optimization
realm. Three primary ones are: (i) optimization over a compact set;
(ii) iterative regularization through early stopping and (iii) 
Tikhonov penalization of the misfit. The standard EKI imposes
approach (i) and the method of  \cite{MAI16} imposes approach (ii).
The purpose of this paper is to demonstrate how 
approach (iii), Tikhonov regularization \cite{BB18,EHNR96},
may also be incorporated into the EKI. Our primary contributions are:

\begin{itemize}

\item We present a straightforward modification of the standard EKI
methodology from \cite{ILS13} which allows for incorporation of Tikhonov
regularization, leading to the TEKI (Tikhonov-EKI) approach. 

\item We study the TEKI approach analytically, building on the
continuous time analysis and gradient flow structure for EKI developed 
in \cite{SS17}; in particular we prove that, for general nonlinear 
inverse problems, the TEKI flow exhibits consensus, asymptotically, 
that is ensemble collapse; for EKI this  result is only known to be true 
in the linear case. 

\item We describe numerical experiments which highlight the benefits
of TEKI over EKI, using inverse problems arising from the eikonal equation
\cite{EDS11} and from Darcy flow \cite{ILS13}. 
%{To show TEKI is consistent amongst other models we repeat the same experiments but with the Darcy flow equation, presented in the appendix, Section \ref{sec:app}.}

\end{itemize}

The outline of the paper is as follows.
In Section \ref{sec:TEKI} we describe the TEKI methodology, introducing
the modified inverse problem which incorporates the additional regularization. 
Section \ref{sec:CTS} is devoted to the derivation of a continuous time analog of the resulting  algorithm, and we also study its properties in the case of linear inverse problems. In Section \ref{sec:NUM} we present numerical 
experiments demonstrating the benefits of using TEKI over EKI, using
an inverse problem arising from the eikonal equation. 
We conclude in Section \ref{sec:CON} with an overview and further research 
directions to consider. {The appendix, Section \ref{sec:app}, contains
supplementary material in the form of further numerical examples, analogous
to Section \ref{sec:NUM}, but replacing the eikonal inverse problem
by one based on Darcy flow; these experiments demonstrate the robustness 
of the TEKI method over different choices of inverse problems.}

\section{EKI With Tikhonov Regularization}
\label{sec:TEKI}
In this section we derive the TEKI algorithm, the 
regularized variant of the EKI algorithm which we introduce in this paper.
We start by recalling how classical Tikhonov regularization works, and 
then demonstrate how to apply similar ideas within EKI.

Assuming that we model $\eta \sim N(0,\Gamma)$ in \eqref{eq:inv} the 
resulting loss function is in the $L^2$ form 
\begin{equation}
\label{eq:loss}
\ell_{Y}(y',y)=\frac12 \| \Gamma^{-1/2}(y'-y)\|_Y^2.
\end{equation}
Recall (see the previous section)
that EKI minimizes 
\begin{equation}
\label{eq:loss2}
\ell_{Y}(G(u),y)=\frac12 \| \Gamma^{-1/2}(G(u)-y)\|_Y^2
\end{equation}
within a subspace defined by the initial ensemble,
provably in the linear case and with similar behaviour observed
empirically in the nonlinear case. 

Tikhonov regularization is associated with defining 
\begin{equation}
\label{eq:reg}
R(u) = \frac{{\boldsymbol{\lambda}}}{2} \|u\|_K^2,
\end{equation}
where $K$ is a Hilbert space which is continuously and compactly embedded
into $X$, and minimizing the sum of $\ell(G(u),y)$ and $R(u).$ {The regularization parameter ${\boldsymbol{\lambda}}>0$ may be tuned to trade-off between 
data fidelity and parsimony, thereby avoiding overfitting.} This may be 
connected to Bayesian regularization if the prior on $u$
is the Gaussian measure $N(0,\boldsymbol{\lambda}^{-1}C_0)$, with $C_0$ trace-class and
strictly positive-definite on $X$. 
Then $K$ is a Hilbert space $K$ equipped with inner product  $\langle C_0^{-\frac12}\cdot, C_0^{-\frac12}\cdot \rangle_X$ and norm 
$\|\cdot\|_{K}=\|C_0^{-\frac12}\cdot\|_X;$ it 
is known as the Cameron-Martin space associated with the Gaussian prior. 
Minimizing  the sum of $\ell(G(u),y)$ and $R(u)$ corresponds to finding a
mode of the distribution \cite{DLSV13}.

To incorporate such prior information into the EKI algorithm
we proceed as follows. We first
extend \eqref{eq:inv} to the equations
\begin{subequations}
\begin{align}
\label{eq:inv_re1}
y&=G(u) + \eta_1, \\ 
\label{eq:inv_re2}
0&= u+\eta_2, 
\end{align}
\end{subequations}
where $\eta_1, \eta_2$ are independent random variables
distributed as $\eta_1 \sim N(0,\Gamma), \eta_2 \sim N(0, {\boldsymbol{\lambda}}^{-1}C_0).$ Let $Z=Y \times X$, we then define the new variables $z,\eta$ 
and mapping $F: X \times X \mapsto Z$ as follows:
\[
z=\begin{bmatrix}
y\\
0
\end{bmatrix},\quad
F(u)=\begin{bmatrix}
{G}(u)\\
u
\end{bmatrix},
\quad
\eta=\begin{bmatrix}
\eta_1\\
\eta_2
\end{bmatrix},
\]
noting that then
\[
\eta \sim N(0,\Sigma), \quad
\Sigma =
\begin{bmatrix}
\Gamma & 0\\
0 & {\boldsymbol{\lambda}}^{-1}C_0
\end{bmatrix}.
\]

We then consider the inverse problem
\begin{equation}
\label{eq:inv_re}
  z = F(u) + \eta
\end{equation}
which incorporates the original equation \eqref{eq:inv} via (\ref{eq:inv_re1}) and
the prior information via (\ref{eq:inv_re2}).
We now define the ensemble mean
\[
\bar{F}_n = \frac{1}{J}\sum^{J}_{j=1}F(u^{(j)}_n),
\]
and covariances
\[
B^{up}_{n} = \frac{1}{J}\sum^{J}_{j=1} \bigl(u_n^{(j)} - \bar{u}_n\bigr)\otimes ({F}
\bigl(u_n^{(j)}) - \bar{{F}}_n\bigr), \quad B^{pp}_{n} = \frac{1}{J}\sum^{J}_{j=1}
\bigl({F}(u_n^{(j)}) - \bar{{F}}_n\bigr)
\otimes   \bigl({F}(u^{(j)}_n) - \bar{{F}}_n\bigr).
\]
The TEKI update formulae are then found by applying the EKI
algorithm to \eqref{eq:inv_re} to obtain
\begin{equation}
\label{eq:updateT}
u^{(j)}_{n+1} = u^{(j)}_n + B^{up}_{n} \big(B^{pp}_n +  \Sigma\big)^{-1}\big(z^{(j)}_{n+1} - {F}(u^{(j)}_n)\big),
\end{equation}
where
\begin{equation}
\label{eq:dataT}
z^{(j)}_{n+1} = z+ \zeta^{(j)}_{n+1}, \quad \zeta^{(j)}_{n+1} \sim 
\mathcal{N}(0,\Sigma')
\quad \rm{i.i.d.}
\end{equation}
Typical choices for $\Sigma'$ are $\mathbf{0}$ and $\Sigma$. 
%In our analysis which will show in Section \ref{sec:CTS}, we will make the common assumption that $\Sigma'=0$. 
Notice that the resulting $L^2$ loss function \eqref{eq:loss} is, in this
case, 
\begin{equation}
\label{eq:loss3}
\ell_Z(z',z)=\frac12 \| \Sigma^{-1/2}(z'-z)\|_{Z}^2.
\end{equation}
leading, with $z'=F(u)$, to the loss function
\begin{equation}
\label{eq:func1}
\calI(u;y) := \frac{1}{2}\| \Gamma^{-1/2}(y - G(u))\|_X^2 + \frac{{\boldsymbol{\lambda}}}{2}  \|u\|_K^2.
\end{equation}
It is in this sense that TEKI regularizes EKI, the latter being associated
with the unregularized objective function \eqref{eq:loss2}.

\begin{remark}
\label{rem:ISP}
Both the EKI algorithm \eqref{eq:update} and the TEKI algorithm \eqref{eq:updateT}
have the property that all ensemble members remain in the linear span of the initial
ensemble for all time. This is proved in \cite{SS17} for EKI; the proof for 
TEKI is very similar and hence not given. For EKI (resp. TEKI) it follows simply from the fact that $C_n^{up}$
(resp. $B_n^{up}$) projects onto the linear span of the current ensemble and then 
uses an induction.
\end{remark}

\section{Continuous Time Limit of TEKI}
\label{sec:CTS}

In this section we aim to study the use of Tikhonov regularization within EKI 
through analysis of a continuous time limit of TEKI.
 {For economy of notation we assume the regularization constant 
${\boldsymbol{\lambda}}$ to take the value $1$ throughout. 
This incurs no loss of generality, since one can always replace $(\boldsymbol{\lambda},C_0)$ with $(1, \boldsymbol{\lambda}^{-1}C_0)$, and the TEKI formulation remains the same. } 
 
In subsection \ref{ssec:DCTS} we derive the continuous time limit of the TEKI
algorithm whilst in subsection \ref{ssec:ACTS} we state and prove 
the general existence Theorem \ref{t:eu} for the TEKI flow. 
In subsection \ref{ssec:obs}
we demonstrate ensemble collapse of the TEKI flow, Theorem \ref{thm:Ctabove};
this shows that the ensemble members reach consensus.
We also prove two lemmas which together characterize an invariant subspace
property of TEKI flow, closely related to Remark \ref{rem:ISP}. 
Subsection \ref{ssec:const} contains derivation of
two {\em a priori} bounds on the TEKI
flow, one in the linear setting and the other in the general setting. 
In the final subsection \ref{ssec:lt} we study the long-time behaviour of TEKI flow
in the linear setting, generalizing related work on the EKI flow in \cite{SS17}.

\subsection{Derivation Of Continuous Time Limit}
\label{ssec:DCTS}
We first recall the derivation of  the continuous time limit of the EKI 
algorithm \eqref{eq:update} from \cite{SS17} as that for TEKI is very similar.
For this purpose, we set $\Gamma'=0$, rescale
$\Gamma \mapsto h^{-1}\Gamma$ so that (approximately for $h \ll 1$) 
$(C_n^{pp}+\Gamma)^{-1}\mapsto h\Gamma^{-1}.$ 
We then view $u_n^{(j)}$ as an approximation of
a continuous function $u^{(j)}(t)$ at time $t=nh$ and let $h \to 0$. 
To write down the resulting flow succintly we
let $\mathfrak{u} \in {X}^J$ denote the
collection of $\{u^{(j)}\}_{j \in \{1,\dots,J\}}.$
Now define
$$D_{jk}(\mathfrak{u}):=\langle \Gamma^{-1/2} (G(u^{(j)})-y),\Gamma^{-1/2}(G(u^{(k)})-\bar{G}) \rangle_Y,$$
where
$$\bar{u}:=\frac{1}{J}\sum_{m=1}^J u^{(m)}, \quad \bar{G}:=\frac{1}{J}\sum_{m=1}^J G(u^{(m)}).$$
The continuum limit  of \eqref{eq:update} is then
\begin{align}
\label{eq:has}
\frac{du^{(j)}}{dt}
&=-\frac1J \sum_{k=1}^J (u^{(k)}-\bar{u})\otimes (G(u^{(k)})-\bar{G})\Gamma^{-1}(G(u^{(j)})-y) \\
&=-\frac{1}{J}\sum\limits_{k=1}^J D_{jk}(\mathfrak{u})(u^{(k)}-\bar{u})
=-\frac{1}{J}\sum\limits_{k=1}^J D_{jk}(\mathfrak{u})u^{(k)}.
\end{align}
Here we used the fact that replacing $u^{(k)}(t)$ by $u^{(k)}(t)-\bar{u}(t)$ does not
change the flow since $D_{jk}(\mathfrak{u}(t))$ 
sums to zero over $k$; we will use this
fact occasionally in what follows, and without further comment.
The equations may be written as
\begin{equation}
\label{eq:E1}
\frac{d\fraku}{dt}=-\frac{1}{J}D(\fraku)\fraku
\end{equation}
for appropriate Kronecker operator
$D(\fraku) \in {\mathcal{L}}(X^J,X^J)$
defined from the $D_{jk}(\fraku).$
Note also that we hid the dependence on time $t$ in our derivation above, and we will often do so in the discussion below.

The resulting flow is insightful because it demonstrates that, in the linear
case $G(\cdot)=A\cdot$, each ensemble member undergoes a gradient descent for the 
loss function \eqref{eq:loss} preconditioned by the empirical covariance $C(\fraku)$
defined by
\begin{equation}
\label{eq:C}
C(\fraku)  = \frac{1}{J}\sum^{J}_{m=1} \bigl(u^{(m)} - \bar{u}\bigr)\otimes
 \bigl(u^{(m)} - \bar{u}\bigr);
\end{equation}
specifically we have
\begin{equation}
\label{eq:has2}
\frac{du^{(j)}}{dt}=-C(\mathfrak{u})\nabla_u \ell_Y (Au^{(j)},y). 
\end{equation}
Note that although each ensemble member performs a gradient flow,
they are coupled through the empirical covariance.

We now carry out a similar derivation
for the TEKI algorithm; doing so will demonstrate explicitly that 
the method introduces a Tikhonov regularization. 
Consider the TEKI algorithm \eqref{eq:updateT}, setting $\Sigma'=0$, rescaling
$\Sigma \mapsto h^{-1}\Sigma$ and  viewing $u_n^{(j)}$ as an approximation of
a continuous function $u^{(j)}(t)$ at time $t=nh.$ The limiting flow is
\begin{align}
\frac{du^{(j)}}{dt} 
\notag
&=
-\frac{1}{J}\sum^{J}_{k=1} \bigl(u^{(k)} - \bar{u}\bigr)\otimes ({F}
\bigl(u^{(k)}) - \bar{{F}}\bigr)\Sigma^{-1}\bigl(F(u^{(j)})-z\bigr)\\
\notag
=&
-\frac{1}{J}\sum^{J}_{k=1} \left(\left\langle \Gamma^{-1/2} (G(u^{(k)})-\bar{G})  ,\Gamma^{-1/2}(G(u^{(j)})-y)\right\rangle_Y+\langle  u^{(j)}, u^{(k)}-\bar{u}\rangle_K \right) \bigl(u^{(k)} - \bar{u}\bigr)\\
\notag
&=-\frac{1}{J}\sum\limits_{k=1}^J \Bigl(D_{jk}(\mathfrak{u})
+\langle u^{(j)}, u^{(k)}-\bar{u} \rangle_K\Bigr)(u^{(k)}-\bar{u})
\\\label{eq:has3}
&=-\frac{1}{J}\sum\limits_{k=1}^J E_{jk}(\mathfrak{u})
(u^{(k)}-\bar{u}).
\end{align}
where
$$E_{jk}(\mathfrak{u}):=D_{jk}(\mathfrak{u})
+\langle u^{(j)}, u^{(k)}-\bar{u} \rangle_K.$$
This may be written as
\begin{equation}
\label{eq:E2}
\frac{d\fraku}{dt}=-\frac{1}{J}E(\fraku)\fraku
\end{equation}
for appropriate Kronecker matrix 
$E(\fraku) \in {\mathcal{L}}(X^J,X^J)$ 
defined from the $E_{jk}(\fraku).$

We note that the flow may be written as
\begin{align}
\frac{du^{(j)}}{dt}
\label{eq:has4}
=-\frac{1}{J}\sum\limits_{k=1}^J D_{jk}(\mathfrak{u})(u^{(k)}-\bar{u})
-C(\mathfrak{u}) \nabla_u R(u^{(j)})
\end{align}
where
\begin{equation} 
\label{eq:R}
R(u)=\frac12\|u\|_K^2.
\end{equation}
So we see explicitly that the algorithm includes a Tikhonov regularization,
preconditioned by the empirical covariance $C(\mathfrak{u}).$ 
In the linear case $G(\,\cdot\,)=A\,\cdot$, define
\begin{equation}
\label{eq:J}
\calI_{\rm{linear}}(u;y)=\ell_Y(Au^{(j)},y)+R(u).
\end{equation}
Noting that,
with $\ell_Y(\cdot,\cdot)$ and $R(\cdot)$ defined by 
\eqref{eq:loss} and \eqref{eq:R}, this coincides with the
 \eqref{eq:func1} when specialized to the linear case.
In particular, we also see that the TEKI flow has the form
\begin{equation}
\label{eq:has5}
\frac{du^{(j)}}{dt}=-C(\mathfrak{u})\nabla_u \calI_{\rm{linear}}(u^{(j)};y).
\end{equation}
Each ensemble member thus undergoes  a gradient flow with respect to the Tikhonov regularized least squares loss function $\calI_{\rm{linear}}(u;y)$, preconditioned by the empirical covariance of the collection of all the ensemble members.

\begin{remark}
Additive covariance inflation, as described in \cite{SS17}, modifies the
EKI gradient flow \eqref{eq:has2} by addition of a fixed invertible
covariance matrix to the empirical covariance. In contrast \eqref{eq:has5}
fixes the empirical covariance and instead modifies the objective function
by addition of a regularizer.  
\end{remark}

\subsection{Existence For TEKI Flow}
\label{ssec:ACTS}

Recall that the Cameron-Martin space associated with the Gaussian measure
$N(0,C_0)$ on $X$ is the domain of $C_0^{-\frac12}.$ We have the following
result. 

%\red{Remove: Throughout this section we assume that $Y$ is finite dimensional
%but we impose no such restriction on $X$.}

\begin{thm}
\label{t:eu}
%$Y$ is finite dimensional, 
Suppose the initial ensemble $\{u^{(j)}(0)\}_{j=1}^J$ 
is chosen to lie in $K$ and that $G:K \mapsto Y$ is $C^1.$
Let ${\mathcal A}$ denote the linear span of $\{u^{(j)}(0)\}_{j=1}^J$ and
${\mathcal A}^J$ the $J$-fold Cartesian product of this set. 
Then equation \eqref{eq:E2} has
a unique solution in $C^1([0,T);{\mathcal A}^J)$ for some $T>0$.
\end{thm}

\begin{remark}
The same theorem may be proved for \eqref{eq:E1} under the milder assumptions
that the initial ensemble $\{u^{(j)}(0)\}_{j=1}^J$ 
is chosen to lie in $X$ itself and that  $G: X \mapsto Y$ is $C^1.$  
\end{remark}

%\begin{proof}[Short Proof of Theorem \ref{t:eu}]
%Show that if a solution exists then is satisfies the property that
%the solution remains in ${\mathcal A}$. Show that because of this
%the $X$ inner-product appearing in \eqref{eq:has3} is well-defined and
%so flow locally well-defined in ${\mathcal A}$. Then get equations for
%coefficients of representation in ${\mathcal A}$. This is finite dimensional.
%Prove existence/uniqueness of a local solution. Does this proof work? I have
%a nagging worry that we might need samples in domain of $C_0^{-1}$ but hopefully
%not. 
%\end{proof}

\begin{proof}[Proof of Theorem \ref{t:eu}]
The right hand-side of \eqref{eq:E2} is of the form $E(\fraku)\fraku$
and  $E:\calA^J \mapsto {\mathcal{L}}(\calA^J,\calA^J)$. Thus
it suffices to show that $E$ is 
differentiable at $\fraku \in \calA^J;$ then the right hand side of \eqref{eq:E2}
is locally Lipschitz as a mapping of the finite dimensional space $\calA^J$ into itself
and standard ODE theory gives a local in time solution. 
Lemma \ref{lem:checkderivate} verifies the required differentiability.
\end{proof}

\begin{lem}
\label{lem:checkderivate}
The function $E:\calA^J \mapsto {\mathcal{L}}(\calA^J,\calA^J)$
is Frechet differentiable with respect to $\fraku\in \calA^J$. 
\end{lem}
\begin{proof}
To prove this we write down the Frechet partial derivative of each component of $E$ 
with respect  to $u^{(i)}$, applied in perturbation direction $v\in \mathcal{A}$; we use
$\nabla G(u)$ to denote the Frechet derivative of $G:K \mapsto Y$ at point $u \in K$. 
Now note that 
\begin{align*}
\left\langle v,\frac{\partial}{\partial u^{(i)} } D_{jk}(\fraku)\right \rangle_K &= 
-\frac1J \langle \Gamma^{-1/2}(G(u^{(j)})-y), \Gamma^{-1/2}\nabla G(u^{(i)} ) v\rangle_Y\\
&\quad+\mathbf{1}_{i=j}\langle \Gamma^{-1/2}\nabla G(u^{(j)}) v, \Gamma^{-1/2}(G(u^{(k)})-\bar{G})\rangle_Y\\
&\quad +\mathbf{1}_{i=k}\langle \Gamma^{-1/2}(G(u^{(j)})-y), \Gamma^{-1/2}\nabla G(u^{(k)}) v\rangle_Y. 
\end{align*}
When $G$ is $C^1$, $\nabla G(u^{(i)})$ is a bounded operator from $K$ to $Y$, so the quantity above is bounded. 
Next, we define
\[
P_{jk}(\fraku):=\langle u^{(j)}, u^{(k)}-\bar{u} \rangle_K,
\]
this is finite because it is bounded above by $ \|u^{(j)}\|_K\| u^{(k)}-\bar{u}\|_K$ using the Cauchy-Schwartz
inequality. Then 
\begin{align*}
\left\langle v, \frac{\partial}{\partial u^{(i)} }P_{jk}(\fraku)\right \rangle_K &= 
-\frac1J \langle  u^{(j)}, v\rangle_K+\mathbf{1}_{i=j}\langle v, u^{(k)}-\bar{u}\rangle_K+\mathbf{1}_{i=k}\langle u^{(j)},  v\rangle_K. 
\end{align*}
It is straightforward to verify this is bounded for $v\in \calA$. 
Since $E$ is formed by summing $D$ and $P$ the proof is complete.
%Finally 
%\[
%\left\langle v, \frac{\partial}{\partial u_i } u^{(k)}-\bar{u}\right \rangle_K=-\frac1J v+\mathbf{1}_{i=k} v,
%\]
%is also bounded from $\calA$ to itself. Therefore our claim holds by the chain rule. 
\end{proof}

\subsection{Ensemble Collapse For TEKI Flow}
\label{ssec:obs}
From Theorem \ref{t:eu} we know that the vector space
$\calA$ is invariant for the TEKI flow. 
Furthermore, when restricted to $\calA$, $C_0$ is positive definite, 
so $\|\,\cdot\,\|_K=\|C_0^{-1/2}\,\cdot\,\|_X$ and $\|\,\cdot\,\|_X$ 
are equivalent norms on the vector space $\mathcal{A}$. In particular, the following constants are well defined and strictly positive
\begin{equation}
\label{eq:lambda}
\lambda_m:=\inf_{v\in \calA, \|v\|^2_X=1}\|v\|^2_K,\quad 
\lambda_M:=\sup_{v\in \calA, \|v\|^2_X=1}\|v\|^2_K.
\end{equation}
Note that $\lambda_m$ and $\lambda_M$ do depend on $\calA$, which is defined
through the initial choice of ensemble members. 

The empirical covariance $C(\fraku(t))$ can also be viewed as a matrix in the finite dimensional linear space $\calA$. The following theorem demonstrates
that its operator norm can be bounded from above uniformly in time, and
establishes asymptotic in time collapse of the ensemble, 
provided that the solution exists for all time.

\begin{thm}
\label{thm:Ctabove}
For the TEKI flow the following upper bound holds while a solution exists:
\[
\|C(\fraku(t))\|_X\leq \frac{1}{\|C(\fraku(0))\|^{-1}_X+2\lambda_m t}. 
\]
Here $\|C(\fraku(t))\|_X$ is the operator norm of $C(\fraku(t))$ on $(\calA, \|\,\cdot\,\|_X)$ and $\lambda_m$ is defined in \eqref{eq:lambda}. 
\end{thm}

\begin{proof}
Recall the dynamical system for $u^{(j)}(t)$:
\[
\frac{d}{dt} u^{(j)}=-\frac{1}{J}\sum\limits_{k=1}^J \Bigl(D_{jk}(\mathfrak{u})
+\langle u^{(j)}, u^{(k)}-\bar{u} \rangle_K\Bigr)(u^{(k)}-\bar{u}),
\]
Averaging over $j$, we have the ordinary differential equation (ODE) for $\bar{u}(t)$. Taking the difference, we have
\begin{align*}
\frac{d}{dt} (u^{(j)}-\bar{u})&=-\frac{1}{J}\sum\limits_{k=1}^J \Bigl(\langle \Gamma^{-1/2} (G(u^{(j)})-\bar{G}), \Gamma^{-1/2}(G(u^{(k)})-\bar{G})\rangle_Y
+\langle u^{(j)}-\bar{u}, u^{(k)}-\bar{u} \rangle_K\Bigr) \\ &\times (u^{(k)}-\bar{u}).
\end{align*}
Then because $C(\fraku(t))=\frac1J \sum_{j=1}^J (u^{(j)}(t)-\bar{u}(t))\otimes (u^{(j)}(t)-\bar{u}(t))$, we find that 
\begin{align*}
\frac{dC(\fraku(t))}{dt}&=-\frac2{J^2} \sum_{j,k=1}^J\langle u^{(j)}-\bar{u}, u^{(k)}-\bar{u}\rangle_K 
(u^{(k)}-\bar{u})\otimes (u^{(j)}-\bar{u})\\
&\quad -\frac2{J^2}\sum_{j,k=1}^J\langle \Gamma^{-1/2} (G(u^{(j)})-\bar{G}), \Gamma^{-1/2}(G(u^{(k)})-\bar{G})\rangle_Y 
(u^{(k)}-\bar{u})\otimes (u^{(j)}-\bar{u}).
\end{align*}
Now we consider projecting the ODE above on  a fixed $v\in X$. Denote 
\[
v_k(t)=\langle  v, u^{(k)}(t)\rangle_X,\quad \bar{v}(t)=\langle v, \bar{u}(t) \rangle_{X}.
\] 
Note that 
\[
\langle v, (u^{(k)}-\bar{u})\otimes (u^{(j)}-\bar{u}) v\rangle_X= \langle v, u^{(k)}-\bar{u}\rangle_X \langle v, u^{(j)}-\bar{u}\rangle_X=
(v^{(k)}-\bar{v})(v^{(j)}-\bar{v}).
\]
The projection of $\frac{dC(\fraku(t))}{dt}$ on $v$ is given by 
\begin{align}
\notag
&J^2\langle v , \frac{d}{dt} C(\fraku(t))  v\rangle_X\\
\notag
 &= -2\sum_{j,k=1}^J \langle u^{(j)}-\bar{u}, u^{(k)}-\bar{u}\rangle_K
(v_k-\bar{v}) \cdot (v_j-\bar{v})\\
\notag
&\quad -2\sum_{j,k=1}^J \langle \Gamma^{-1/2} (G(u^{(j)})-\bar{G}),  \Gamma^{-1/2}(G(u^{(k)})-\bar{G})\rangle_Y 
(v_k-\bar{v}) \cdot (v_j-\bar{v})\\
\label{eq:Cformulate}
%&= \sum_{j,k=1}^J \langle \Gamma^{-1} (v_j-\bar{v})(G(u^{(j)}_t)-\bar{G}), (v_k-\bar{v})(G(u^{(k)}_t)-\bar{G})\rangle_Y \\
&=-2\left\|\sum_{j=1}^J  (v_j-\bar{v})(u^{(j)}-\bar{u})\right\|^2_K 
-2\left\|\Gamma^{-1/2}\sum_{j=1}^J  (v_j-\bar{v})(G(u^{(j)})-\bar{G})\right\|^2_Y.
\end{align}
Note that 
\[
\frac1J \sum_{j=1}^J  (v_j-\bar{v})(u^{(j)}-\bar{u})=\frac1J\sum_{j=1}^J \langle v, u^{(j)}-\bar{u}\rangle_X(u^{(j)}-\bar{u})=
 C(\fraku)v,
\]
so if $v\in\calA$ 
\[
\langle v , \frac{d}{dt} C(\fraku(t))v\rangle_X\leq -2\|C(\fraku)v \|_K^2\leq -2\lambda_m\|C(\fraku(t)) v\|_X^2.
\]
Here we used that for all $v\in \calA$
\[
C(\fraku(t)) v=\frac1J\sum_{j=1}^J \langle v, (u^{(j)}-\bar{u})\rangle_X (u^{(j)}-\bar{u}) \in \calA.
\]
Consider $C(\fraku(t))$ as a matrix in $(\calA, \|\cdot\|_X)$, and let $w(t)$ be the unit-norm eigenvector with maximum eigenvalue, we observe that because 
$$0=\frac d {dt} \|w(t)\|^2_X=2\langle w(t), \frac d {dt}w(t)\rangle_X,$$ 
so 
\begin{align*}
\frac{d}{dt} \|C(\fraku(t)) \|_X=&\frac d{dt}\langle w(t), C(\fraku(t))w(t)\rangle_X\\
&=\langle w , \frac d{dt}C(\fraku) w\rangle_X+2 \langle \frac d{dt} w(t), C(\fraku(t)) w(t)\rangle_X\\
&=\langle w , \frac d{dt}C(\fraku) w\rangle_X+2 \|C(\fraku)\|_K\langle \frac d{dt} w(t) , w(t)\rangle_X \\
&\leq -2\lambda_m \|C(\fraku(t))w(t)\|^2_X=-2\lambda_m\|C(\fraku(t))\|^2_X.
\end{align*}
So $$\frac{d}{dt} \|C(\fraku(t))\|^{-1}_X=-\|C(\fraku(t))\|^{-2}_X\frac{d}{dt} \|C(\fraku(t))\|_X\geq 2\lambda_m,$$ and hence we have our claim. 
\end{proof}

\begin{remark}
The bound in the preceding theorem shows that the TEKI ensemble collapses, even
in the case of nonlinear $G$; previous collapse results for EKI concern only
the linear setting. 
The rate of collapse for each ensemble member
is $\mathcal{O}(\frac{1}{\sqrt{t}})$. In classical Kalman filter theory, upper bounds for the covariance matrix can be obtained through an observability condition. In the TEKI algorithm, the inclusion of a (prior) observation $u$ in $F(u)$ enforces the system to be observable. This provides the intuition for the upper bound we prove for 
the TEKI covariance.
\end{remark}

We conclude this subsection with a lemma and corollary which dig a little
deeper into the properties of the solution ensemble, within the invariant
subspace $\calA.$

\begin{lem}
\label{lem:onedirection}
For any $u^\bot\in X$, if $ \langle u^\bot, u^{(j)}(0)-\bar{u}(0) \rangle_X=0$ for all $j=1,\ldots,J$ then the TEKI flow will not change along the direction of $u^\bot$ while the solution exists:
\[
\langle u^\bot, u^{(j)}(t)\rangle_X =\langle u^\bot, \bar{u}(0)\rangle_X.
\]
In particular $C(\fraku(t))u^\bot\equiv \mathbf{0}$. 
\end{lem}
\begin{proof}
First of all, recall that \eqref{eq:Cformulate} holds for all $v\in X$. We let $v=u^\bot$, which leads to 
\[
0\leq \langle u^\bot, C(\fraku(t)) u^\bot\rangle_X \leq \langle u^\bot, C(\fraku(0)) u^\bot\rangle_X=\frac1J \sum_{j=1}^J \langle u^\bot, u^{(j)}(0)-\bar{u}(0)\rangle^2_X=0.
\]
Then from
\[
\langle u^\bot, C(\fraku(t)) u^\bot\rangle_X =\frac1J \sum_{j=1}^J \langle u^\bot, u^{(j)}(t)-\bar{u}(t)\rangle^2_X
\]
we find that $\langle u^\bot, u^{(j)}(t)-\bar{u}(t)\rangle_X=0$.  Next we note that 
\[
\frac{d}{dt} \langle u^\bot, u^{(j)}(t)\rangle_X=-\frac{1}{J}\sum\limits_{k=1}^J \Bigl(D_{jk}(\mathfrak{u})
+\langle  C_0^{-1/2}u^{(j)},C_0^{-1/2}( u^{(k)}-\bar{u}) \rangle_X\Bigr)\langle u^\bot, u^{(k)}-\bar{u}\rangle_X=0.
\]
So  $\langle u^\bot, u^{(j)}(t)\rangle_X=\langle u^\bot, u^{(j)}(0)\rangle_X=\langle u^\bot, \bar{u}(0)\rangle_X$. Lastly, for any fixed $v$
\[
\langle v,C(\fraku(t)) u^\bot\rangle_X=\frac1J\sum_{j=1}^J \langle v, u^{(j)}(t)-\bar{u}(t)\rangle_X\langle u^\bot, u^{(j)}(t)-\bar{u}(t)\rangle_X=0.
\]
So $C(\fraku(t)) u^\bot\equiv \mathbf{0}$.
\end{proof}
Lemma \ref{lem:onedirection} suggests that we define the following 
subspace $\calB\subseteq\calA$:
\[
\calB:=\text{span}\{u^{(j)}(0)-\bar{u}(0),j=1,\cdots,J\}.
\]
Let $P_\calB$ be the projection to $\calB$ with respect to $\|\,\cdot\,\|_X$, and 
\begin{equation}
\label{eq:ubot}
u^\bot_0:=\bar{u}(0)-P_\calB \bar{u}(0).
\end{equation}
For notational simplicity, we write $v\bot \calB$ if $\langle v, u\rangle_X=0$ for all $u\in\calB$. Then $u^\bot_0 \bot \calB$, and $u^{(j)}(0)-u^\bot_0\in \calB$ for all $j$.  By Lemma \ref{lem:onedirection}, we know for any $v\bot \calB$
\[
\langle v, u^{(j)}(t)\rangle_X=\langle v, u^{(j)}(0)\rangle_X=\langle v, u^\bot_0\rangle_X\Leftrightarrow \langle v, u^{(j)}(t)-u^\bot_0\rangle_X=0.
\]
%
%Since $\calA=\text{span} \{\calB, \bar{u}_0\}$, so rank($\calA$)-rank($\calB$) $\leq 1$. In other words,  there is a unique $u^\bot_0\in \calA$, which can be $\mathbf{0}$, such that $u^\bot_0 \bot \calB$, and $u^{(j)}-u^\bot_0\in \calB$.  
In other words, we  further improve results in Theorem \ref{t:eu} to 
\begin{cor}
\label{cor:affine}
The TEKI flow stays in the affine space $u^\bot_0+\calB$, that is 
\[
u^{(j)}(t)-u^{\bot}_0\in \calB \quad \text{while the solution exists.}
\]
\end{cor}

\subsection{A Priori Bounds On TEKI Flow}
\label{ssec:const}
In many inverse problems prior information is available in terms
of rough upper estimates on $\|u\|^2_K$, where $K$ is an appropriately chosen
Banach space. Classically Tikhonov regularization is used to achieve 
such bounds, and in this subsection we show how similar bounds may be imposed
through the TEKI flow approach. In the study of the EnKF for state
estimation some general conditions that guarantee boundedness of the 
solutions are investigated in \cite{KLS14, TMK16}. However, in general, 
EnKF-based state estimation can have catastrophic growth phenomenon 
\cite{KMT16}. For inverse problems, and TEKI in particular, the situation
is more favourable. We study the linear setting first, and then the nonlinear
case. Recall the definition \eqref{eq:J} of $\calI_{\rm{linear}}$.

\begin{proposition}
\label{lem:linearcase}
If the observation operator $G$ is linear, then 
the TEKI has a solution $u \in C([0,\infty),\cA)$ and, for all $t \ge 0$,
\[
\|u^{(j)}(t)\|^2_K\leq 2 \calI_{\rm{linear}}(u^{(j)}(0);y)
\]
\end{proposition}
\begin{proof}
Simply note that in the linear case, the TEKI flow can be written as a gradient 
flow in the form \eqref{eq:has5}, so that 
\[
\frac{d}{dt} \calI_{\rm{linear}}(u^{(j)}(t);y)=-\langle \nabla_u 
\calI_{\rm{linear}}(u^{(j)}(t);y), C(\fraku) \nabla_u \calI_{\rm{linear}}(u^{(j)}(t);y)\rangle_K\leq 0.
\]
Therefore $\calI_{\rm{linear}}(u^{(j)}(t);y)\leq \calI_{\rm{linear}}(u^{(j)}(0);y)$. This implies that 
\[
\frac12\|u^{(j)}(t)\|_K^2=R(u^{(j)}(t))\leq \calI_{\rm{linear}}(u^{(j)}(t);y)\leq \calI_{\rm{linear}}(u^{(j)}(0);y).
\]
As the solution is bounded it cannot blow-up and hence exists
for all time.
\end{proof}

It is difficult to show that TEKI flow is bounded for a general, nonlinear,
observation operator. However bounds can be achieved for a modified 
observation operator which incorporates prior upper bounds on $\|u\|_K.$
In particular if we seek solution satisfying $\|u\|_K\leq M$ for some known
constant $M$ then we define 
\[
\widetilde{G}(u)=\phi_M(\|u\|_K) G(u);
\]
here $\phi_M(x)$ is a smooth transition function satisfying $\phi_M(x)=1$ if $x<M$ and $\phi_M(x)=0$ if $x>M+1$. Using $\widetilde{G}(u)$ instead of $G$ 
is natural in situations where we seek solutions satisfying
$\|u\|_K\leq M$. To understand this setting we work in the remainder of this
section under the following assumption:
\begin{assumption}\label{aspt:bdobs}
There is a constant $M$, so that $G(u)=\mathbf{0}$ if $\|u\|_K>M+1$.
\end{assumption}

\begin{prop}
Let Assumption \ref{aspt:bdobs} hold. Then for any fixed $T$ the TEKI flow 
has unique solution $u \in C([0,\infty),\cA)$ satisfying, for every ensemble
member $j$, 
\[
\sup_{t \ge T}\|u^{(j)}(t)\|_K\leq \max\left\{\|u^{(j)}(T)\|_K, M+\sqrt{\frac{2\lambda_M J}{\lambda_mT}}+1\right\}.
\]
The constants $\lambda_m$ and $\lambda_M$ are given by \eqref{eq:lambda}. 
\end{prop}
\begin{proof}
By Theorem \ref{thm:Ctabove} we deduce that, assuming a solution exists
for all time,
$$\sup_{t \ge T}\|C(\fraku(t))\|_X<\frac1{2\lambda_mT}.$$ 
Note that, for any $j$,
\[
\frac1J  \|u^{(j)}-\bar{u}\|^4_X\leq \langle u^{(j)}-\bar{u},C(\fraku) (u^{(j)}-\bar{u})\rangle_X\leq \|C(\fraku)\|_X\|u^{(j)}-\bar{u}\|_X^2.
\]
As a consequence, assuming a solution exists for all time,
then every ensemble member $j$ satisfies
$$
\sup_{t \ge T} \|u^{(j)}(t)-\bar{u}(t)\|_X^2\leq J\sup_{t\ge T}\|C(\fraku(t))\|_X<\frac J{2\lambda_mT}.
$$
Therefore, again assuming a solution exists for all time, 
every ensemble member $j$ satisfies
\begin{equation}
\label{eq:dev}
\sup_{t \ge T}\|u^{(j)}(t)-\bar{u}(t)\|_K^2\leq \lambda_M
\sup_{t \ge T} \|u^{(j)}(t)-\bar{u}(t)\|_X^2\leq \frac {\lambda_MJ}{2\lambda_mT}.
\end{equation}
Now assume that for some ensemble member $k$ and some time $t \ge T$
we have
\begin{equation}
\label{eq:holds}
\|u^{(k)}(t)\|_K>M+2\sqrt{\frac{\lambda_M J}{2\lambda_mT}}+1.
\end{equation}
It follows from \eqref{eq:dev} with $j=k$ that, for all $t \ge T$,
$$\|u^{(k)}(t)\|_K-\|\bar{u}(t)\|_K \le \sqrt{\frac{\lambda_M J}{2\lambda_mT}}$$
and hence that
$$\|\bar{u}(t)\|_K \ge M+\sqrt{\frac{\lambda_M J}{2\lambda_mT}}+1.$$ 
Now from \eqref{eq:dev} with any $j$ we deduce that, for all $t \ge T$,
$$\|\bar{u}(t)\|_K-\|u^{(j)}(t)\|_K \le \sqrt{\frac{\lambda_M J}{2\lambda_mT}}$$
and hence that, for all ensemble members $j$,
$$\|u^{(j)}(t)\|_K \ge M+1.$$
It follows that, if \eqref{eq:holds} holds, then 
\[
D_{k \ell}(\mathfrak{u})=\langle \Gamma^{-1/2} (G(u^{(k)})-y),\Gamma^{-1/2}(G(u^{(\ell)})-\bar{G}) \rangle_Y=0.
\]
Then 
\[
\frac{d}{dt} u^{(k)}(t)=-C(\fraku) C_0^{-1} u^{(k)}\Rightarrow 
\frac{d}{dt} \|u^{(k)}(t)\|_K^2=-2\langle C_0^{-1} u^{(k)}, C(\fraku) C_0^{-1} u^{(k)}\rangle_X\leq 0. 
\] 
It follows that, for $t\geq T$, the function $t \mapsto \|u^{(k)}(t)\|_K$ 
is non-increasing whenever it is larger than $M+\sqrt{\frac{2\lambda_M J}{\lambda_mT}}+1$.  
This demonstrates the desired upper bound on the solution which, in turn,
proves global existence of a solution.
\end{proof}

\subsection{Long-time Analysis For TEKI Flow: The Linear Setting}
\label{ssec:lt}
Theorem \ref{thm:Ctabove} shows that the TEKI ensemble collapses as time evolves. As the collapse is
approached, it is natural to use a linear approximation to understand the TEKI flow.
This motivates the analysis in this subsection where we consider the linear setting  $G(u)=Au$ 
and study the asymptotic behavior of the TEKI flow.  For simplicity we 
make the following assumption, remarking that while generalization of the results below to Hilbert spaces is possible, the
setting is substantially more technical, and does not provide much more scientific
understanding.

\begin{assumption}\label{aspt:FD}
Both $X$ and $Y$ are finite dimensional spaces and matrix $C_0$ is strictly
positive-definite on $X$.
\end{assumption}

From Corollary \ref{cor:affine}, we know the TEKI flow is restricted to
the affine subspace $u^\bot_0+\calB\subset K$. 
Given this constraint, it is natural to expect the limit point of $u^{(j)}(t)$ to be of form $u^\bot_0+u^\dagger_{\calB}$, where
\[
u^\dagger_{\calB}=\argmin_{u\in \calB} \Bigl\{\|C_0^{-1/2}(u+u_0^\bot)\|^2_X+\|\Gamma^{-1/2}(A (u+u^\bot_0)-y)\|^2_Y \Bigr\}.
\]
Then the constrained-Karush-Kuhn-Tucker (KKT) condition yields that 
\[
(C_0^{-1}+A^*\Gamma^{-1}A )(u_{\calB}^\dagger+u_0^\bot)-A^*\Gamma^{-1} y=: v^\dagger\bot\calB. 
\]
Here $A^*$ is the adjoint of $A: (X,\|\,\cdot\,\|_X)\mapsto (Y,\|\,\cdot\,\|_Y)$. 

{Note that
$\Omega:=C_0^{-1}+A^*\Gamma^{-1}A$ is the posterior precision matrix of 
the Bayesian inverse problem associated to inverting $A$ subject to additive
Gaussian noise $N(0,\Gamma)$ and prior $N(0,C_0)$ on $u$.  Since we often consider elements in the subspace $\calB$, we also denote the restriction of $\Omega$ in $\calB$ as $\Omega_B$. Note that
$0<\langle u, \Omega_\calB u\rangle_X<\infty$  for all nontrivial $u\in \calB$, $\Omega_\calB$ is positive definite on $\calB$, while $\Omega_{\calB}^{-1}$ and $\Omega_\calB^{1/2}$ are both well defined. }
%In particular, if we let $P_{\calB}$ be the projection  from $K$ to $\calB$ with respect to $\|\,\cdot\,\|_X$,
% $\Omega_B:=P_\calB^* \Omega P_\calB$ is the projection of $\Omega$ on $(\calB, \|\,\cdot\,\|_X)$. $\Omega_B$ is a positive definite matrix because $C_0$ is for $B$, therefore  
%$\Omega_B^{1/2}$ and  $\Omega^{-1}_B:=P_\calB^* \Omega^{-1} P_\calB$ are all well defined. }

\begin{thm}
\label{t:lt}
Let Assumption \ref{aspt:FD} hold and assume further that $G(u)=Au$. 
Then the TEKI flow exists for all $t>0$ and the solution converges to 
$u^\bot_0+u^\dagger_{\calB}$  with rate of $O(\frac1{\sqrt{t}})$. In particular,  $e^{(j)}(t)=u^{(j)}(t)-u_0^\bot-u^\dagger_{\calB}$ is bounded by 
\[
\|e^{(j)}(t)\|^2_Z\leq \frac{m_0}{1+2m_0 t}\|e^{(j)}(0)\|^2_Z.
\]
Here $\|\,\cdot\,\|_Z$ is the norm  equivalent to $\|\,\cdot\,\|_X$ on $\calB$ given
by 
\[
\|u\|^2_Z=\|u\|^2_K+ \|\Gamma^{-1/2}Au\|^2_Y=\langle  \Omega_\calB u, u\rangle_X. 
\]
Furthermore, the constant $m_0$ is given by 
\[m_0:=\min_{u\in \calB, \|u\|_X=1}\langle \Omega_\calB^{1/2} u,C(\fraku(0)) \Omega_\calB^{1/2}u\rangle_X. 
\]
\end{thm}

Before proving the theorem we discuss how
the constraint $u^\dagger_\calB\in\calB$ changes 
the solution in relation to the unconstrained optimization. 
For that purpose, consider the unconstrained problem  
\[
u^\dagger=\argmin_{u\in K} \Bigl\{\|u+u_0^\bot\|^2_K+\|\Gamma^{-1/2}(A (u+u_0^\bot)-y)\|^2_Y\Bigr\}.
\]
(This corresponds to finding the maximum a posteriori estimator for the Bayesian
inverse problem refered to above.)  The KKT condition indicates that 
\[
\Omega (u^\dagger+u_0^\bot)=A^*\Gamma^{-1} y.
\]
{Note that $u^\dagger$ is in the space $K$, whilst
$u^\dagger_\calB$ is in the  subspace $\calB.$}
It is natural to try and understand the relationship between $u^\dagger$ 
and $u^\dagger_{\calB}$ since this sheds light on the optimal choice of
$\calB$ and hence of the initial ensembles. To this end we have:

% since we can choose the initial ensemble, so that $\calB$ is a proper finite dimensional truncation in the Karhunen-Lo\`{e}ve expansion, and  $u^\dagger_{\calB}$ represents the desired components of $u^\dagger$. 
\begin{prop}
{Under the same conditions as Theorem \ref{t:lt}}, let $P_{\calB}$ be the projection  from $K$ to $\calB$ with respect to $\|\,\cdot\,\|_X$, and $P_{\bot}=\bI-P_{\calB}$. Then $u^\dagger_{\calB}$ can be written as 
\label{prop:MAP}
\[
 u^{\dagger}_{\calB}=P_{\calB}u^\dagger+\Omega_\calB^{-1}P_{\calB} \Omega P_{\bot}u^\dagger.
\]
In particular, if $\calB$ and its orthogonal complement have no correlation through $\Omega$, that is $\langle u, \Omega v\rangle_X=0$ for all $u\in \calB$ and $v\bot \calB$, then $u^{\dagger}_{\calB}=P_{\calB}u^\dagger$. 
\end{prop}
\begin{proof}
Recall the KKT conditions, 
\[
\Omega(u^\dagger+u_0^\bot)=A^*\Gamma^{-1}y,\quad \Omega(u_{\calB}^\dagger+u_0^\bot)=A^*\Gamma^{-1}y+v^\dagger,
\]
where $v^\dagger \bot \calB$. They lead to
\[
\Omega u_{\calB}^\dagger=\Omega  u^\dagger+v^\dagger=\Omega P_{\calB}u^\dagger+\Omega P_{\bot}u^\dagger+v^\dagger.
\]
Projecting this equation into $\calB$, we find
\[
P_{\calB}\Omega P_{\calB}  u_{\calB}^\dagger=P_\calB \Omega P_\calB u^\dagger+P_\calB \Omega P_\bot u^\dagger. 
\]
{Note that for any $v_1,v_2\in \calB$, $\langle v_1,\Omega_\calB v_2\rangle_X=\langle v_1, P_\calB \Omega P_\calB v_2\rangle_X$, so $\Omega_\calB v_2= P_\calB \Omega P_\calB v_2$. Therefore we have
\[
\Omega_\calB (u_{\calB}^\dagger-P_\calB u^\dagger)=P_\calB \Omega P_\bot u^\dagger.
\]
Finally note that $\Omega_{\calB} $ is positive definite and hence invertible within $\calB$. Applying $\Omega_\calB^{-1}$ on both sides, we have our claim.}
\end{proof}

\begin{proof}[Proof of Theorem \ref{t:lt}]

We investigate the dynamics of  $e^{(j)}(t)=u^{(j)}(t)-u^\bot_0-u^\dagger_{\calB}\in \calB$. Note that
\begin{align*}
\frac{d}{dt}  e^{(j)}(t)&=-\frac{1}{J}\sum\limits_{k=1}^J \Bigl(\langle \Gamma^{-1}( Au^{(j)}-y), A(u^{(k)}-\bar{u})\rangle_Y
+\langle u^{(j)}, u^{(k)}-\bar{u} \rangle_K\Bigr)(u^{(k)}-\bar{u})\\
&=-\frac{1}{J}\sum\limits_{k=1}^J \Bigl(\langle  A^*\Gamma^{-1}( Au^{(j)}-y), u^{(k)}-\bar{u}\rangle_X
+\langle u^{(j)}, C_0^{-1}(u^{(k)}-\bar{u}) \rangle_X\Bigr)(u^{(k)}-\bar{u})\\
&= - C(\fraku) A^*\Gamma^{-1} (Au^{(j)}-y)- C(\fraku) C_0^{-1}u^{(j)}\\
&= - C(\fraku) (A^*\Gamma^{-1} A+C_0^{-1})u^{(j)}+ C(\fraku) A^*\Gamma^{-1} y\\
&=- C(\fraku) (A^*\Gamma^{-1} A+C_0^{-1})u^{(j)}+ C(\fraku)  (C_0^{-1}+A^*\Gamma^{-1}A )\Bigl(u_{\calB}^\dagger+u_0^\bot-v^\dagger\Bigr)\\
&=- C(\fraku) (A^*\Gamma^{-1} A+C_0^{-1})e^{(j)}(t)- C(\fraku)v^\dagger. 
\end{align*}
But Lemma \ref{lem:onedirection} shows that $C(\fraku) v^\dagger=\mathbf{0}$
so  we have established that
\[
\frac{d}{dt}  e^{(j)}(t)=- C(\fraku(t)) \Omega e^{(j)}(t).
\]
{
Since we know $e^{(j)}(t)\in \calB$, $C(\fraku(t))w=\mathbf{0}$ for any $w\bot \calB$, the equation above can be written as 
\[
\frac{d}{dt}  e^{(j)}(t)=- C(\fraku(t)) \Omega_\calB e^{(j)}(t).
\]}
%Finally
This leads to 
\begin{align*}
\frac12 \frac{d}{dt} \|e^{(j)}(t)\|_Z^2&=-\langle  \Omega_\calB e^{(j)}  ,C(\fraku) \Omega_\calB e^{(j)}\rangle_X\\
&=-\langle  \Omega_\calB^{\frac{1}{2}} e^{(j)}  ,D(\fraku) \Omega_\calB^{\frac{1}{2}}e^{(j)}\rangle_X,
\end{align*}
where $D(\fraku)=\Omega_\calB^{\frac12} C(\fraku) \Omega_\calB^{\frac12}$ on $\calB$. Lemma \ref{lem:Dt} in below shows that for any $v\in \calB$ with $\|v\|_X=1$, and $m_0$ as defined above, 
\[
\langle v, D(\fraku) v\rangle_X\geq \frac{1}{m^{-1}_0+2t}.
\]
%\quad m_0:=\min_{u\in \calB, \|u\|_X=1}\langle u,D(\fraku(0)) u\rangle_X. 
%\]
Therefore
\[
\frac{d}{dt} \|e^{(j)}(t)\|_Z^2\leq -\frac{2}{m^{-1}_0+2t}  \|\Omega_\calB^{1/2} e^{(j)}(t)\|_K^2
= -\frac{2}{m_0^{-1}+2t}\|e^{(j)}(t)\|_Z^2. 
\]
This leads to 
\[
\frac{d}{dt} \log \|e^{(j)}(t)\|_Z^2\leq -\frac{2}{m^{-1}_0+2t}\Rightarrow \|e^{(j)}(t)\|_Z^2\leq \frac{1}{1+2m_0t}\|e^{(j)}(0)\|_Z^2. 
\]
\end{proof}

\begin{lem}
\label{lem:Dt}
Let he same conditions as in Theorem \ref{t:lt} hold and define
$D(\fraku)=\Omega_\calB^{1/2} C(\fraku) \Omega_\calB^{1/2} $. Then given any $v\in \calB, \|v\|_Z=1$,  
\[
\langle  \Omega_\calB^{1/2} v, D(\fraku(t)) \Omega_\calB^{1/2}v\rangle_X\geq \frac{1}{m_0^{-1}+2t}.
\]
\end{lem}

\begin{proof}
Recall \eqref{eq:Cformulate} and set $G(u)=Au$ to obtain 
\begin{align*}
\langle v ,\frac{d}{dt} C(\fraku(t)) v\rangle_X&= -\frac{2}{J^2}\left\|\sum_{j=1}^J  (v_j-\bar{v})(u^{(j)}-\bar{u})\right\|^2_K 
-\frac{2}{J^2}\left\|\Gamma^{-1/2}A\sum_{j=1}^J  (v_j-\bar{v})(u^{(j)}-\bar{u})\right\|^2_Y\\
&=-2\|C_0^{-1/2}C(\fraku(t))v\|_K^2-2\|\Gamma^{-1/2}AC(\fraku(t))v\|_Y^2\\
&=-2\langle C(\fraku(t))v, (C_0^{-1}+A^*\Gamma^{-1}A)C(\fraku(t))v\rangle_X
\end{align*}
%Here we use $A_X$ to denote the mapping of $A: (\calB, \|\,\cdot\,\|_X)\mapsto (Y, \|\,\cdot\,\|_Y)$, and  $A^*$ is its adjoint. For $v\in \calB, y\in Y $, $A_X v=A v$ and $A^* y=C_0^{-1}A^* y$. 
%$\langle C_0^{-1}A^*y, v\rangle_K=\langle A^*y, v\rangle_X=(y, Av)_Y$. 
Since this is true for any $v$ we deduce that $C(\fraku(t))$ as a matrix on $\calB$ satisfies
\[
\frac{d}{dt}C(\fraku(t))=-2C(\fraku(t))(C_0^{-1}+ A^* \Gamma^{-1} A) C(\fraku(t)). 
\]
{Recall that by Lemma \ref{lem:onedirection}, $C(\fraku) v=\mathbf{0}$ for all $v\bot \calB$. As a consequence $C(\fraku)=P_\calB C(\fraku)P_\calB$, and therefore we can write
\[
\frac{d}{dt}C(\fraku(t))=-2C(\fraku(t))\Omega_\calB C(\fraku(t)). 
\] }
So by the chain rule,
\[
\frac{d}{dt} D(\fraku(t))= -2D(\fraku(t))^2. 
\]
As a consequence we find that each eigenvector $v$ of $D(\fraku(0))$ remains an 
eigenvector of $D(\fraku(t))$, and its eigenvalue $\lambda=\lambda(t)$ 
solves the ODE
\[
\frac{d}{dt}\lambda(t)=\frac{d}{dt} \langle v,D(\fraku(t))v \rangle_X= -\langle v,D(\fraku(t))^2v\rangle_X=-2\lambda^2. 
\] 
The solution is given by $\lambda(t)=\frac{1}{\lambda(0)^{-1}+2t}$. Letting $\lambda(0)$ to be minimum eigenvalue gives our claim. 
\end{proof}

\section{Numerical Experiments}
\label{sec:NUM}

In this section we describe numerical results comparing EKI with the 
regularized TEKI method. Our EKI and TEKI algorithms 
are based on time-discretizations
of the continuum limit, rather than on the discrete algorithms stated in 
Sections \ref{sec:intro} and \ref{sec:TEKI}; we describe the adaptive
time-steppers used in subsection
\ref{ssec:TD}. In subsection \ref{ssec:SD} we present the 
spectral discretization used to create prior samples, and demonstrate how to  
introduce the additional regularization of prior samples required for the 
TEKI approach. 
Subsection \ref{ssec:IEE} contains numerical experiments comparing
EKI and TEKI. The inverse problem is to find the slowness function
in an eikonal equation, given noisy travel time data. 
{We have also conducted numerical experiments for the permeability
in a porous medium equation. But because the results spell out exactly the
same message as those for the eikonal equation, we do not repeat them here;
rather we confine them to the appendix.}

\subsection{Temporal Discretization}
\label{ssec:TD}

The specific ensemble Kalman algorithms that we use are found by applying
the Euler discretization to the continuous time limit of each algorithm.
Discretizing \eqref{eq:has3} with adaptive time-step $h_n$ gives
\begin{subequations}
\label{eq:has33}
\begin{align}
u_{n+1}^{(j)}&= u_{n}^{(j)}-\frac{h_n}{J}\sum\limits_{k=1}^J E_{jk}(\mathfrak{u}_n)(u^{(k)}_n-\bar{u}_n)\\
&=u_{n}^{(j)}-\frac{h_n}{J}\sum\limits_{k=1}^J \Bigl(D_{jk}(\mathfrak{u}_n)
+\langle  C_0^{-1}u^{(j)}_n, u^{(k)}_n-\bar{u}_n \rangle_X\Bigr)(u^{(k)}_n-\bar{u}_n).
\end{align}
\end{subequations}
For the adaptive time-step we take, as implemented in \cite{KS18},
\begin{equation}
\label{eq:ts}
h_n=\frac{h_0}{\|E(\mathfrak{u}_n)\|_{F}+\delta},
\end{equation}
for some $h_0, \delta \ll 1,$ where $\|\cdot\|_F$ denotes the Frobenius norm, 
and $E$ is the matrix with entries $E_{jk}$ (rather than its Kronecker form
used earlier in \eqref{eq:E2}. The integration method for the EKI flow \eqref{eq:has} is identical, but with $E_{jk}$ replaced by $D_{jk}.$

\subsection{Spatial Discretization}
\label{ssec:SD}

We consider all inverse problems on the two dimensional spatial domain $\mathcal{D}=[0,1]^2.$ 
We let $-\triangle$ denote the Laplacian on $\mathcal{D}$ subject to homogeneous Neumann
boundary conditions. We then define
$$C_0=\left(-\triangle+\tau^2\right)^{-\alpha},$$
where $\tau \in \mathbb{R}^{+}$  denotes the inverse lengthscale 
of the random field and $\alpha \in \mathbb{R}^+$ determines the
regularity; specifically draws from the random field are H\"older with
exponent upto $\alpha-1$ (since spatial dimension $d=2$). 
From this we note that the eigenvalue problem
$$C_0\varphi_k=\lambda_k \varphi_k,$$
has solutions, for $\bbZ=\{0,1,2,\cdots\}$,
$$\varphi_k(x)=\sqrt{2}\cos(k\pi  x), \quad \lambda_k=\left(|k|^2\pi^2+\tau^2\right)^{-\alpha}, \quad k \in \bbZ_+^2.$$
Here $X=L^2(\mathcal{D},\bbR)$ and the $\varphi_k$ are orthonormal in 
$X$ with respect to the standard inner-product.
Draws from the measure $N(0,C_0)$ are given by the Karhunen-Lo\`{e}ve (KL) expansion
\begin{equation}
\label{eq:KL}
u=\sum_{k \in \bbZ_+^2} \sqrt{\lambda_k}\xi_k \varphi_k(x), \quad \xi_k \sim N(0,1)\quad\rm{i.i.d.}\, .
\end{equation}
This random function will be almost surely in $X$ and in $C(\mathcal{D},\bbR)$
provided that $\alpha>1$ and we therefore
impose this condition. 

Recall that for TEKI to be well-defined we require an initial ensemble
to lie in the Cameron-Martin space of the Gaussian measure $N(0,C_0)$. The
draws in \eqref{eq:KL} do not satisfy this criterion; indeed in infinite
dimensions samples from Gaussian measure never live in the Cameron-Martin space.
Instead we consider an expansion in the form
\begin{equation}
\label{eq:KL2}
v=\sum_{k \in \bbZ_+^2} \lambda_k^a\xi_k \varphi_k(x), \quad \xi_k \sim N(0,1)\quad\rm{i.i.d.}\ , 
\end{equation} 
and determine a condition on $a$ which ensures that such random functions lie 
in the domain of $C_0^{-\frac12}$, the required Cameron-Martin space. 
We note that
%\begin{align*}
%\bbE\|v\|_K^2 & = \bbE \|C_0^{-\frac12} v\|_X^2\\
%&= \bbE \|\sum_{k \in \bbZ_+^2} \lambda_k^{a-\frac12}\xi_k \varphi_k(x)\|_X^2\\ 
%&= \sum_{k \in \bbZ_+^2} \lambda_k^{2a-1}.
%\end{align*}
\begin{equation*}
\bbE\|v\|_K^2  = \bbE \|C_0^{-\frac12} v\|_X^2
= \bbE \|\sum_{k \in \bbZ_+^2} \lambda_k^{a-\frac12}\xi_k \varphi_k(x)\|_X^2
= \sum_{k \in \bbZ_+^2} \lambda_k^{2a-1}.
\end{equation*}

Since $\mathcal{D}$ is a two dimensional domain, the eigenvalues of the Laplacian grow asymptotically like $j$
if ordered on a one dimensional lattice $\bbZ_+$ indexed by $j$.
Thus it suffices to find $a$ to ensure 
$$\sum_{j \in \bbZ_+} j^{-\alpha(2a-1)}<\infty.$$
Hence we see that choosing
$a>\frac12+\frac{1}{2\alpha}$
will suffice.  The initial ensemble for both the EKI and TEKI is found 
by drawing functions $v$ with $a$ satisfying this inequality. 
The random function \eqref{eq:KL2} 
is H\"older with exponent upto $2a\alpha-1.$
\footnote{For non-integer $\beta$ we use the terminology that function is  
H\"older with exponent $\beta$ if
the function is in $C^{\floor*{\beta}}$ and its $\floor*{\beta}$-th derivatives are  H\"older  $\beta-\floor*{\beta}.$
In the context of this paper integer $\beta$ can be avoided because random Gaussian functions
are always  H\"older on an interval of exponents which is open from the right.
See \cite{DS16}.}

\begin{figure}[h!]
\centering
 \includegraphics[width=\linewidth]{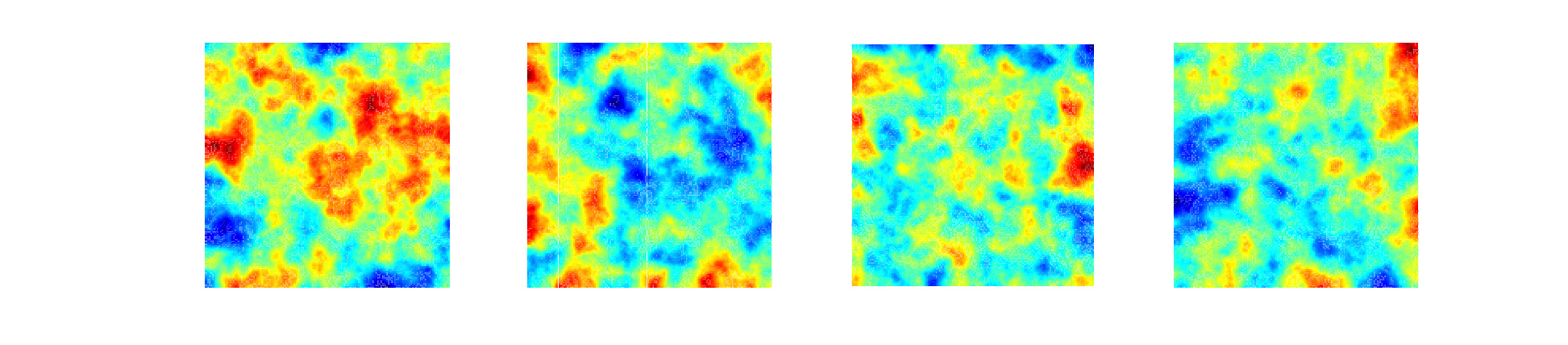}
\caption{KL draws from the prior.}
 \label{fig:EKI_draws}
\end{figure}

\begin{figure}[h!]
\centering
 \includegraphics[width=\linewidth]{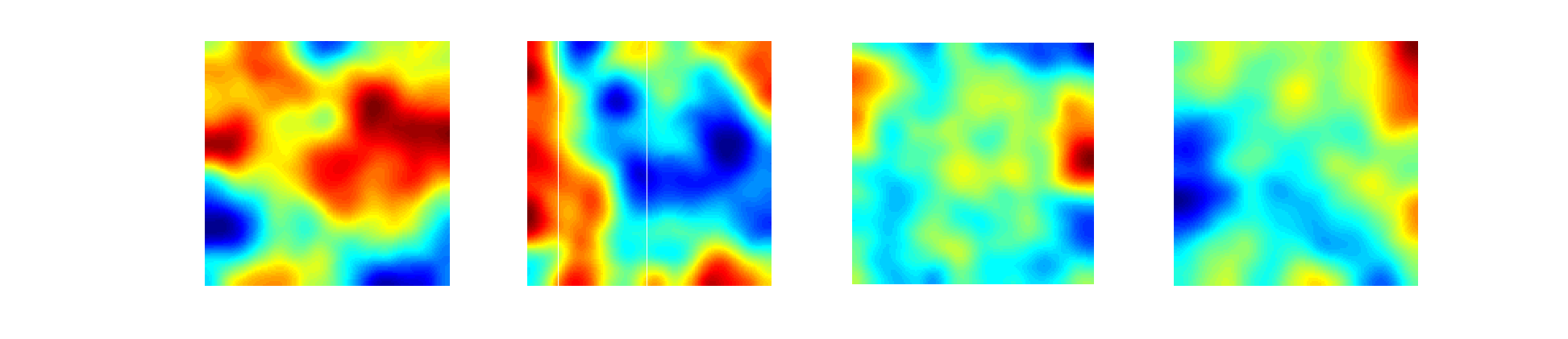}
 \caption{KL draws from the Cameron-Martin space of the prior.}
 \label{fig:TEKI_draws}
\end{figure}

To illustrate the foregoing we consider the Gaussian measure $N(0,C_0)$ 
which arises when $\alpha=2$ and with inverse lengthscale $\tau=15$. 
We study realizations from the KL expansion \eqref{eq:KL} and from
the TEKI-regularized expansion \eqref{eq:KL2} with $a=1>3/4$, using
common realizations of the random variables $\{\xi_k\}_{k \in \mathbb{Z}^2_+}$. 
Figure \ref{fig:EKI_draws} shows four random draws from the KL expansion 
\eqref{eq:KL} and Figure \ref{fig:TEKI_draws} from \eqref{eq:KL2}. 
The required higher regularity of initial samples for the TEKI method
is apparent. 
The functions in Figure \ref{fig:EKI_draws} have H\"older exponent
up to $1$, whilst those in Figure \ref{fig:TEKI_draws}
have H\"older exponent up to $3$.

\subsection{Inverse Eikonal Equation} 
\label{ssec:IEE}
We test and compare the EKI and TEKI on an inverse problem arising from
the eikonal equation. This partial differential equation arises in
numerous scientific disciplines, and in particular in seismic travel 
time tomography. Given a slowness or inverse velocity function 
$s(x) \in C^0(\bar{\mathcal{D}})$, characterizing the medium, and a source
location $x_0 \in \mathcal{D}$, the 
forward eikonal equation is to solve for travel time 
$T(x) \in C^0(\bar{\mathcal{D}})$ satisfying  
\begin{align}
\label{eq:eikonal}
|\nabla T(x)| &= s(x), \quad x \in \mathcal{D}\setminus \{x_0\} , \\
\label{eq:bc}
T(x_0) &= 0, \\
\label{eq:soner}
\nabla T(x) \cdot \nu(x) & \geq 0, \quad x \in \partial \mathcal{D}.
\end{align}
The forward solution $T(x)$ represents the shortest travel time from $x_0$ to a point in the domain $\mathcal{D}$. The  Soner boundary condition 
\eqref{eq:soner} imposes wave propagates along the unit outward normal 
$\nu(x)$ on the boundary of the domain. For the slowness function $s(x)$ we assume the positivity $s(x)>0$ which ensures well-posedness.
The unique solution can be characterized via the minimization procedure found in \cite{PLC82}.

The inverse problem is to determine the speed function $s$ from measurements
(linear mollified pontwise functionals $l_j(\cdot)$) of the travel time function $T$; for example
we might measure $T$ at specific locations in the domain $\bar{\mathcal{D}}.$
In order to ensure positivity of the speed function during inversion
we write $s=\exp(u)$ and invert for $u$ rather than $s$.
The data is assumed to take the form
\begin{equation}
\label{eq:func}
 y_j = l_j(T) + \eta_j, \quad j=1,\cdots,J, 
\end{equation}
where the $\eta_j$ are Gaussian noise, assumed independent, mean zero and
covariance $\Gamma$.
By defining $G_j(u) = l_j(T)$, we can rewrite \eqref{eq:func} as the inverse problem 
\begin{equation}
y = G(u) + \eta, \quad \eta \sim N(0,\Gamma).
\end{equation}
Further details on the well-posedness of the forward and inverse eikonal equation can be found by Elliott et al. in \cite{EDS11}.

The discretization of the forward model is based on a fast marching method 
\cite{EDS11,JAS99}, employing a uniform mesh with spacing $h_{*}=0.01$. 
On the left-hand boundary we choose $5$ random source points with $64$ equidistant pointwise measurements in the domain. For the inversion, we choose $\Gamma = \gamma^2I$ with $\gamma=0.01$. We fix the ensemble size at $J=100$ and the maximum number
of iterations at $23.$ To define the adaptive time stepping procedure
we take $h_0 = 0.02$ and $\delta=0.05$. 

Recall that the initial ensemble for EKI and TEKI, when chosen at random,
differ in terms of regularity: TEKI draws lie in the Cameron-Martin space
and hence are more regular than those for EKI. In order to thoroughly compare
the methodologies we will consider three different truth functions 
$u^{\dagger}$, one each matching the regularities of the EKI and TEKI draws
respectively,
and one with regularity lying between the regularities of the two EKI
and TEKI initializations. The EKI draws in each of cases 1, 2 and
3 are found by taking $\alpha=2$ (and by definition $a=0.5$) 
and the TEKI draws by taking $\alpha=2$ and $a=1.$ The truth in each
case is found by taking $\alpha=2, a=0.5$ (case 1), $\alpha=3.2, a=0.5$ 
(case 2) and $\alpha=2, a=1.$ (case 3). The resulting maximal H\"older
exponents are shown in Table \ref{table:1}. (Strictly speaking the maximal
 H\"older regularity is any value less than or equal to that displayed
in the table.)  
We will also study the EKI and
TEKI methods when initialized with the same initial ensemble, namely
the Karhunen-Lo\'{e}ve eigenfuctions $\varphi_k.$

\begin{table}[h!]
\centering
\begin{tabular}{||c c c c ||} 
 \hline
 Case & EKI & $u^{\dagger}$ & TEKI \\ [0.5ex] 
 \hline\hline
 \textbf{1.} & $1$ & $1$  &  $3$ \\ 
 \textbf{2.} & $1$ & $2.2$ &  $3$ \\
 \textbf{3.} & $1$ & $3$ &  $3$ \\ [1ex] 
 \hline
\end{tabular}
\bigskip
\caption{Maximal H\"older exponent for EKI and TEKI initial draws and truth 
$u^{\dagger}.$}
\label{table:1}
\end{table}

In addition to experiments where the initial ensembles are drawn at random 
from \eqref{eq:KL} (for EKI) and from \eqref{eq:KL2} (for TEKI) we
also consider experiments where the initial ensemble comprises
the eigenfunctions 
\begin{equation}
\label{eq:kl_basis}
u^{(j)}(x)=\varphi_{j}(x), \quad j=1,\ldots,J,
\end{equation}
and so it is the same for both EKI and TEKI.
The first motivation for using the eigenfunctions is to facilitate
a comparison between EKI and TEKI when they both use the same 
initial regularity, in contrast to the differing 
regularities in Table \ref{table:1}. The second motivation is that the  
choice of working with eigenfunctions, rather than random draws, has 
been show to guard against overfitting for EKI \cite{ILS13}.

To assess the performance of both methods for each case we consider analyzing this through two quantities, the relative error and the data misfit. These are defined, for EKI, as 
\begin{equation*}
\frac{\| u_{{\tiny{\rm EKI}}} - u^{\dagger}\|_{L^2(\mathcal{D})}}{\| u^{\dagger} \|_{L^2(\mathcal{D})}}, \quad \big\| y - {G}(u_{\rm EKI})\big\|_{\Gamma},
\end{equation*}
and similarly for TEKI. When we evaluate these error and misfit measures, 
we will do so by employing the mean of the current ensemble. To see the effect of overfitting, we use the noise level $\| \eta \| = \| y - G(u^{\dagger})\|$ as a benchmark. Throughout the experiments we show a progression through the $n=23$ iterations, which will be represented through $5$ sub-images related to the $(1^{\textrm{st}},5^{\textrm{th}},11^{\textrm{th}},17^{\textrm{th}},23^{\textrm{th}})$ iterations, ordered from the top left 
to the bottom right. The first image, at step $1$, is simply a single draw
from the initial ensemble; the remaining four images show the mean of the
ensemble at steps $5,11,17$ and $23$. For the KL basis the image shown
at step $1$ is hence just {one of the eigenfunctions $\varphi_j.$ 
As mentioned all of the numerics will be split into the 3 test cases as
described in Table \ref{table:1}.  

\begin{remark}
{We note that for the purposes of all the results presented
we set ${\boldsymbol{\lambda}} =1$. We have conducted additional 
experiments for other values, including 
${\boldsymbol{\lambda}}  = 0.1,10$ leading to no
qualitatively different behaviour than seen here. 
However in general it will be of interest to learn the parameter $\lambda$ 
as is standard in the solution of ill-posed inverse problems
\cite{EHNR96,BB18,GW85}. We do not focus on this question here, however,
as it distracts from the main message of the paper.}
\end{remark}

\subsubsection{Case 1.} 

Our first case corresponds to the first row of Table \ref{table:1},
as well as experiments in which both EKI and TEKI are
initialized with the KL eigenfunctions. The truth  is provided in Figure \ref{fig:truth_1}. We see no evidence of overfitting
and we notice the TEKI solutions outperform the EKI solutions,
and that the KL-initialized solutions are less accurate than those
found from TEKI using random draws to initialize the ensemble: 
see Figure \ref{fig:RE_DM_1}. 
Figures  \ref{fig:EKI_iter_low}--\ref{fig:TEKI_iter_low2} demonstrate the 
progression of the method in each case. As the iteration progresses we start to see differences in reconstruction for both EKI and TEKI. The regularity of the truth 
and the EKI initial ensemble match creating a superficial similarity in this
case; however the TEKI outperforms EKI despite this. When initializing with 
the KL basis, we notice a similar behaviour for both TEKI and EKI. However the added regularization for TEKI over EKI is manifest in a smaller error.

\newpage

% We see firstly from the numerics conducted, with random draws from the prior, that TEKI attains a higher form of regularity. As the iteration progresses we start to see differences in reconstruction for both EKI and TEKI in Figures \ref{fig:EKI_iter_low} and \ref{fig:TEKI_iter_low}. As the value of the regularity for the truth is the same as the initial ensemble for EKI, we see a more similar reconstruction. However from the figures it suggests that TEKI outperforms EKI despite this difference.

%When the initial ensemble is represented through the KL basis, we notice a similar performance for both TEKI and EKI. However the added regularization from TEKI is evident through Figures \ref{fig:EKI_iter_low2} and  \ref{fig:TEKI_iter_low2}, where we see an improvement. It can be observed that EKI with the KL basis outperforms EKI with random draws. This is agrees with the results obtained in \cite{ILS13}. Figure \ref{fig:RE_DM_1}, which shows the relative errors and data misfits, highlight the phenomenom obtained in Figures  \ref{fig:EKI_iter_low}  - \ref{fig:TEKI_iter_low2}.

\begin{figure}[h!]
\centering
\includegraphics[scale=0.25]{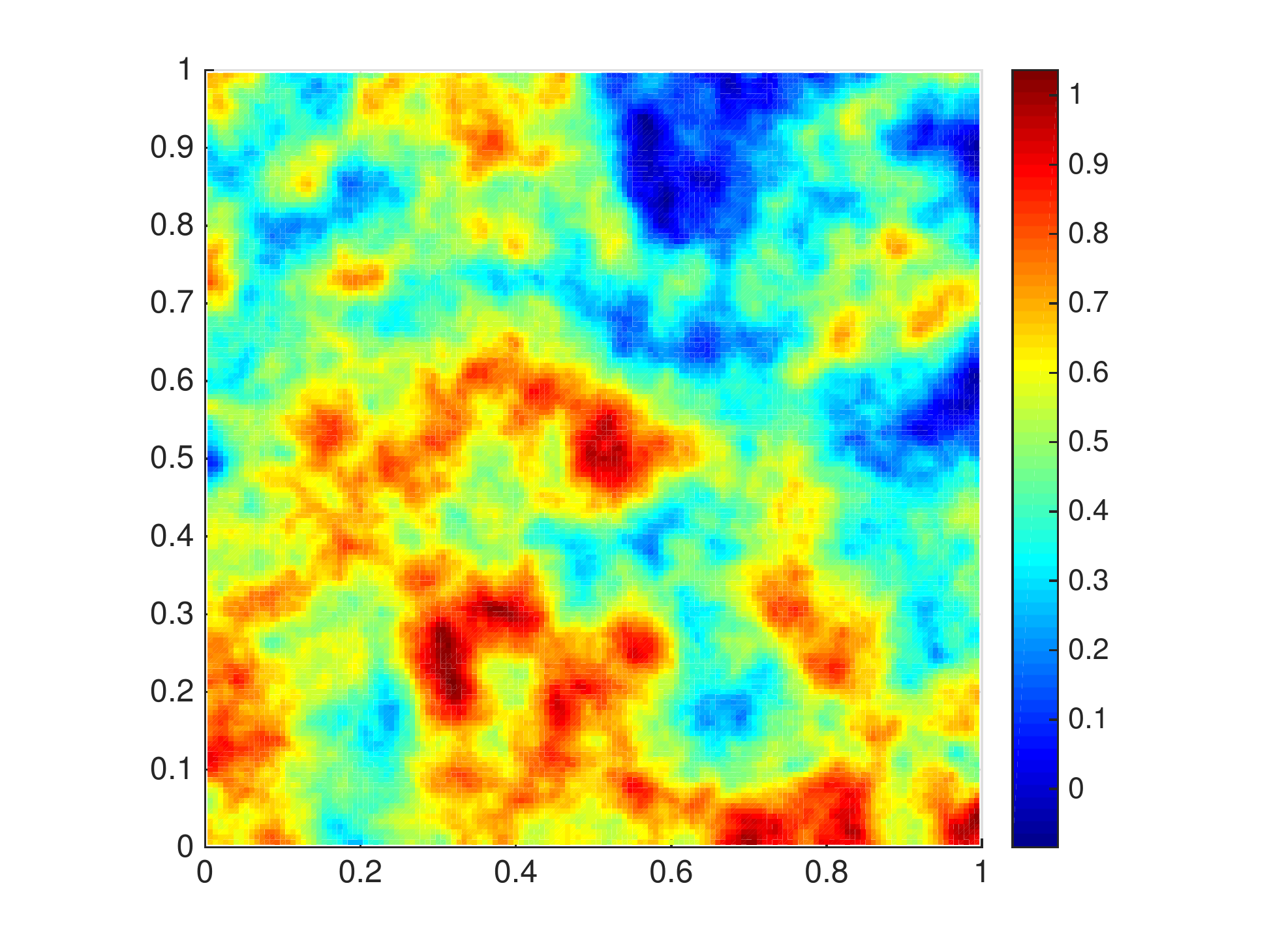}
\caption{Case 1. Gaussian random field truth.}
 \label{fig:truth_1}
\end{figure}

\begin{figure}[h!]
\centering
\includegraphics[width=\linewidth]{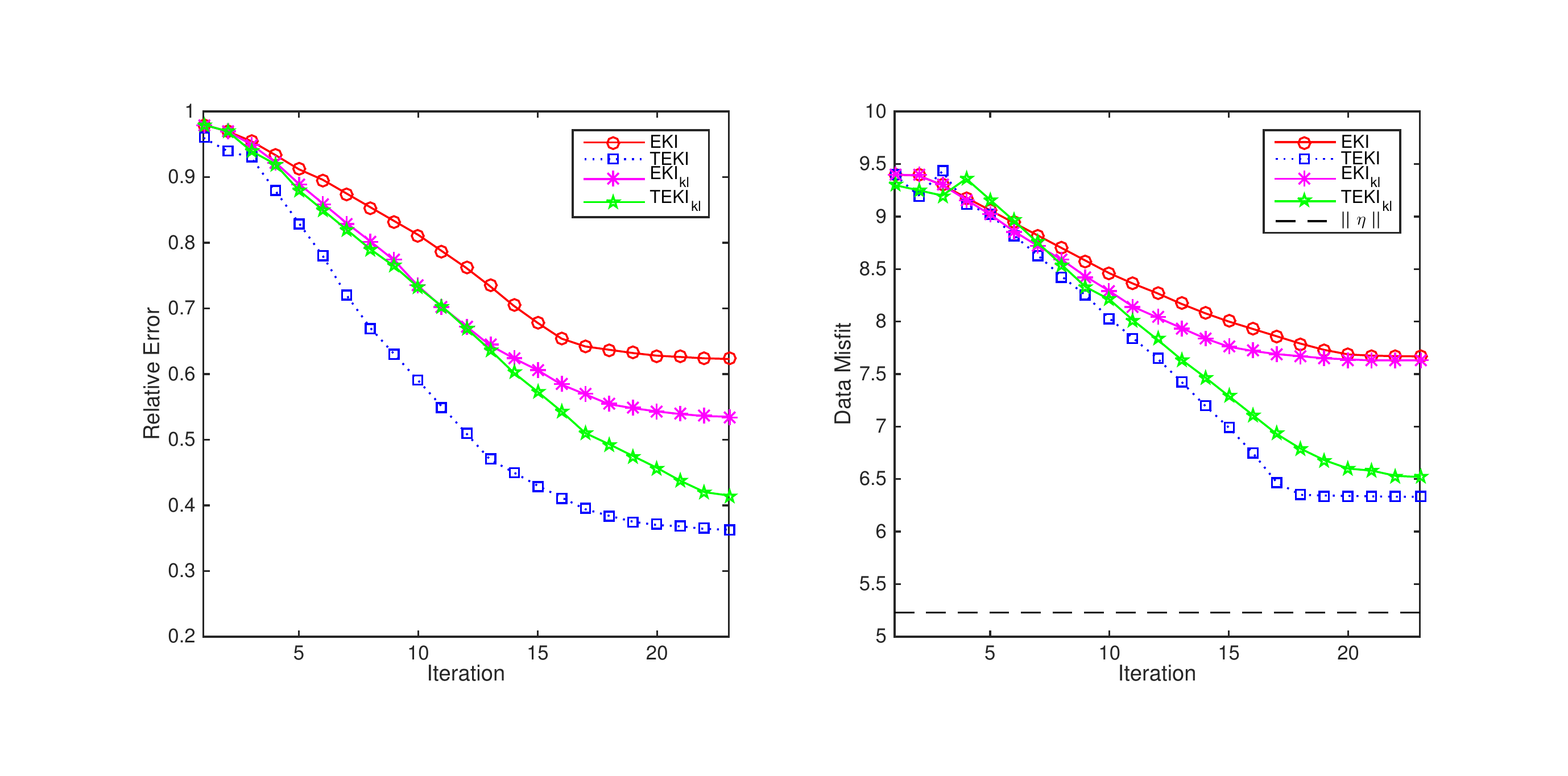}
\caption{Case 1. Relative errors and data misfits of each experiment.}
 \label{fig:RE_DM_1}
\end{figure}

\subsubsection{Case 2.} 
Our second test case compares both methods when the regularity of the truth is 
between that of EKI and TEKI initial ensemble members. For this test case the truth is shown in Figure \ref{fig:truth_2}. The numerics for this test case show 
a similar  ordering of the accuracy of the methods to that observed
in case 1. However Figure \ref{fig:RE_DM_2} also demonstrates that the 
relative error of EKI with random draws starts to diverge. This is linked
to the overfitting of the data, since in this case the data misfit goes 
below the noise level. The results are similar to those obtained in 
\cite{ILS13} obtained for EKI in discrete from \eqref{eq:update}. 
This over-fitting is demonstrated in Figure \ref{fig:EKI_iter} 
which highlights the difficulty of reconstructing the truth from 
Figure \ref{fig:truth_2} within the linear span of the EKI initial ensemble.

For EKI and TEKI with a KL basis, we see immediately that the divergence of the error does not occur here. Instead the EKI algorithm performs relatively well, similarly to TEKI. However the added regularization again leads to smaller errors
in TEKI than in EKI. Interestingly we also notice that there is little difference in TEKI for both the random draws and the KL basis. These results can be seen in Figures \ref{fig:TEKI_iter}--\ref{fig:TEKI_iter2}. It is worth mentioning 
that, although Figure \ref{fig:RE_DM_2} shows that for TEKI with random draws
the misfit reaches the noise level, running for further iterations does
not result in over-fitting (misfit falling below the noise level).

\newpage

\begin{figure}[h!]
\centering
\includegraphics[width=\linewidth]{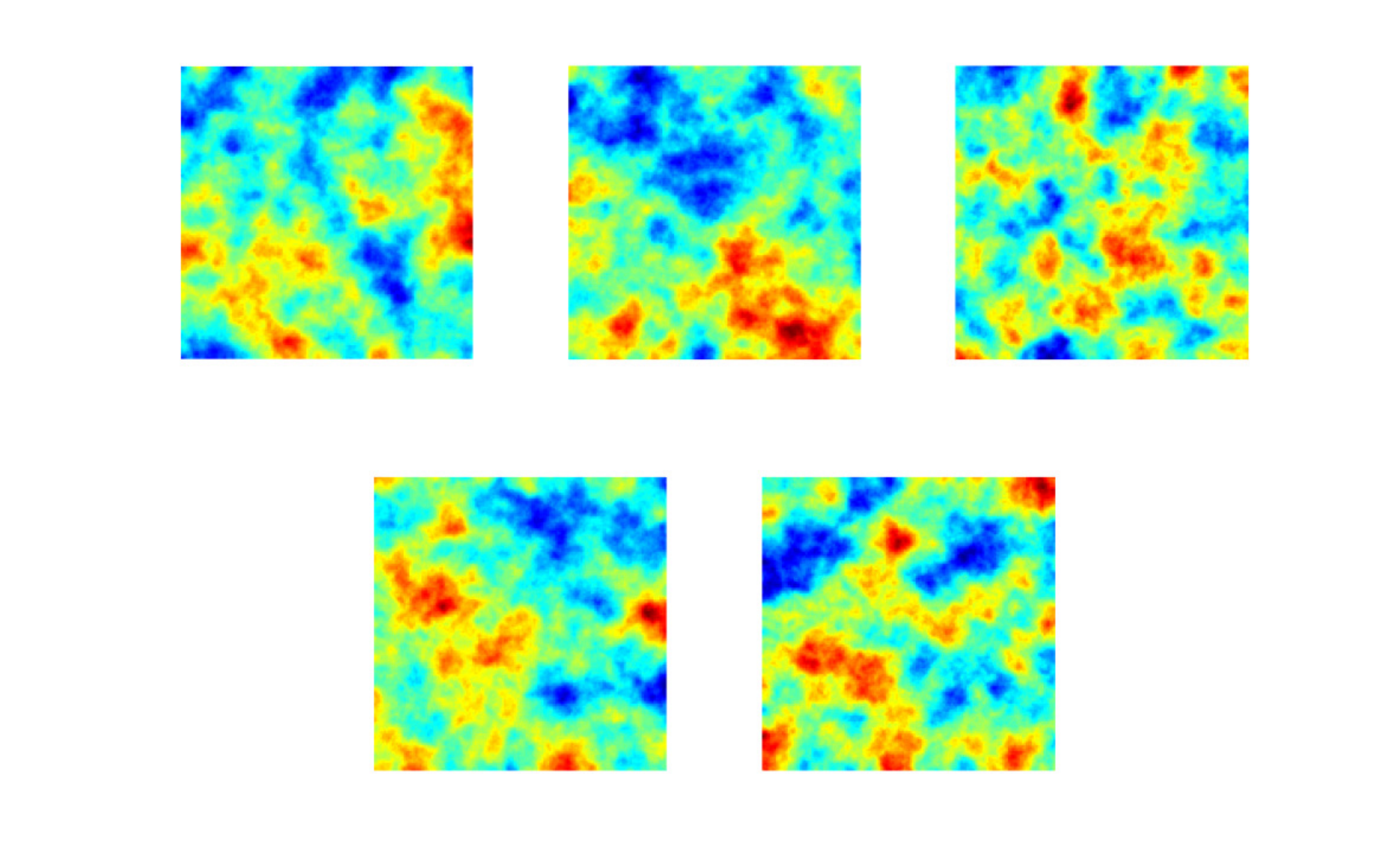}
\caption{Case 1. Progression of EKI through iteration count with prior random draws.}
 \label{fig:EKI_iter_low}
\end{figure}

\begin{figure}[h!]
\centering
\includegraphics[width=\linewidth]{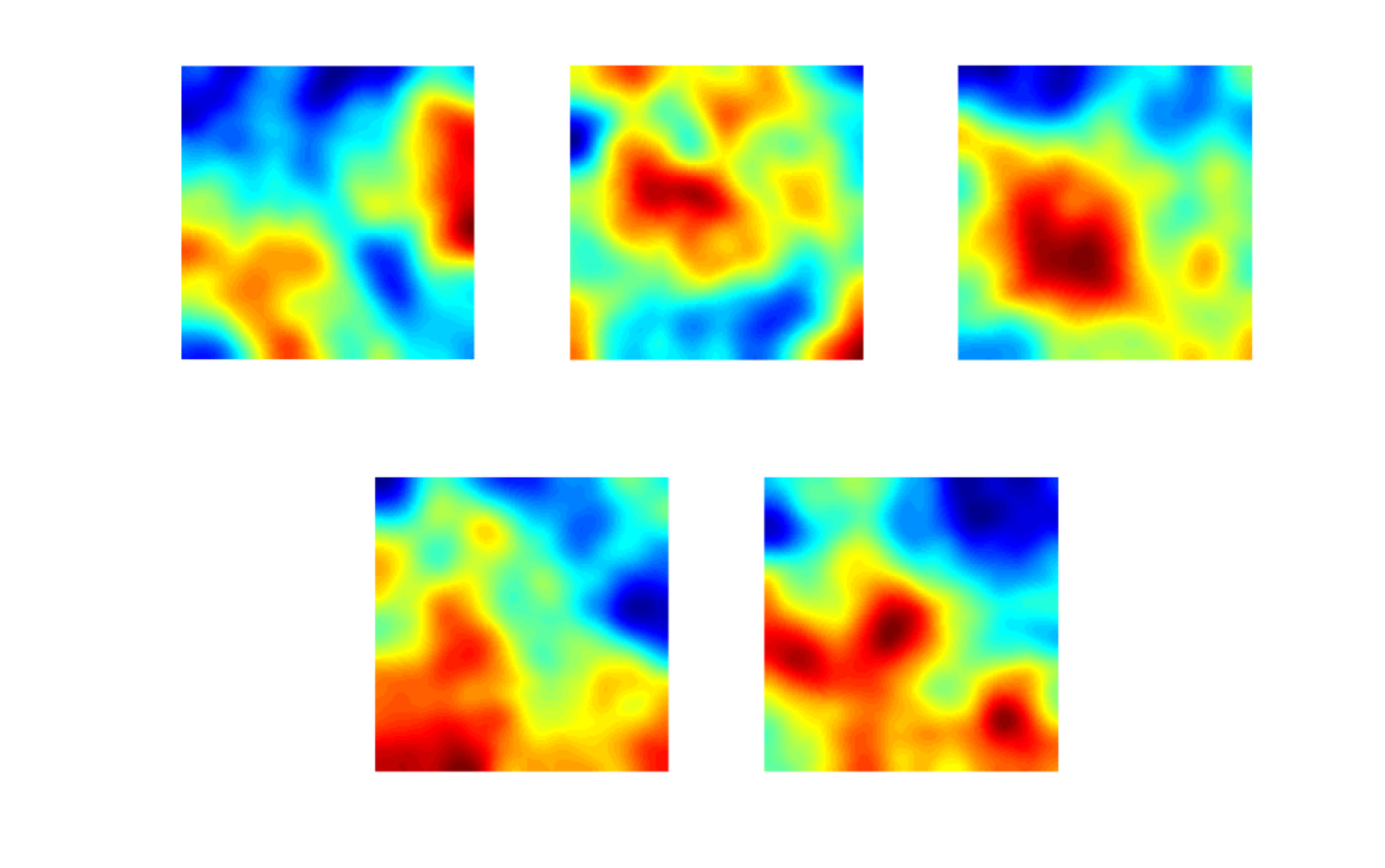}
\caption{Case 1. Progression of TEKI through iteration count with prior random draws.}
 \label{fig:TEKI_iter_low}
\end{figure}

\newpage

\begin{figure}[h!]
\centering
\includegraphics[width=\linewidth]{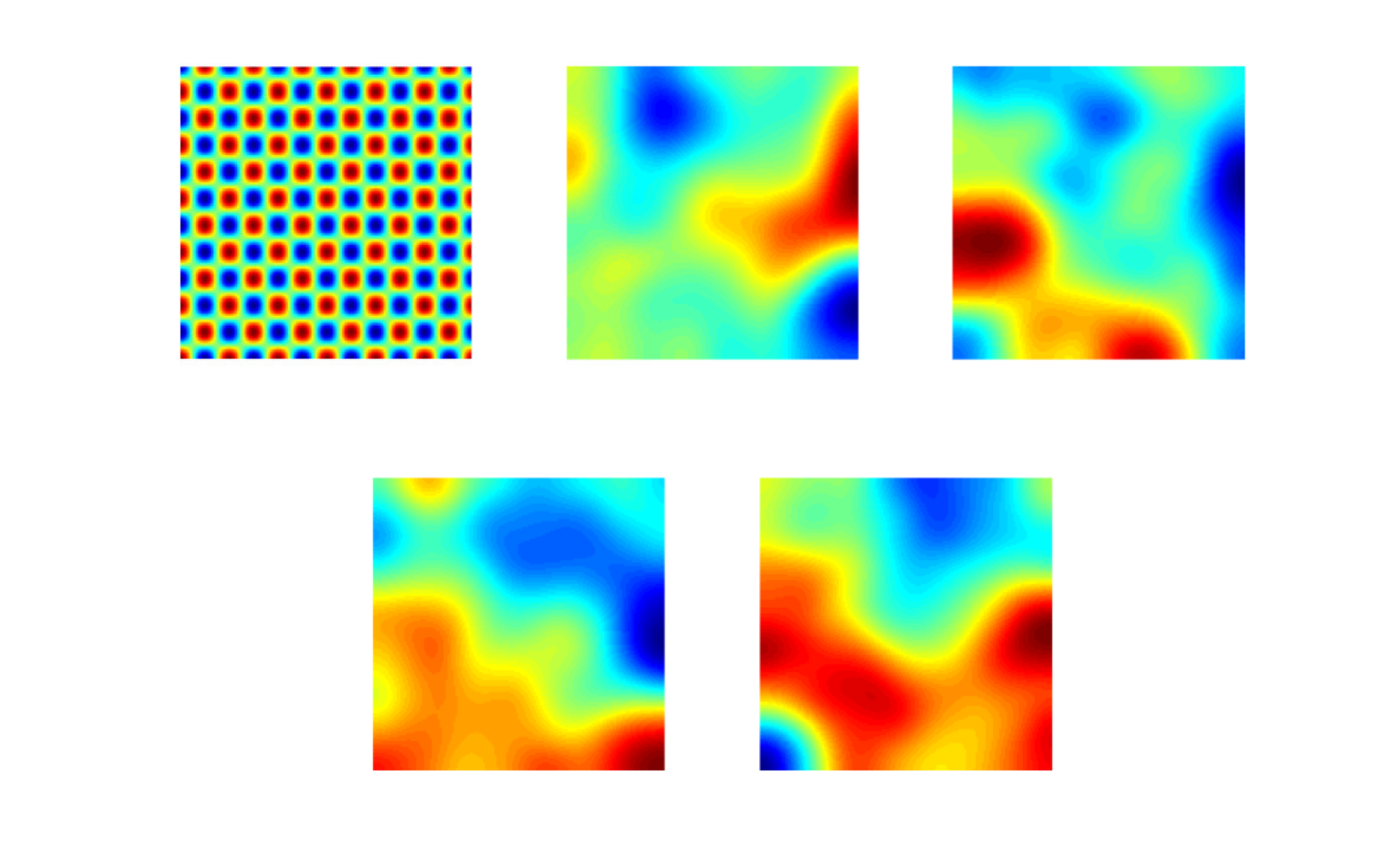}
\caption{Case 1. Progression of EKI through iteration count with KL basis.}
 \label{fig:EKI_iter_low2}
\end{figure}

\begin{figure}[h!]
\centering
\includegraphics[width=\linewidth]{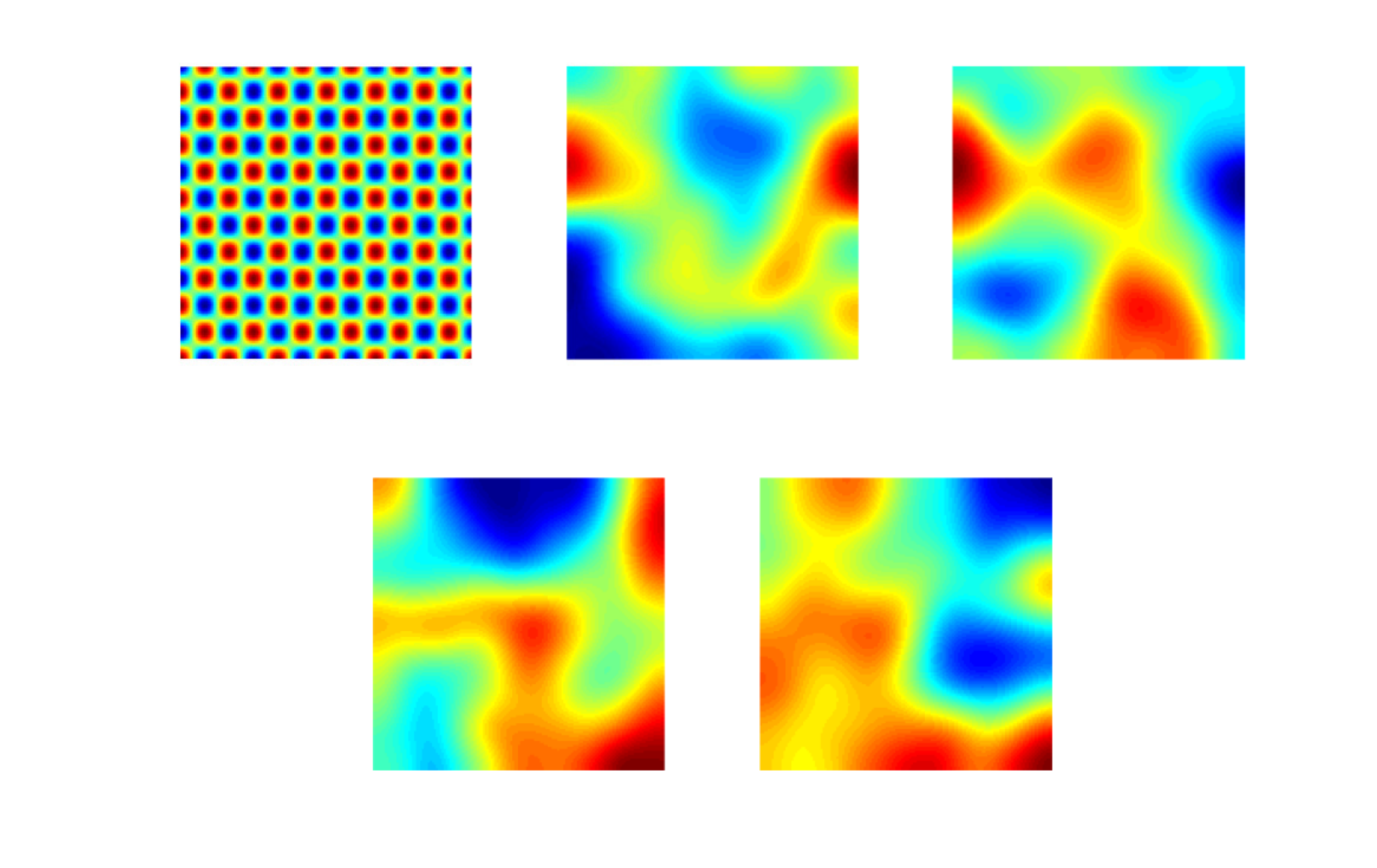}
\caption{Case 1. Progression of TEKI through iteration count with KL basis.}
 \label{fig:TEKI_iter_low2}
\end{figure}

\begin{figure}[h!]
\centering
\includegraphics[scale=0.25]{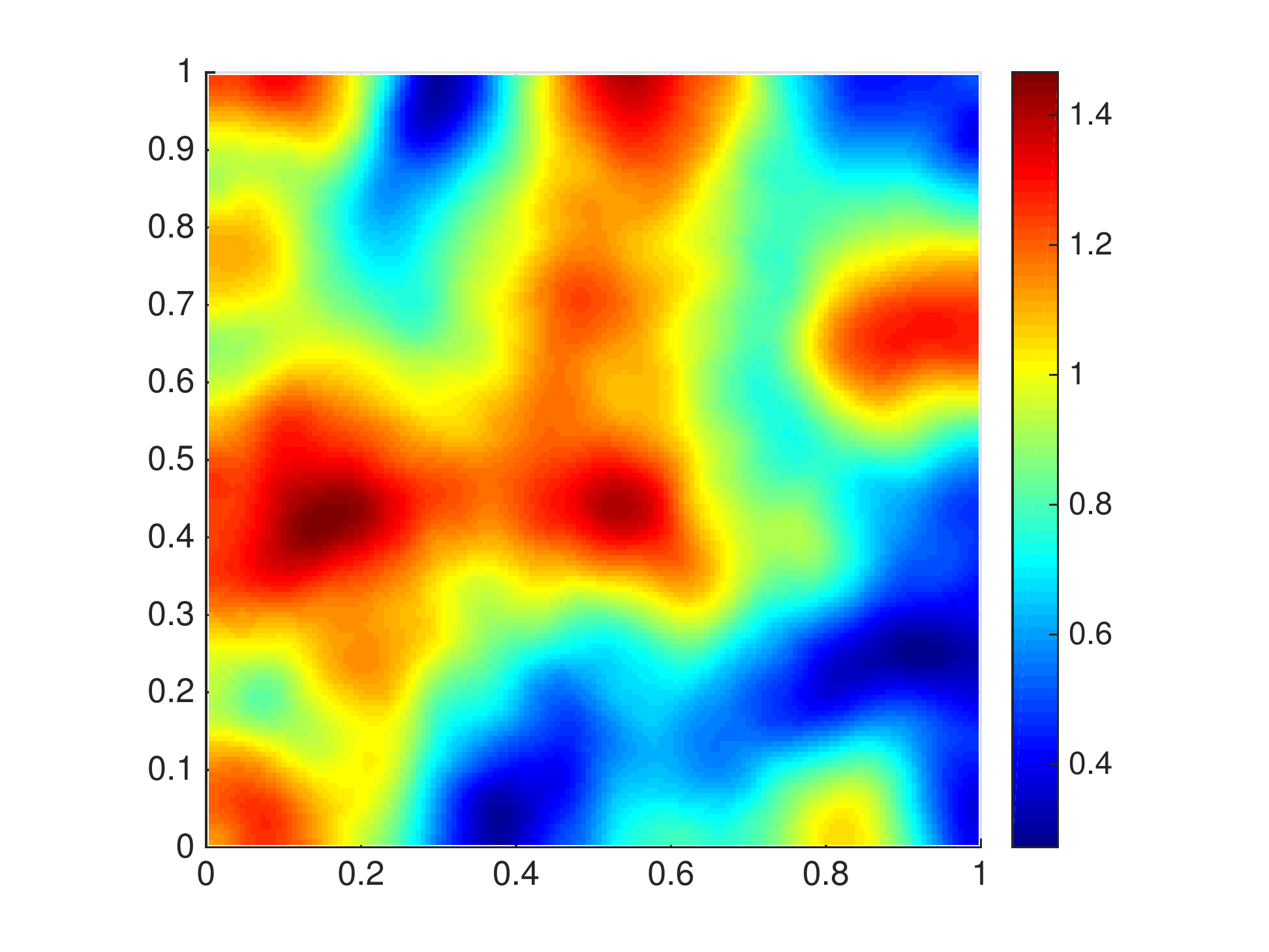}
\caption{Case 2. Gaussian random field truth.}
 \label{fig:truth_2}
\end{figure}

\begin{figure}[h!]
\centering
\includegraphics[width=\linewidth]{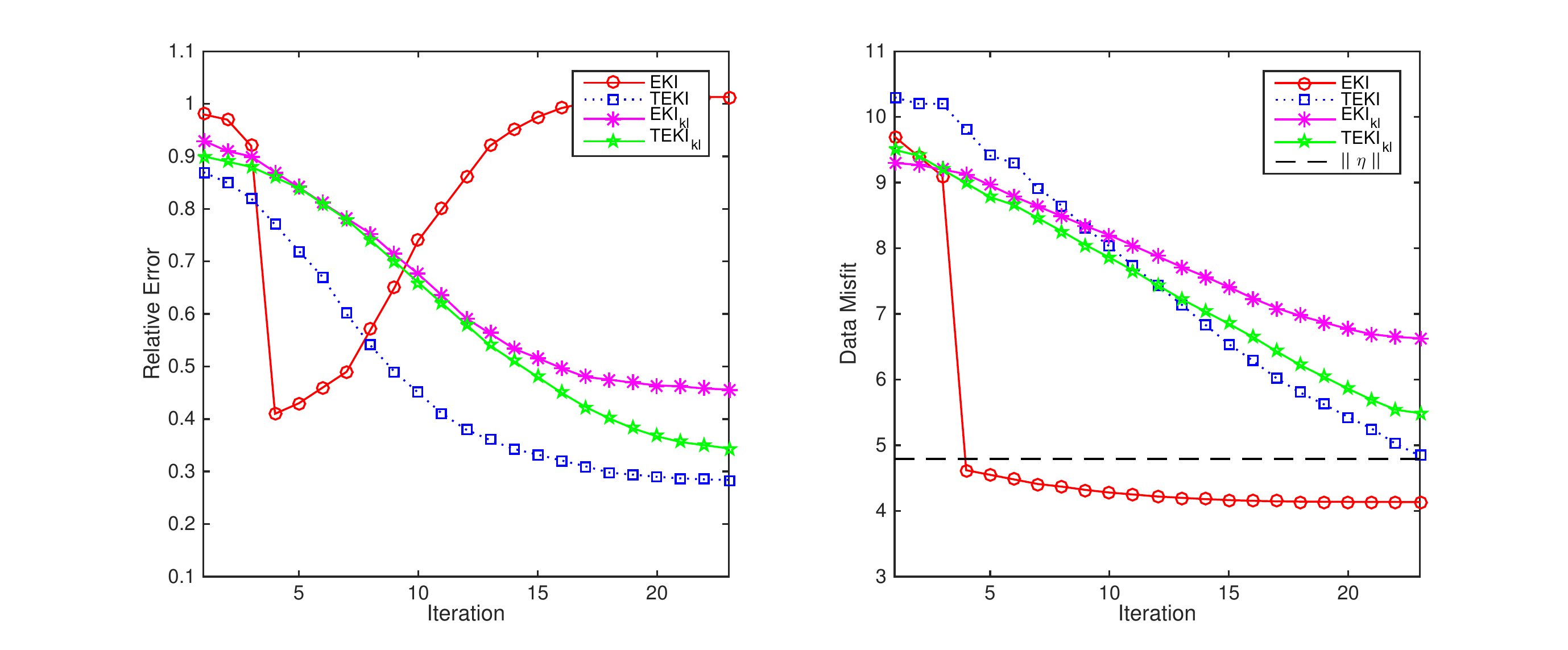}
\caption{Case 2. Relative errors and data misfits of each experiment.}
 \label{fig:RE_DM_2}
\end{figure}

\newpage

\begin{figure}[h!]
\centering
\includegraphics[width=\linewidth]{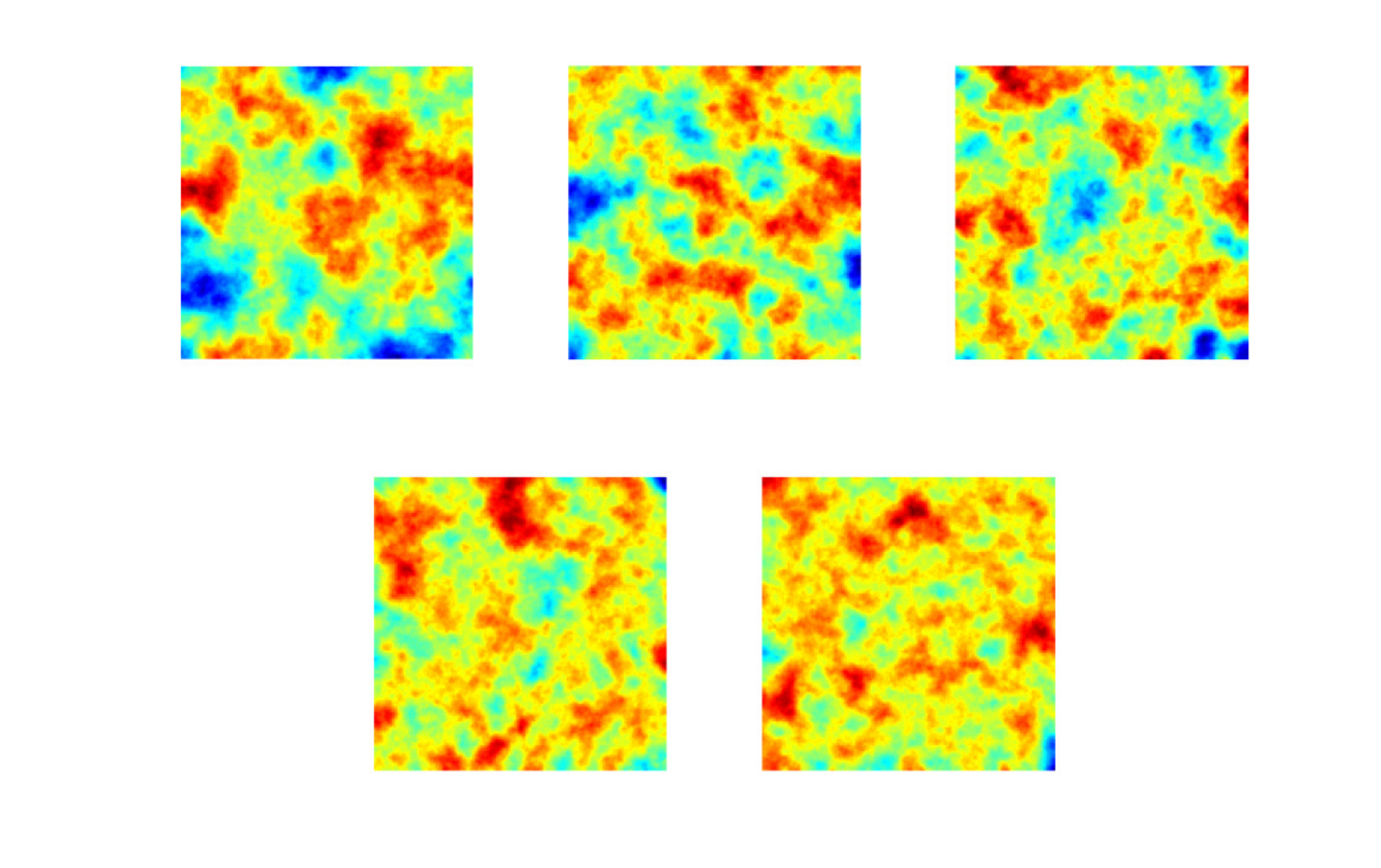}
\caption{Case 2. Progression of EKI through iteration count with prior random draws.}
 \label{fig:EKI_iter}
\end{figure}

\begin{figure}[h!]
\centering
\includegraphics[width=\linewidth]{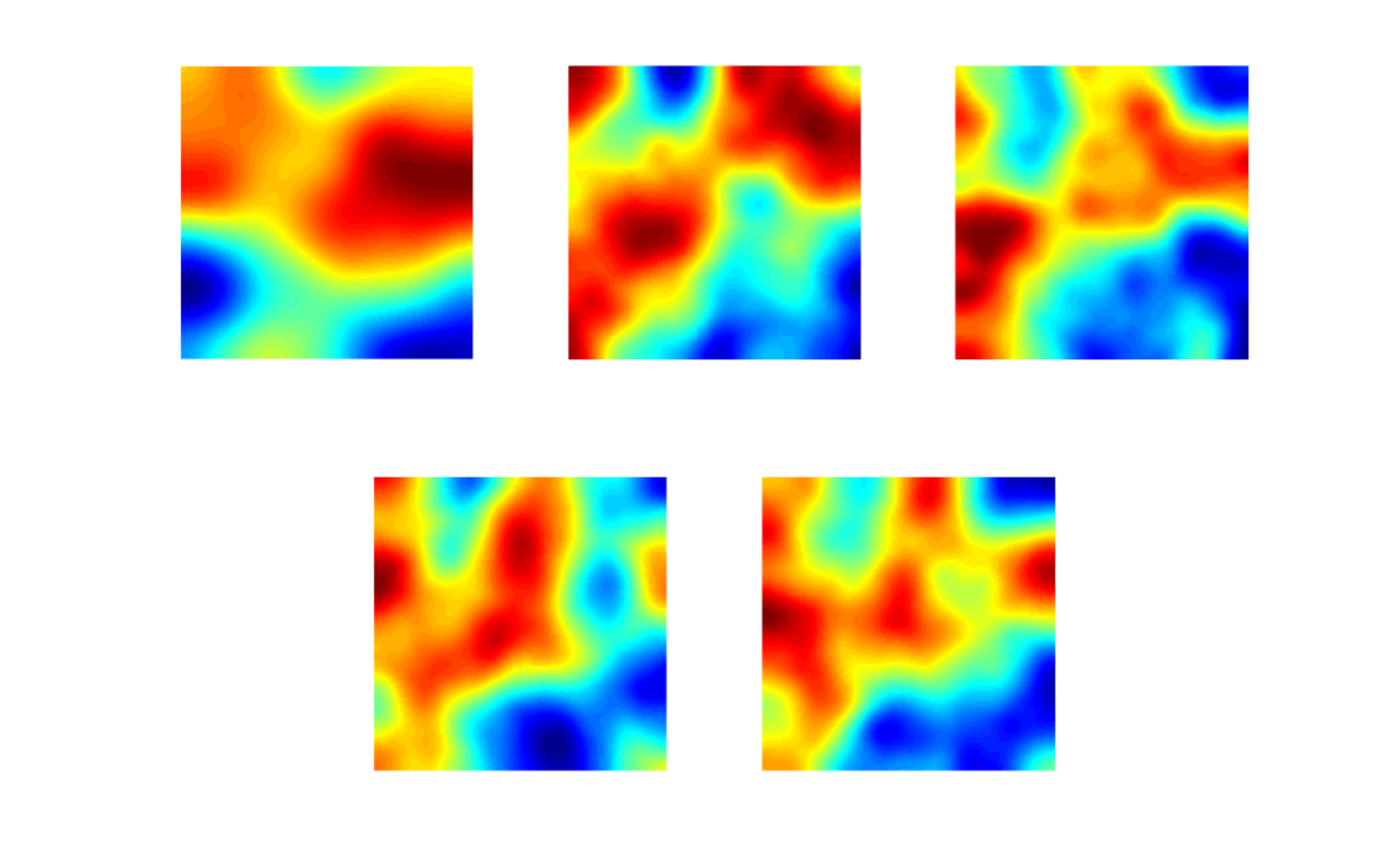}
\caption{Case 2. Progression of TEKI through iteration count with prior random draws.}
 \label{fig:TEKI_iter}
\end{figure}

\newpage

\begin{figure}[h!]
\centering
\includegraphics[width=\linewidth]{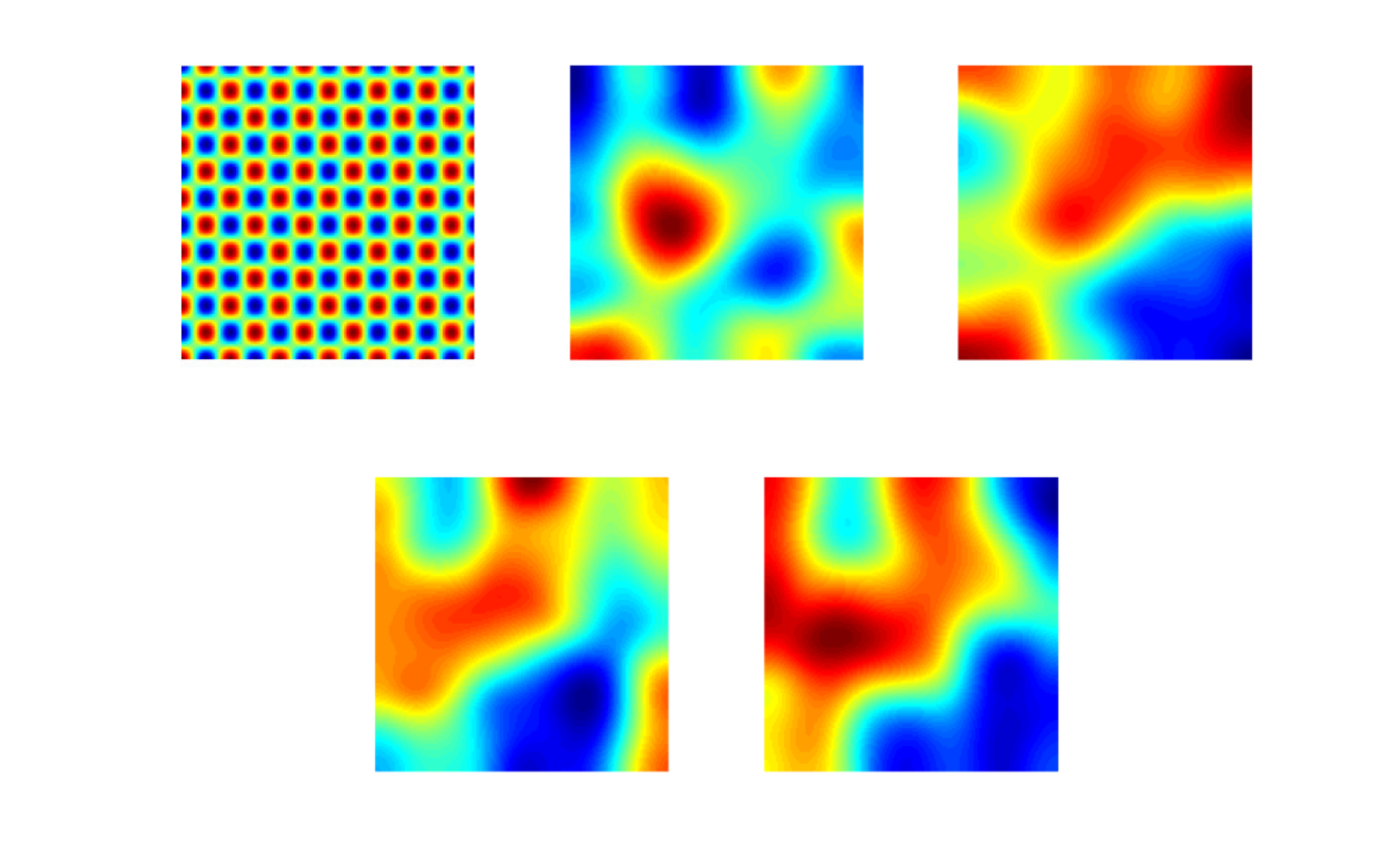}
\caption{Case 2. Progression of EKI through iteration count with KL basis.}
 \label{fig:EKI_iter2}
\end{figure}

\begin{figure}[h!]
\centering
\includegraphics[width=\linewidth]{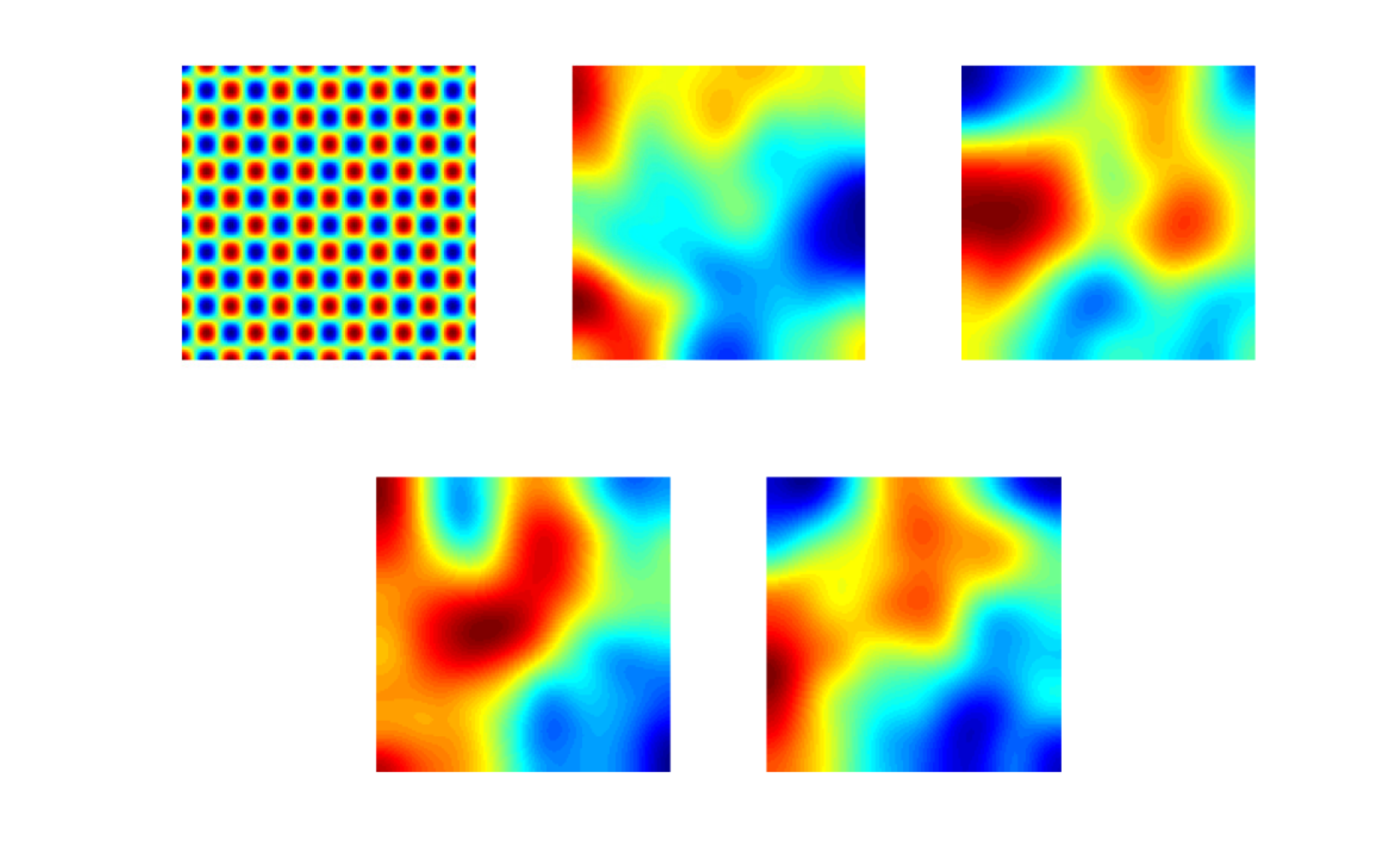}
\caption{Case 2. Progression of TEKI through iteration count with KL basis.}
 \label{fig:TEKI_iter2}
\end{figure}

\newpage

\subsubsection{Case 3.} 
Our third and final test case compares both methods, in a setting in
which the regularity of the random draws for TEKI is the same as for the truth, shown in Figure \ref{fig:truth_3}. Figure \ref{fig:RE_DM_3} demonstrates almost identical outcomes as in Case 2. Figures \ref{fig:EKI_iter_high}--\ref{fig:TEKI_iter_high2} show the progression of the iterations in the four different cases.

As the value of the regularity is higher compared to the previous case, we see the degeneracy of the EKI with random draws. This is highlighted in Figure \ref{fig:RE_DM_3} where we notice the same effect of the overfitting of the data as in Figure \ref{fig:RE_DM_2}. This is similar to Figure \ref{fig:EKI_iter_high} in that an over-fitting
phenomenon leads to a poor fitting of the truth as the iteration progresses.

All other methods, which include TEKI with random draws and both methods 
initialized with the KL basis, perform similarly. This can be accredited to the fact that all of their initial ensembles begin with a high regularity. As we observe from 
Figures \ref{fig:TEKI_iter_high}--\ref{fig:TEKI_iter_high2} the added regularization comes into play with noticeable differences. This can be seen further from Figure \ref{fig:RE_DM_3}.

\begin{figure}[h!]
\centering
\includegraphics[scale=0.25]{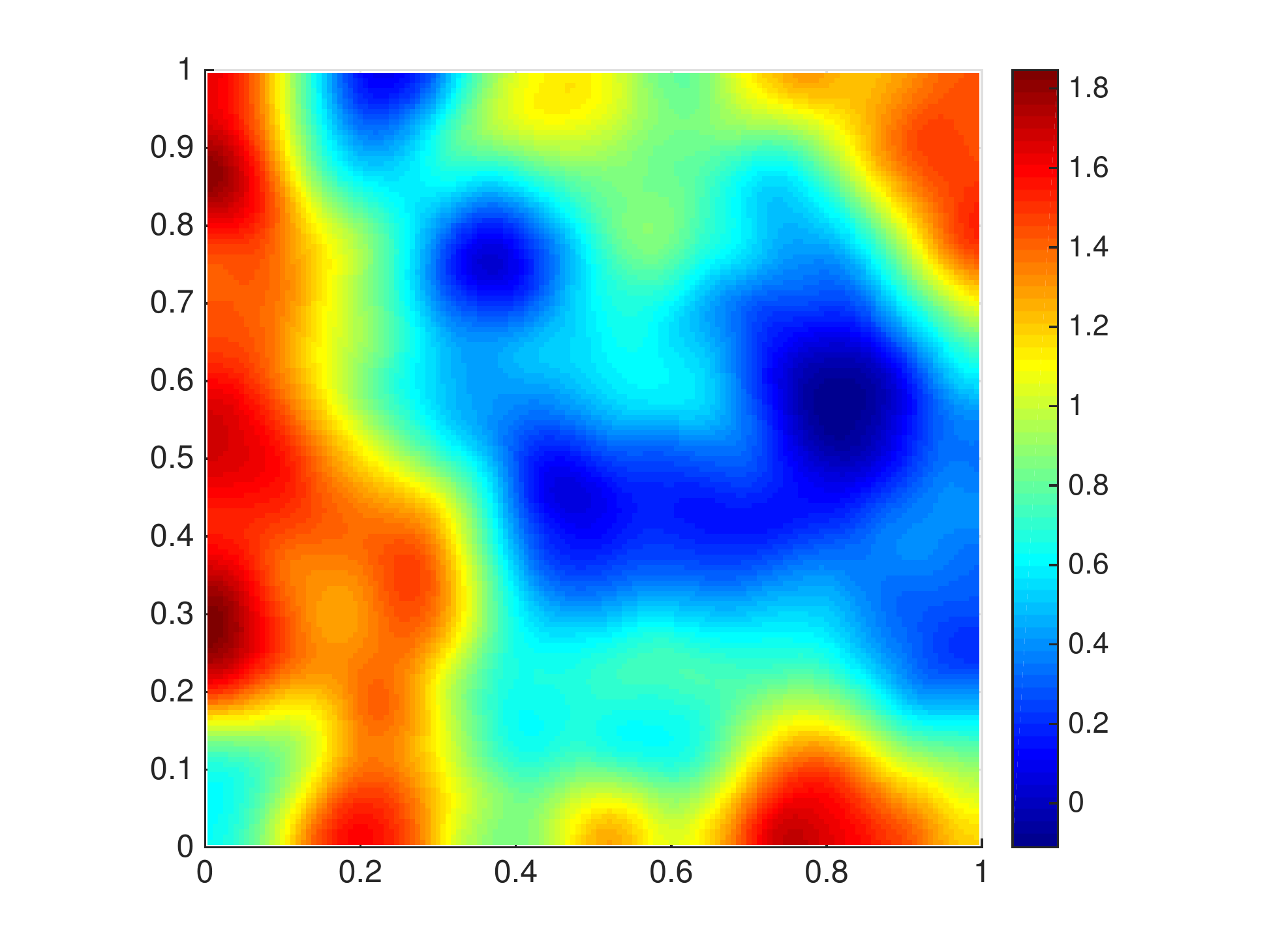}
\caption{Case 3. Gaussian random field truth.}
 \label{fig:truth_3}
\end{figure}

\begin{figure}[h!]
\centering
\includegraphics[width=\linewidth]{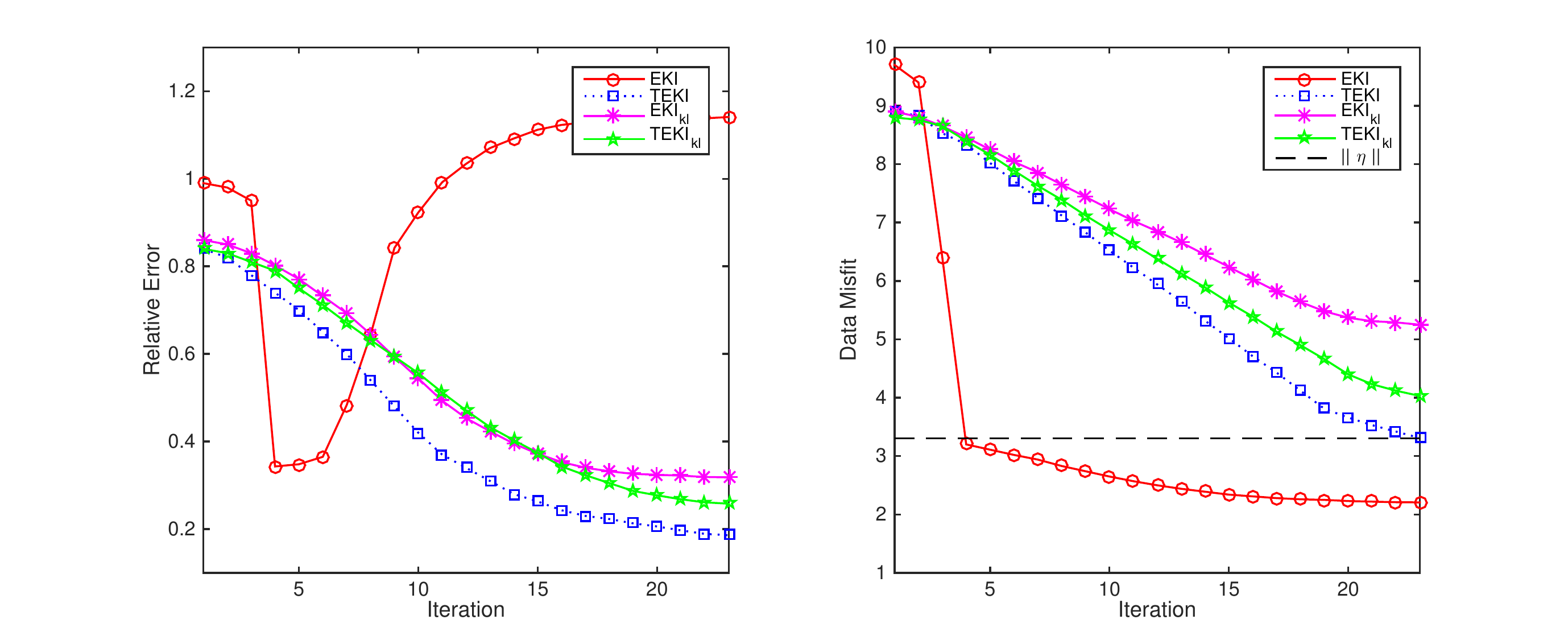}
\caption{Case 3. Relative errors and data misfits of each experiment.}
 \label{fig:RE_DM_3}
\end{figure}

\newpage

\begin{figure}[h!]
\centering
\includegraphics[width=\linewidth]{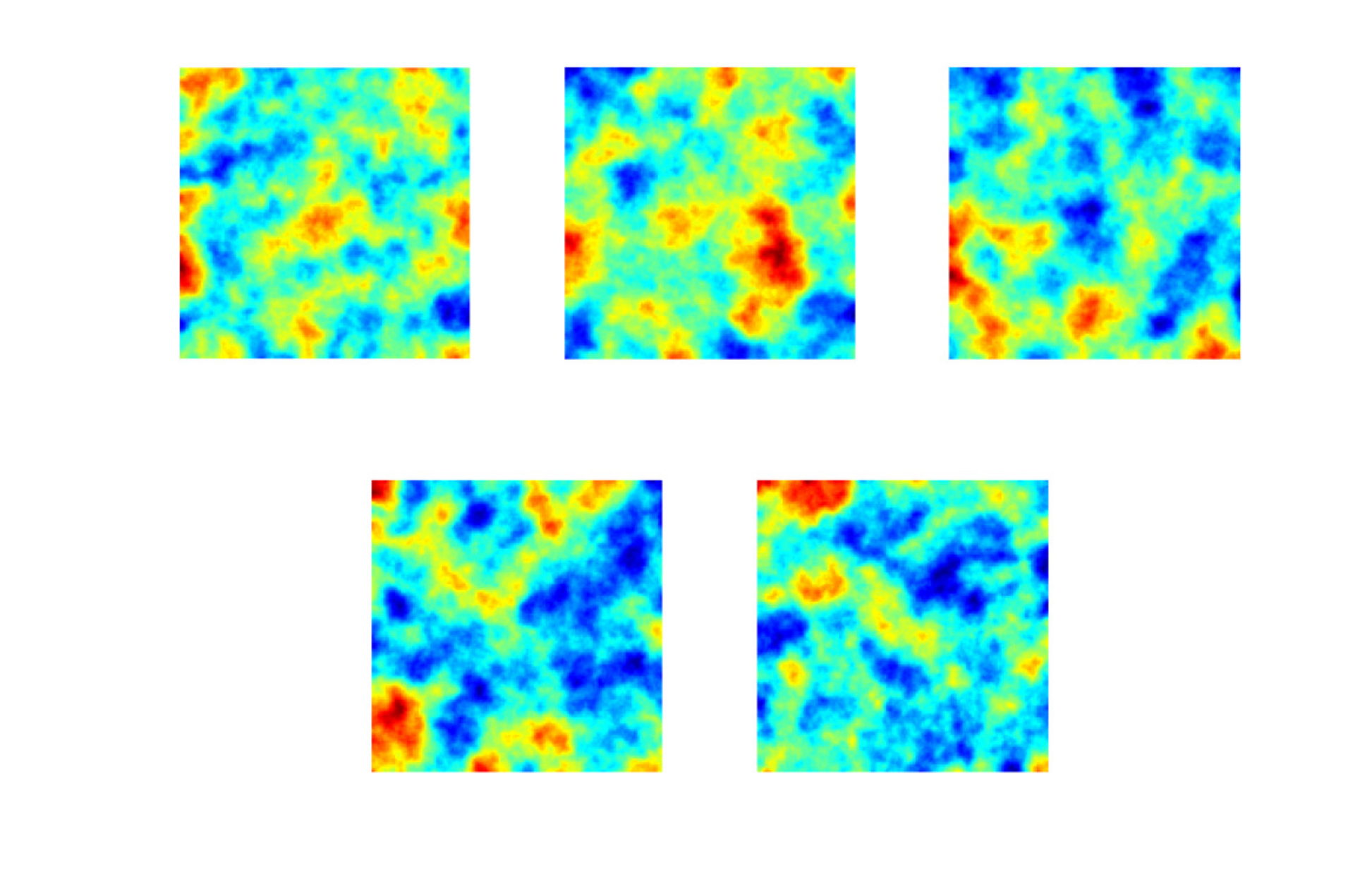}
\caption{Case 3. Progression of EKI through iteration count with random draws.}
 \label{fig:EKI_iter_high}
\end{figure}

\begin{figure}[h!]
\centering
\includegraphics[width=\linewidth]{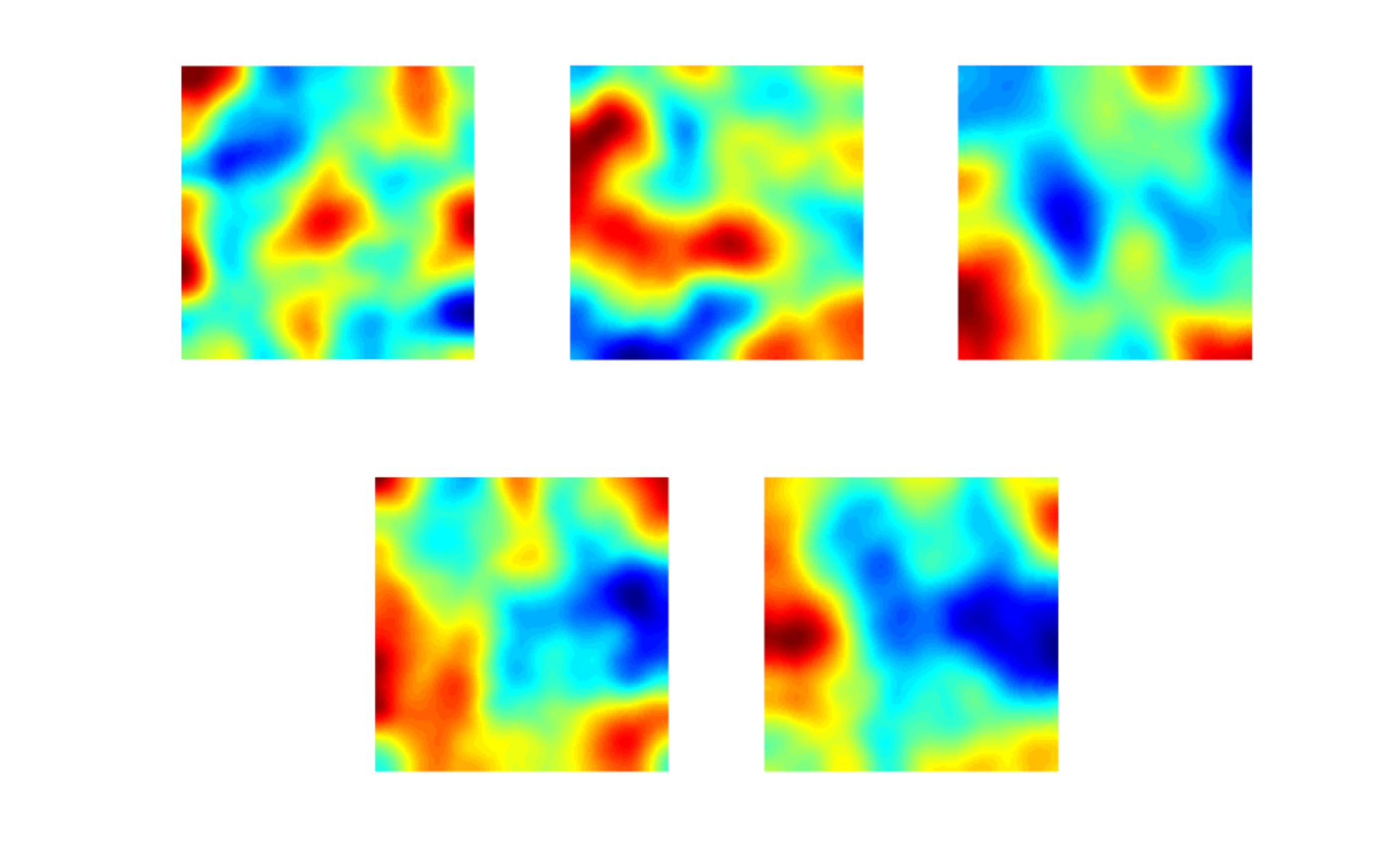}
\caption{Case 3. Progression of TEKI through iteration count with random draws.}
 \label{fig:TEKI_iter_high}
\end{figure}

\newpage

\begin{figure}[h!]
\centering
\includegraphics[width=\linewidth]{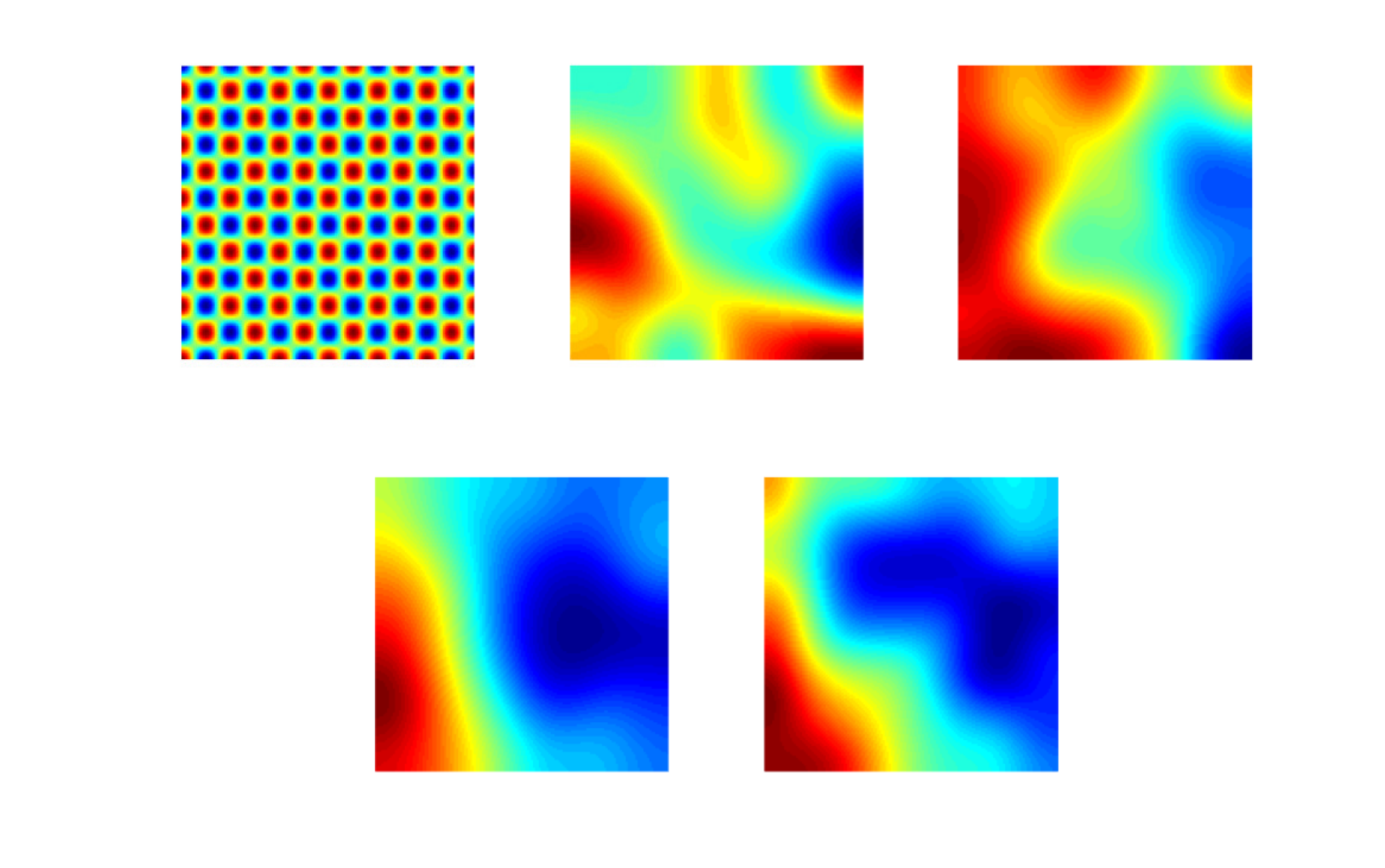}
\caption{Case 3. Progression of EKI through iteration count with KL basis.}
 \label{fig:EKI_iter_high2}
\end{figure}

\begin{figure}[h!]
\centering
\includegraphics[width=\linewidth]{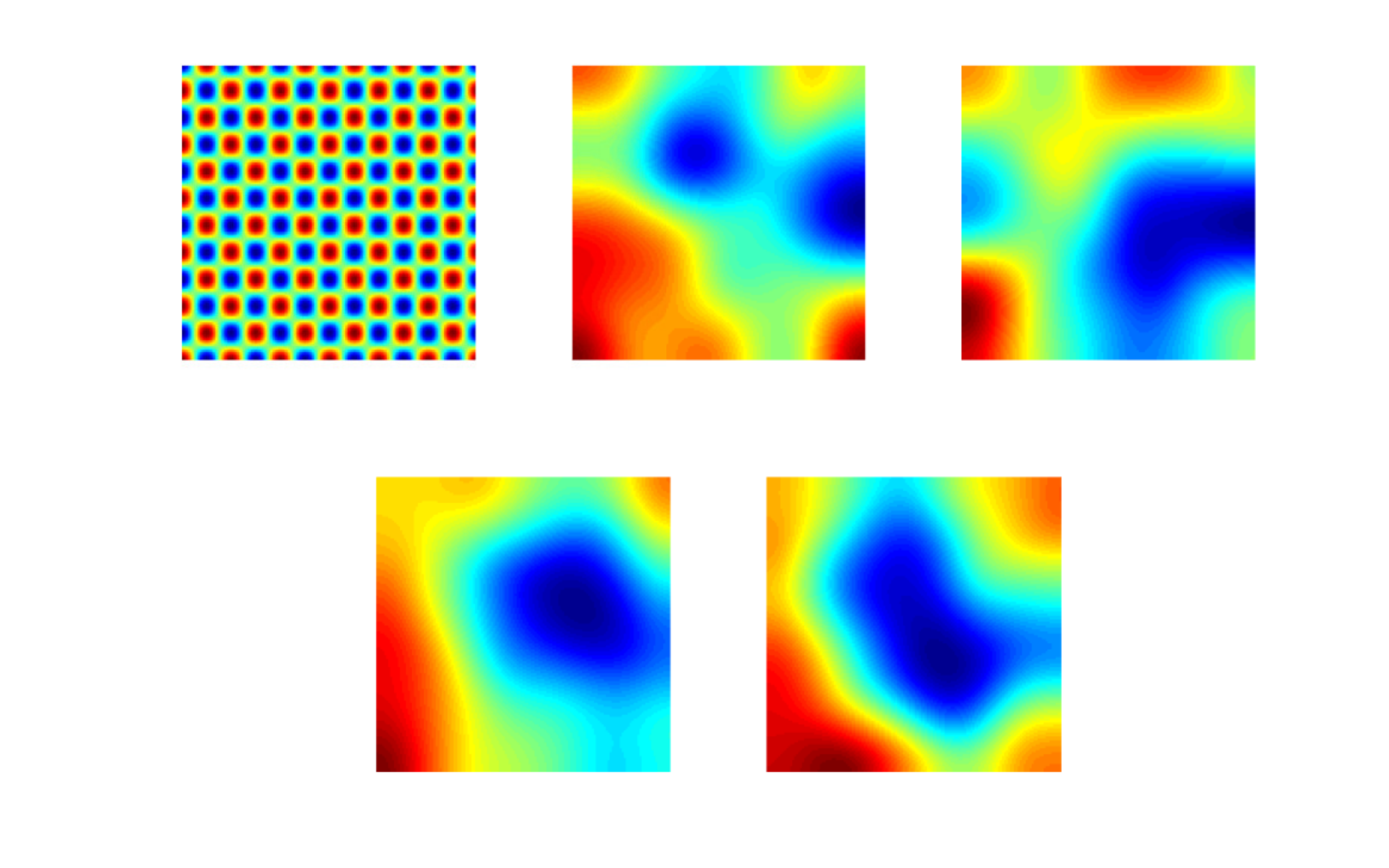}
\caption{Case 3. Progression of TEKI through iteration count with KL basis.}
 \label{fig:TEKI_iter_high2}
\end{figure}

\newpage

%Consider the following set
%\[
%\mathcal{B}:=\{(v^{(1)},\cdots, v^{(J)}): v^{(j)}\in \calA, \|C_0^{-1/2}(v^{(j)}-u^{(j)})\|_X\leq 1 \}
%\]
%Then for any $\frakv\in \mathcal{B}$, 
%\begin{align}
%\label{tmp:Djk1}
%D_{j,k}(\fraku)-D_{j,k}(\frakv)=&\langle \Gamma^{-1}(G(u^{(j)})-y), G(u^{(k)})-\bar{G}\rangle_Y(u^{(k)}-v^{(k)}-\bar{u}+\bar{v})\\
%\label{tmp:Djk2}
%&+\langle \Gamma^{-1}(G(u^{(j)})-y), G(u^{(k)})-G(v^{(k)})\rangle_Y(v^{(k)}-\bar{v})\\
%\label{tmp:Djk3}
%&+\langle \Gamma^{-1}(G(u^{(j)})-G(v^{(j)})), G(v^{(k)})\rangle_Y(v^{(k)}-\bar{v}).
%\end{align}
%\end{proof}
%Denote 
%\[
%M_1(\fraku):=\max_{j,k}\langle \Gamma^{-1}(G(u^{(j)})-y), G(u^{(k)})-\bar{G}\rangle_Y
%\]
%\[
%L(\fraku)=\sup_{\frakv\in \mathcal{B}} 
%\]
%Assume that 
%\[
%\|\partial G (u)\|:=G(u)-G(v)
%\] 

\newpage

\section{Conclusions}
\label{sec:CON}

{Regularization is a central idea in optimization and statistical inference problems. In this work we considered adapting EKI 
methods to allow for Tikhonov regularization, leading to the TEKI
methodology. Inclusion of this Tikhonov regularizer within EKI leads
demonstrably to improved reconstructions of our unknown; we have shown this
on an inverse eikonal equation and porous medium equation,
using both random draws from the prior and 
from the KL basis to initialize the ensemble methods. We also derived 
a continuous time limit of TEKI and studied its properties, including
showing the existence of the TEKI flow and its long-time behaviour. 
In particular we showed that the TEKI flow always reaches consensus --
ensemble members collapse on one another.
There are a several potentially fruitful new directions one can consider which stem
from this work; we outline a number of them.}

\begin{itemize}

\item {The inclusion of regularization in this paper was specific to the case of the Cameron-Martin space and hence Tikhonov-like Sobolev
regularization. It would be of interest to generalize to the regularizers of 
other forms such as $L_1$ and total variation penalties \cite{BB18,EHNR96}.}

\item {Understanding the EnKF as an optimizer is important, specifically in terms of how effective it is in comparison with other derivative-free optimization methods. 
Using the analysis tools being developed in the work in progress \cite{CST18}
could be helpful in this context.} 

\item {
It would of interest to see how the techniques discussed in \cite{CIRS18},
where hierarchical EKI is introduced, could be improved by use of TEKI. The analysis presented here could be extended to the hierarchical setting. }

\item {Related to hierarchical techniques discussed, one could treat the regularization parameter ${\boldsymbol{\lambda}}$ as a further unknown in our inverse problem. As this can be seen as a scaling factor in the 
covariance, it could be treated as an amplitude factor, in the usual way presented through Whittle-Mat\'{e}rn priors \cite{RHL14}, and learned hierarchically as in
\cite{CIRS18}.
Alternatively it might be of interest to study the adaptation of other standard
statistical techniques for estimation of $\boldsymbol{\lambda}$ to this inverse problem
setting \cite{BB18,EHNR96,GW85}.}

\item It is possible to impose convex constraints directly into EKI; see
\cite{ABLSS19}. However non-convex constraints present
difficulties in the framework described in that paper as non-uniqueness
may arise in the optimization problems to be solved at each step of the
algorithm. Nonconvex equality constraints could be imposed by using 
the methods in this paper to impose them in a relaxed form. 
A constraint set defined by the equation $W(u)=0$ could be approximately
imposed by appending \eqref{eq:inv_re1}-\eqref{eq:inv_re2} with the equation $W(u)+\eta_3=0$
and choosing $\eta_3$ to be a Gaussian with small variance. 

\end{itemize}

\section*{Acknowledgments}

NKC acknowledges a Singapore Ministry of Education Academic Research Funds Tier 2 grant [MOE2016-T2-2-135]. The work of AMS was funded by
US ONR grant N00014-17-1-2079 and 
the US AFOSR grant  FA9550-17-1-0185.  The research of XTT is supported by the National University of Singapore grant R-146-000-226-133.
The authors are grateful to Vanessa Styles (University of Sussex)
for providing a solver for the eikonal equation, and guidance on its use.

\bibliographystyle{plain} %alpha
\bibliography{elliptic.bib}

\section{appendix}
\label{sec:app}
\subsection{Darcy flow}
The purpose of this Appendix is to exhibit numerical results demonstrating
that what we showed for the eikonal equation in section  \ref{sec:NUM}
is not specific to that particular inverse problem. 
In order to do this we use exactly the
same experimental set-up as in section \ref{sec:NUM}, simply 
replacing the eikonal equation by Darcy flow in a porous medium and
defining a relevant inverse problem.
Thus to explain the numerical experiments which follow in this section it 
suffices to simply define the forward problem and the observation operator.

Given a domain $D = [0,1]^2$ and  real-valued permeability function $\kappa$ defined on $D$,  the forward model 
is concerned with determining a real-valued pressure (or hydraulic head) function
$p$ on $D$ from
\newpage
\begin{equation}
\label{eq:darcy}
-\nabla\cdot({\kappa}\nabla p) = f, \quad x \in  D, \\
\end{equation}
with mixed boundary conditions
\begin{equation*}
p(x_1,0) = 100, \ \ \ \  \frac{\partial p}{\partial x_1}(1,x_2) = 0, \ \ -\kappa\frac{\partial p}{\partial x_1}(0,x_2) = 500, \ \ \  \frac{\partial p}{\partial x_2} (x_1,1) = 0.
\end{equation*}
Throughout we simply use the source $f \equiv 1$. The inverse problem is concerned with the recovery of $u=\log(\kappa)$ from mollified pointwise linear functionals of the form $G_{j}(u)=l_j(u)$ with $l_j$ denoting mollified pointwise observation 
at $x_j.$ The results that follow have no commentary because the phenomena
exhibited are identical to what we see for the eikonal equation.

\section*{Case 1.}

\begin{figure}[h!]
\centering
\includegraphics[scale=0.25]{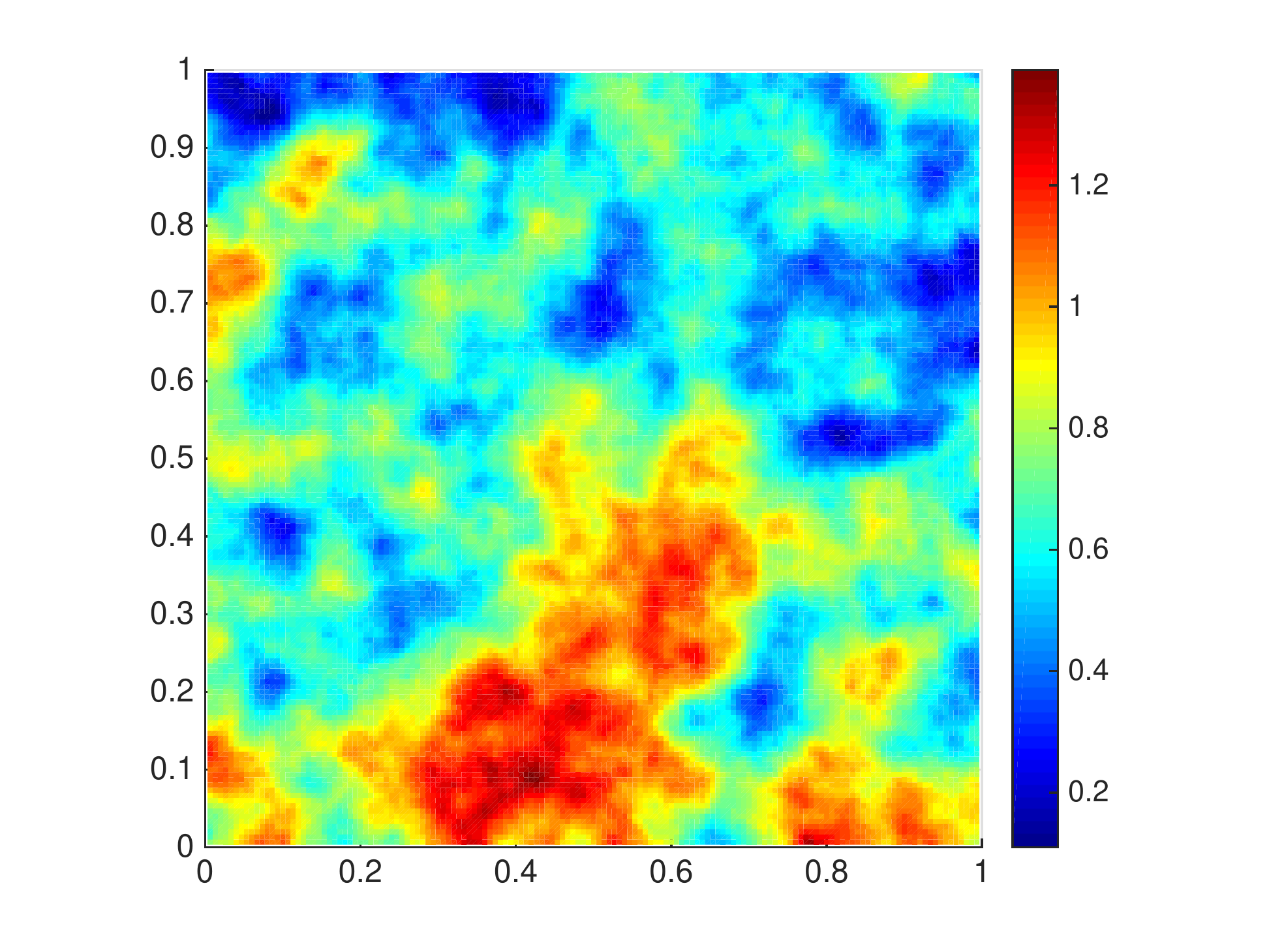}
\caption{Case 1. Gaussian random field truth.}
 \label{fig:truth_1_df}
\end{figure}

\begin{figure}[h!]
\centering
\includegraphics[width=\linewidth]{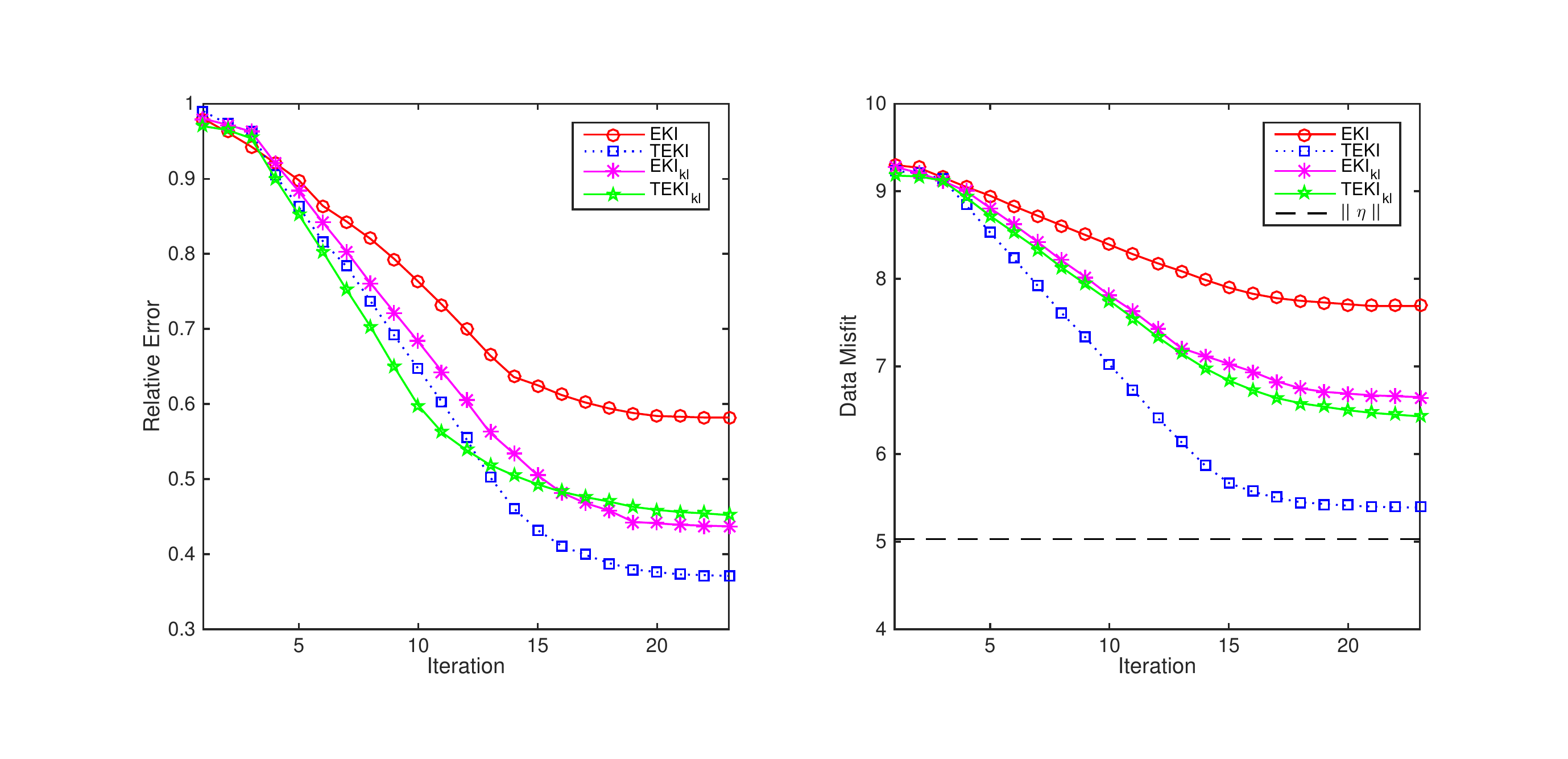}
\caption{Case 1. Relative errors and data misfits of each experiment.}
 \label{fig:RE_DM_1_df}
\end{figure}

\newpage

\begin{figure}[h!]
\centering
\includegraphics[width=\linewidth]{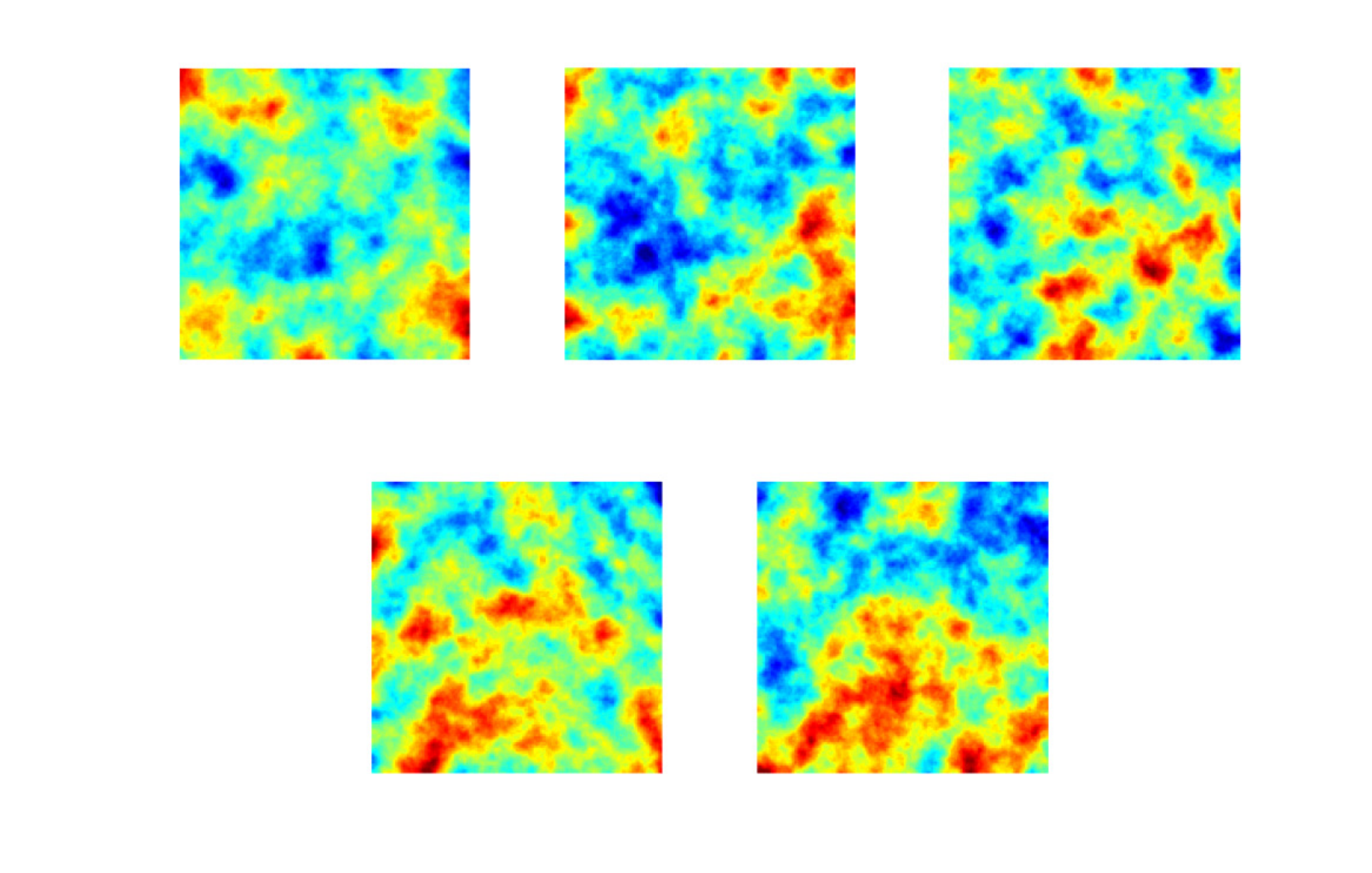}
\caption{Case 1. Progression of EKI through iteration count with prior random draws.}
 \label{fig:1EKI_RD_df}
\end{figure}

\begin{figure}[h!]
\centering
\includegraphics[width=\linewidth]{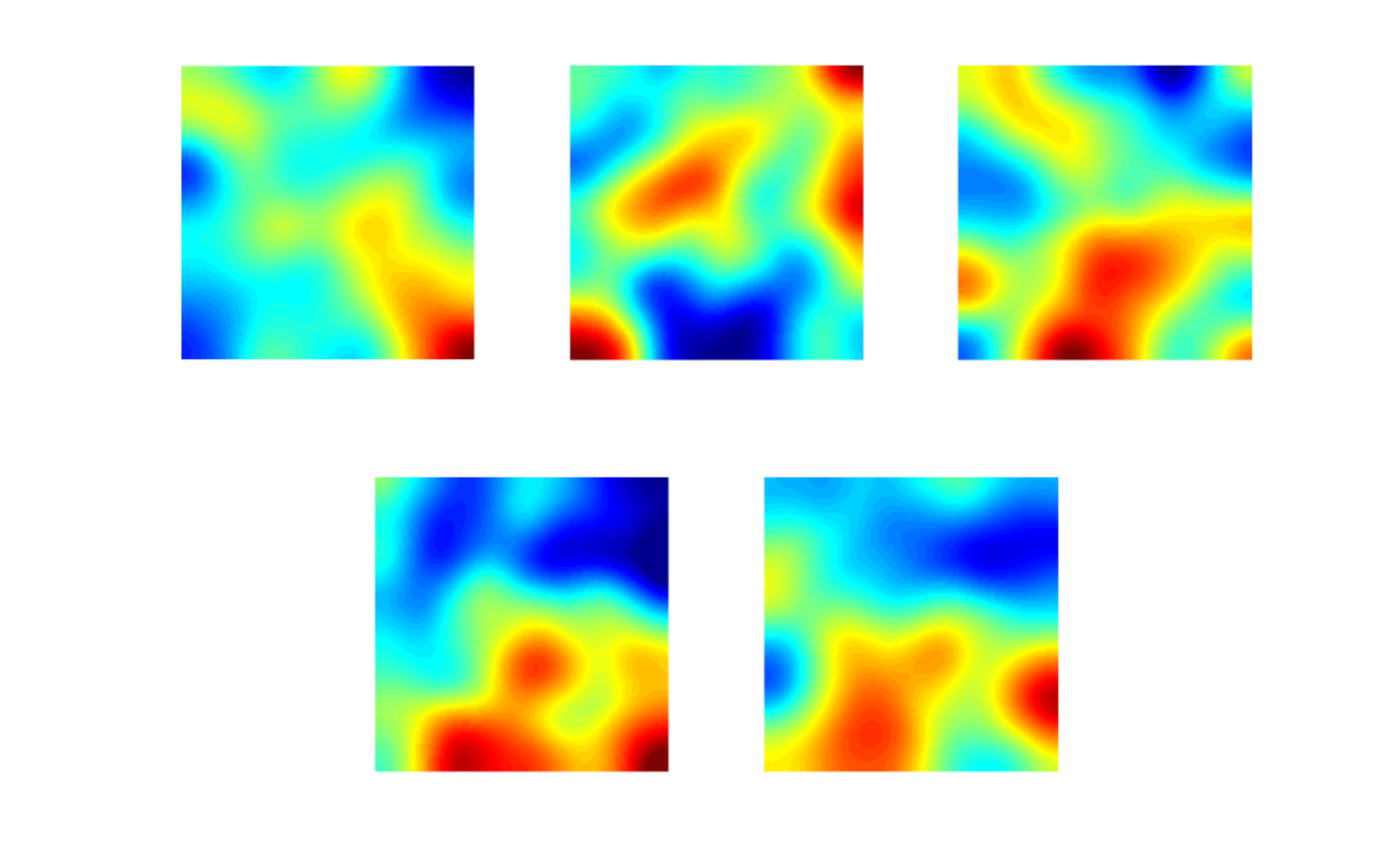}
\caption{Case 1. Progression of TEKI through iteration count with prior random draws.}
 \label{fig:1TEKI_RD_df}
\end{figure}

\newpage

\begin{figure}[h!]
\centering
\includegraphics[width=\linewidth]{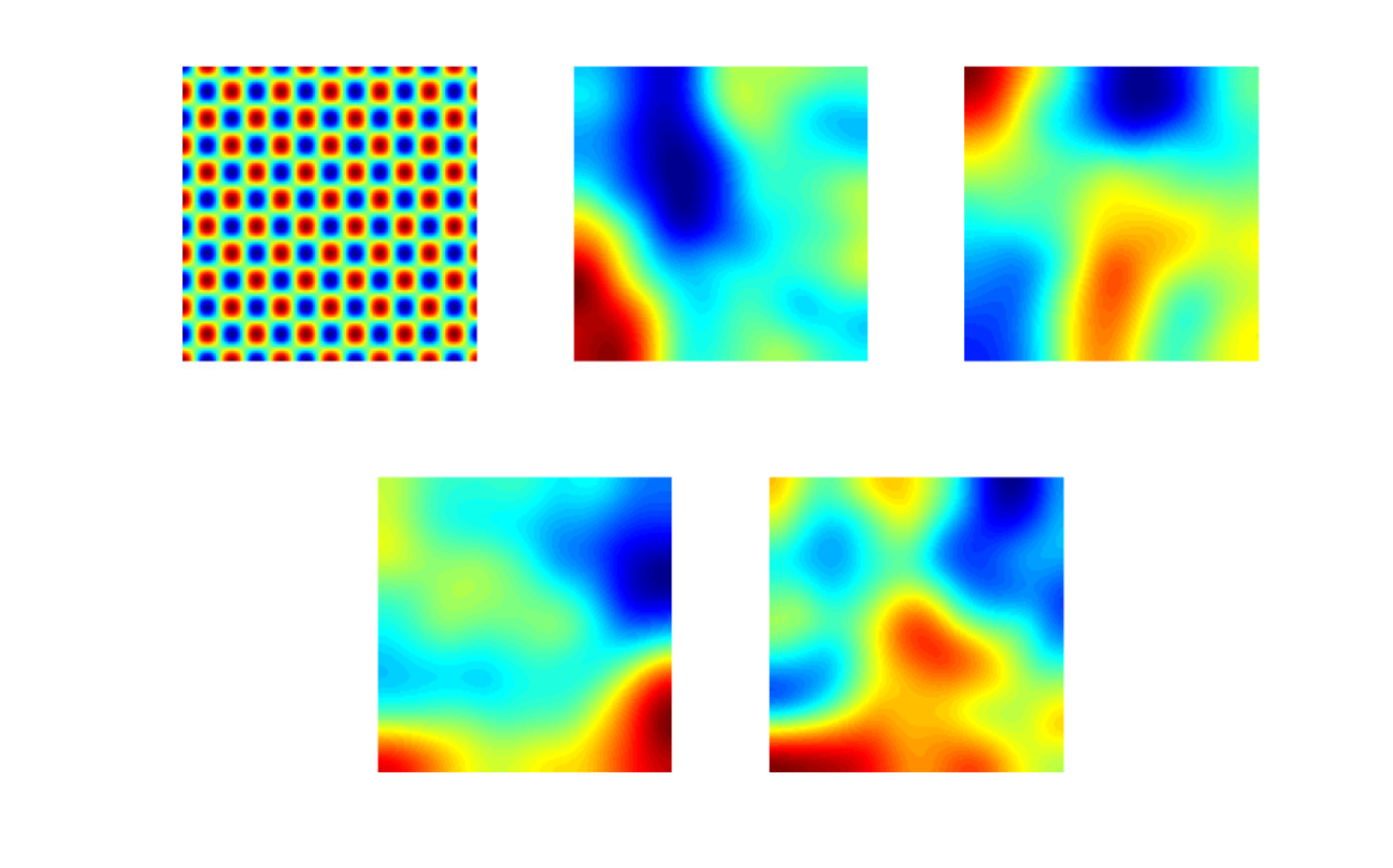}
\caption{Case 1. Progression of EKI through iteration count with prior from KL basis.}
 \label{fig:1EKI_KL_df}
\end{figure}

\begin{figure}[h!]
\centering
\includegraphics[width=\linewidth]{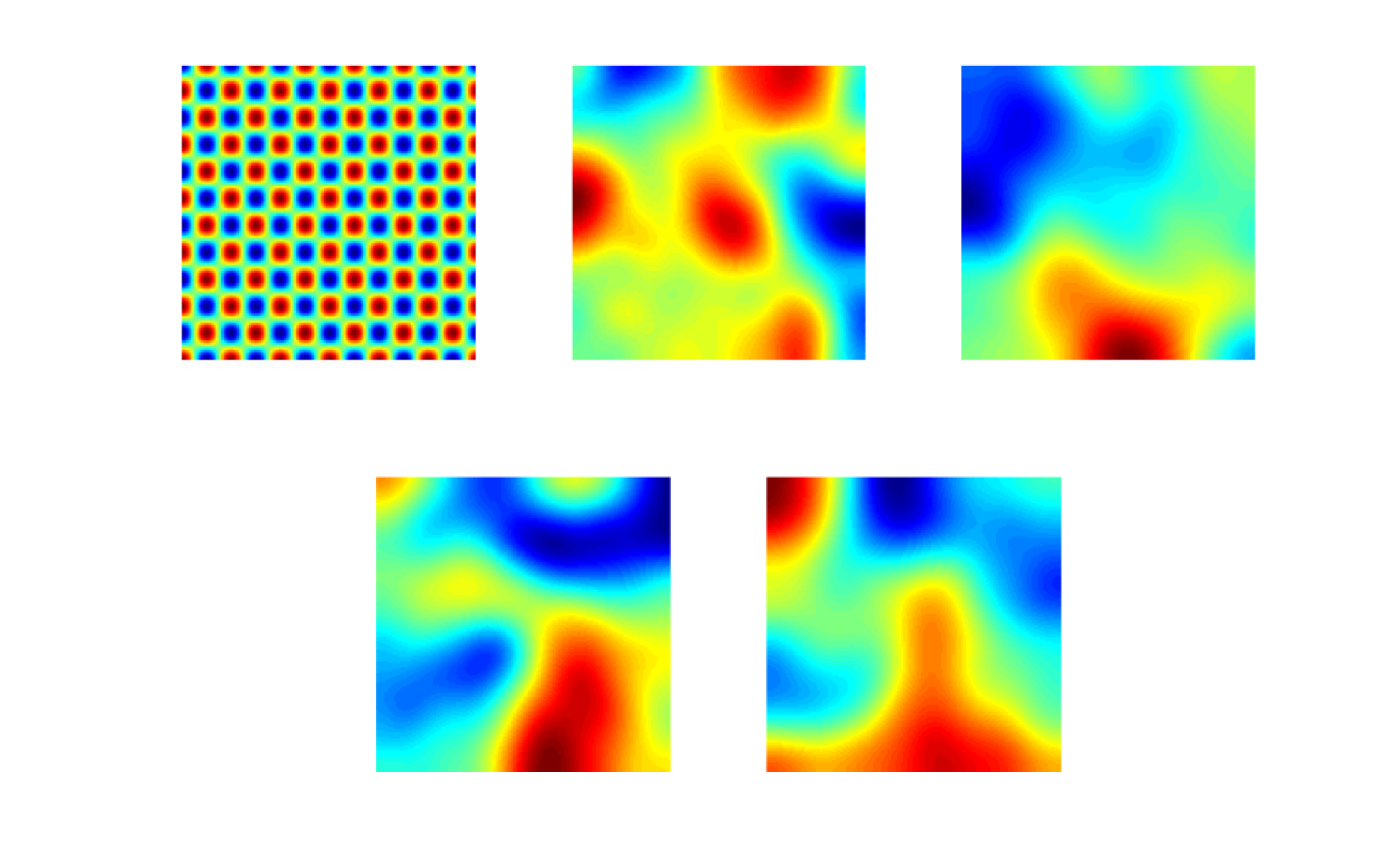}
\caption{Case 1. Progression of TEKI through iteration count with prior from KL basis.}
 \label{fig:1TEKI_KL_df}
\end{figure}

\newpage

\section*{Case 2.}

\begin{figure}[h!]
\centering
\includegraphics[scale=0.25]{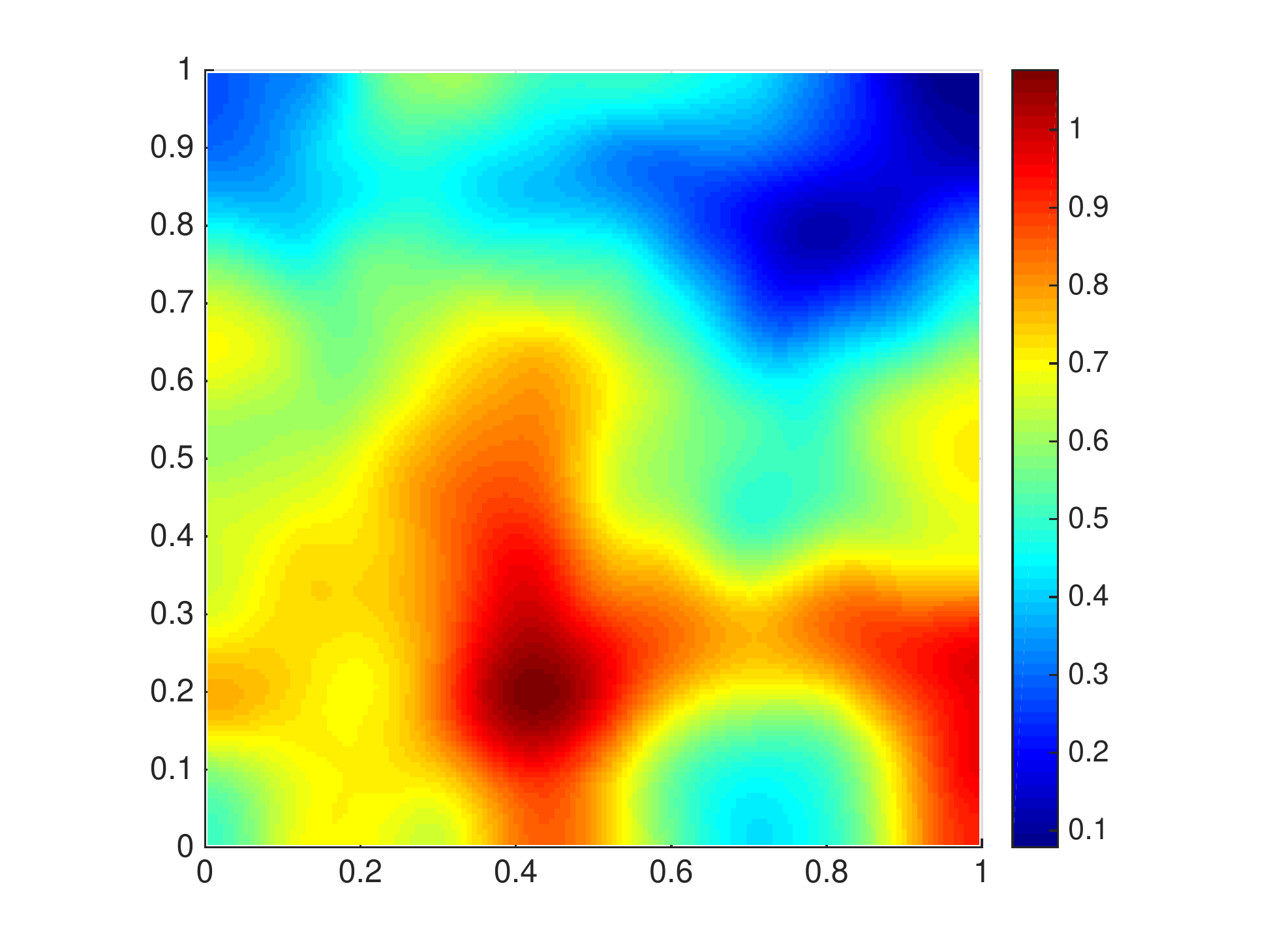}
\caption{Case 2. Gaussian random field truth.}
 \label{fig:truth_2_df}
\end{figure}

\begin{figure}[h!]
\centering
\includegraphics[width=\linewidth]{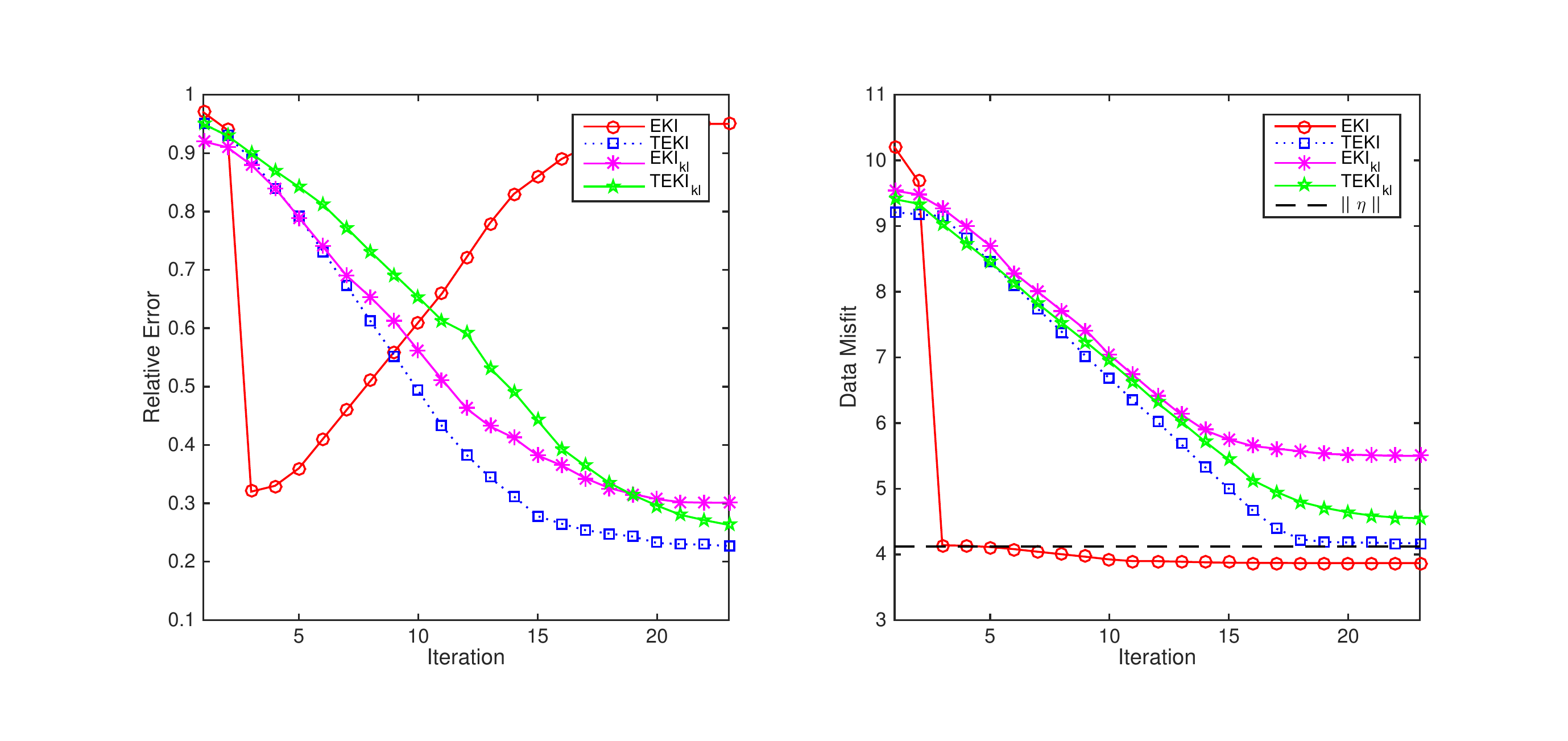}
\caption{Case 2. Relative errors and data misfits of each experiment.}
 \label{fig:RE_DM_2_df}
\end{figure}

\newpage

\begin{figure}[h!]
\centering
\includegraphics[width=\linewidth]{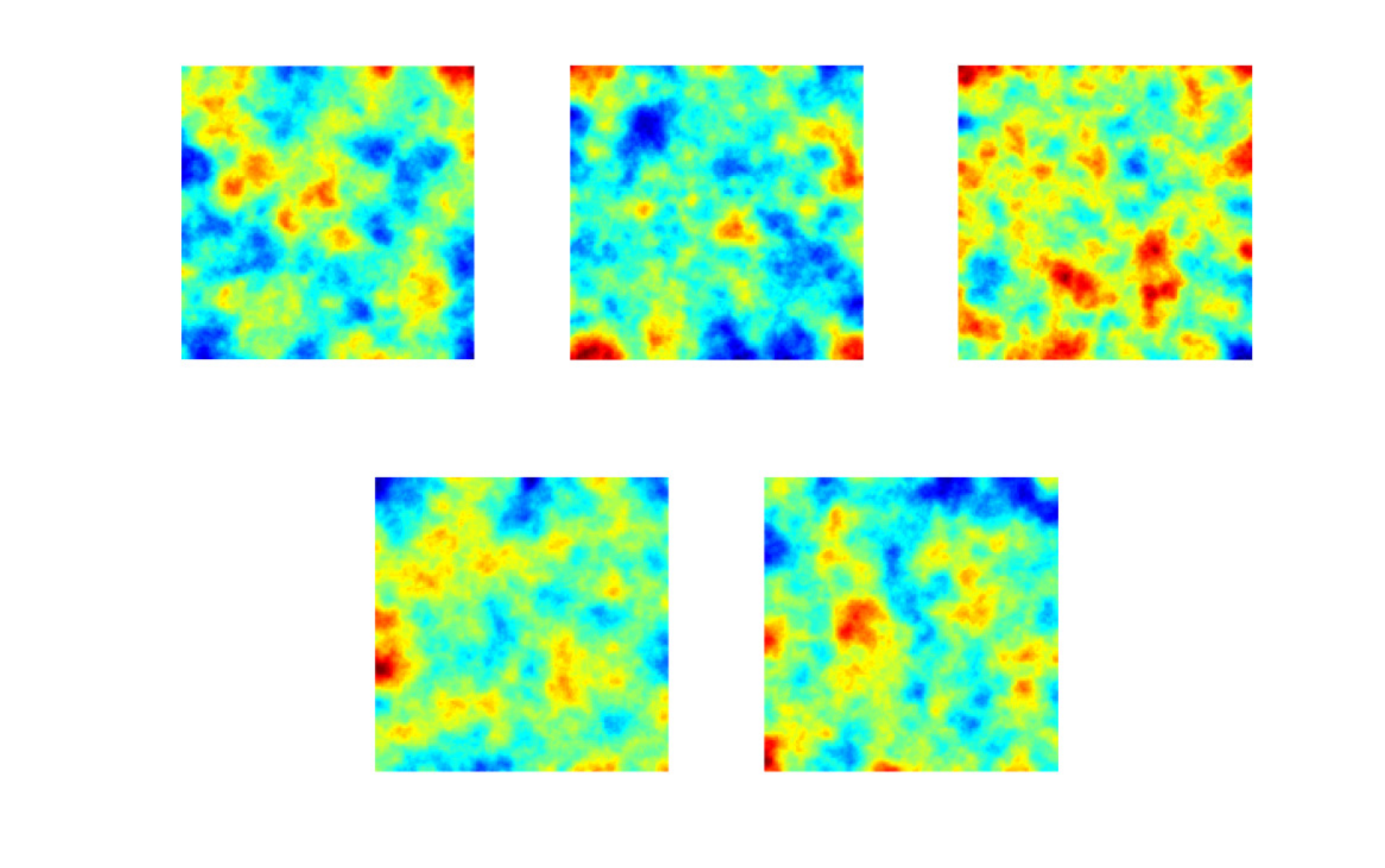}
\caption{Case 2. Progression of EKI through iteration count with prior random draws.}
 \label{fig:2EKI_RD_df}
\end{figure}

\begin{figure}[h!]
\centering
\includegraphics[width=\linewidth]{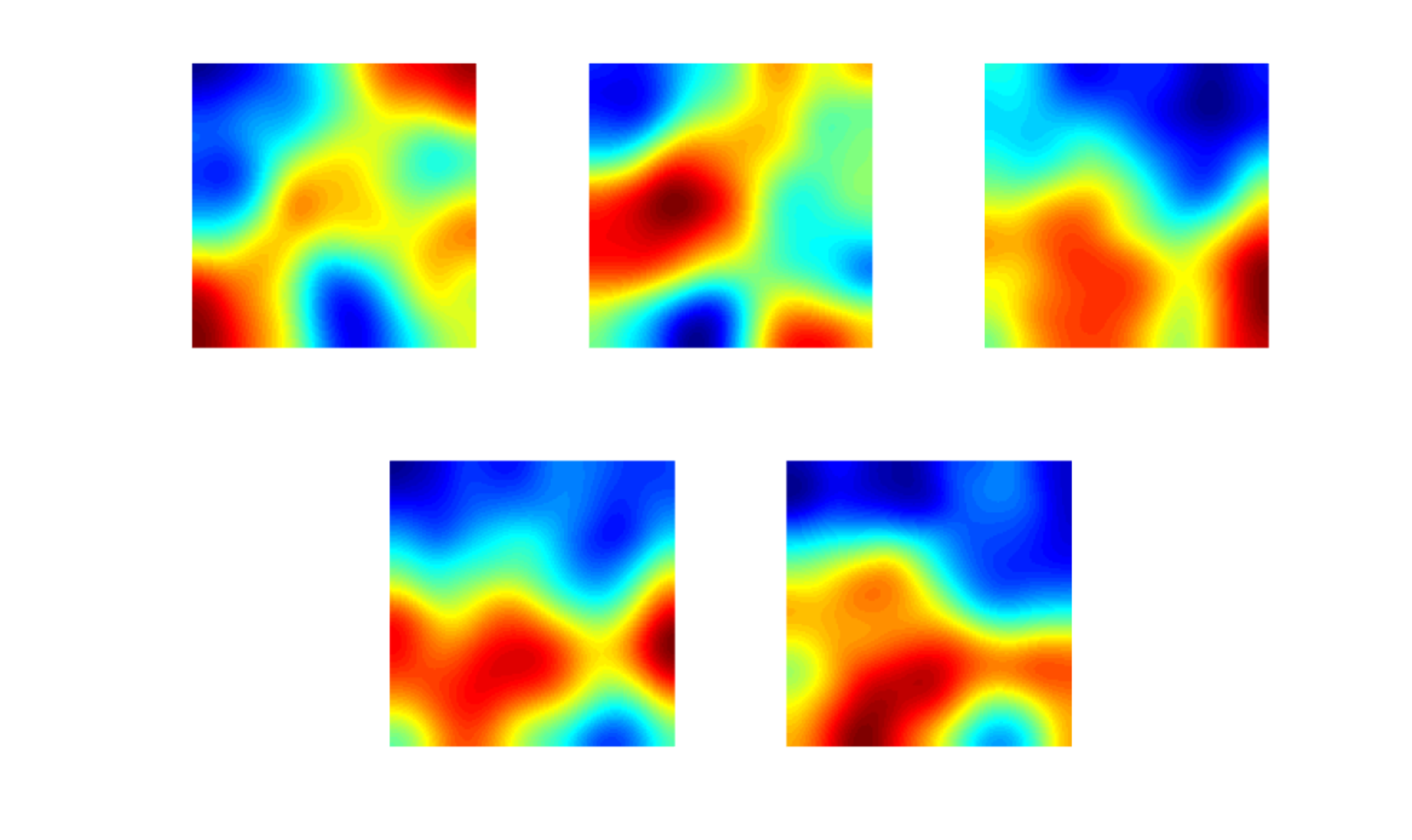}
\caption{Case 2. Progression of TEKI through iteration count with prior random draws.}
 \label{fig:2TEKI_RD_df}
\end{figure}

\newpage

\begin{figure}[h!]
\centering
\includegraphics[width=\linewidth]{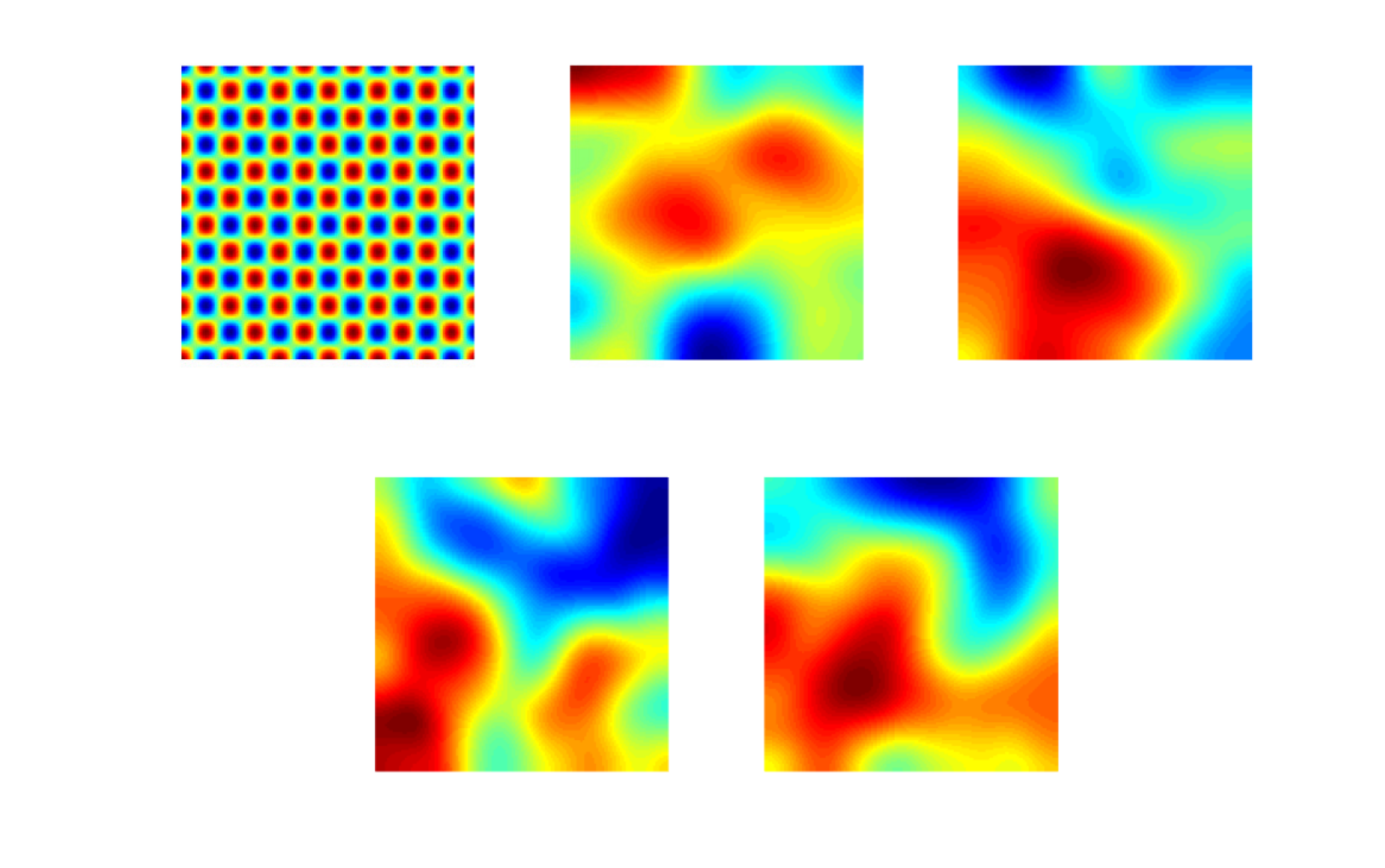}
\caption{Case 2. Progression of EKI through iteration count with prior from KL basis.}
 \label{fig:2EKI_KL_df}
\end{figure}

\begin{figure}[h!]
\centering
\includegraphics[width=\linewidth]{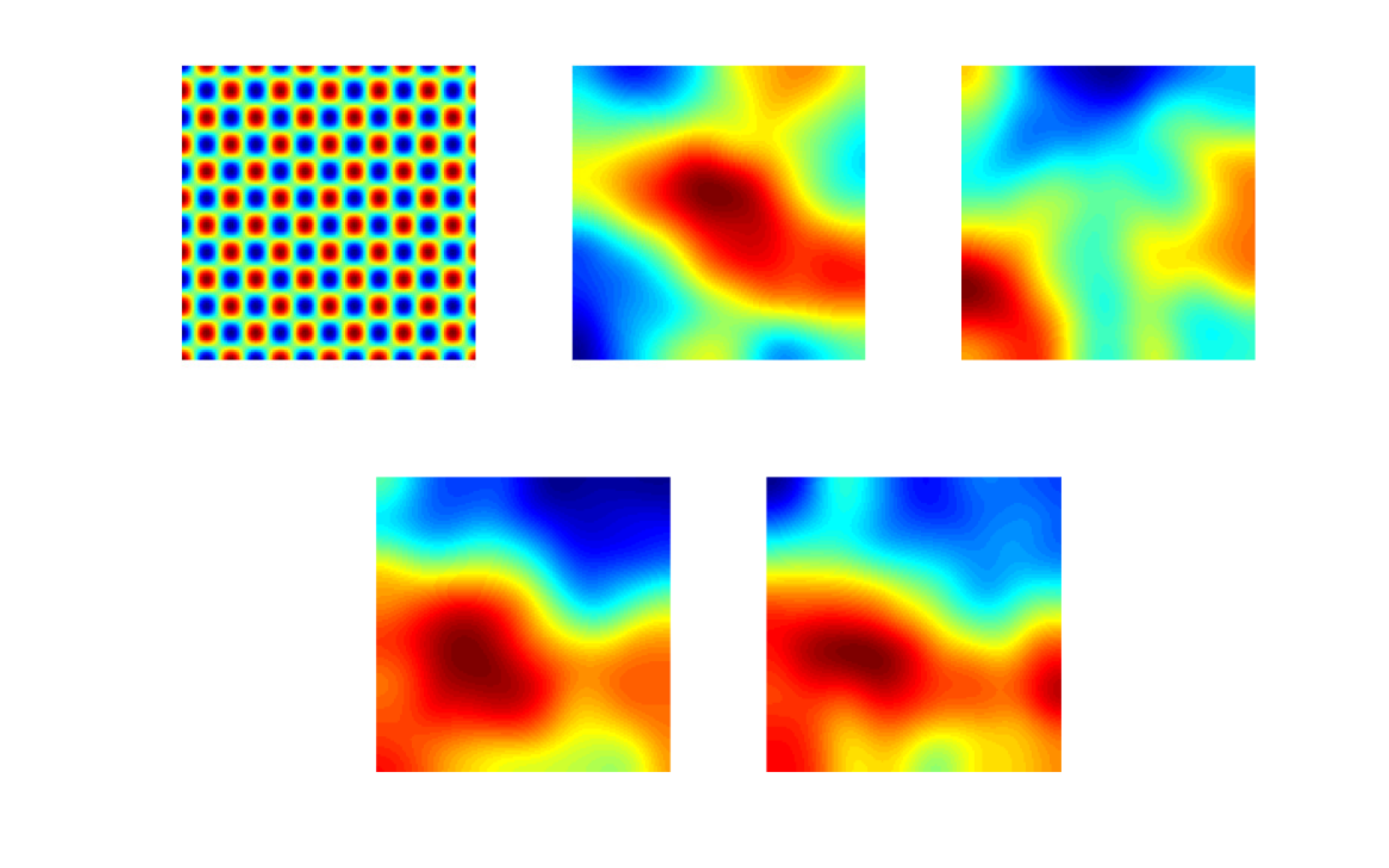}
\caption{Case 2. Progression of TEKI through iteration count with prior from KL basis.}
 \label{fig:2TEKI_KL_df}
\end{figure}

\newpage

\section*{Case 3.}

\begin{figure}[h!]
\centering
\includegraphics[scale=0.25]{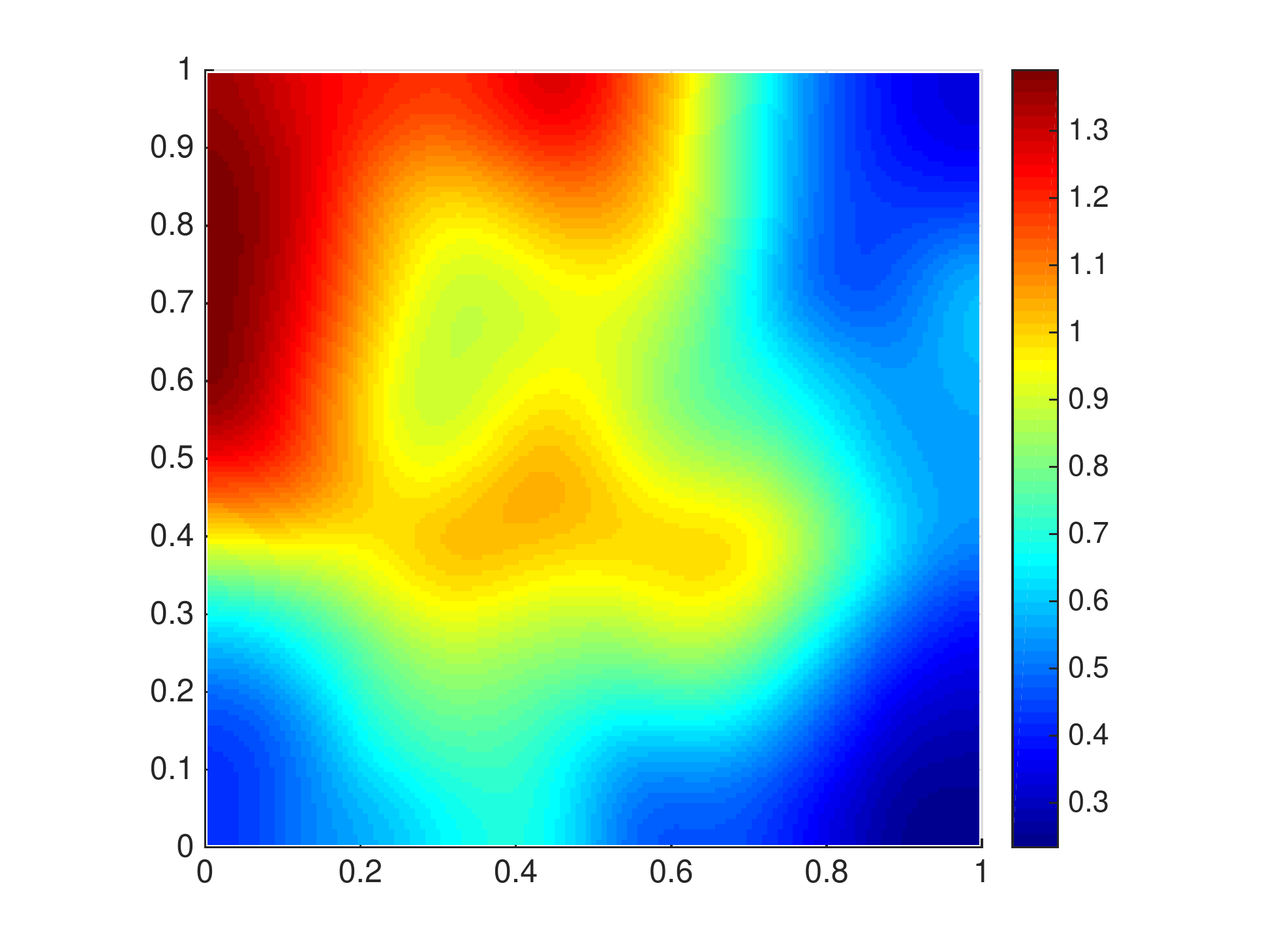}
\caption{Case 3. Gaussian random field truth.}
 \label{fig:truth_3_df}
\end{figure}

\begin{figure}[h!]
\centering
\includegraphics[width=\linewidth]{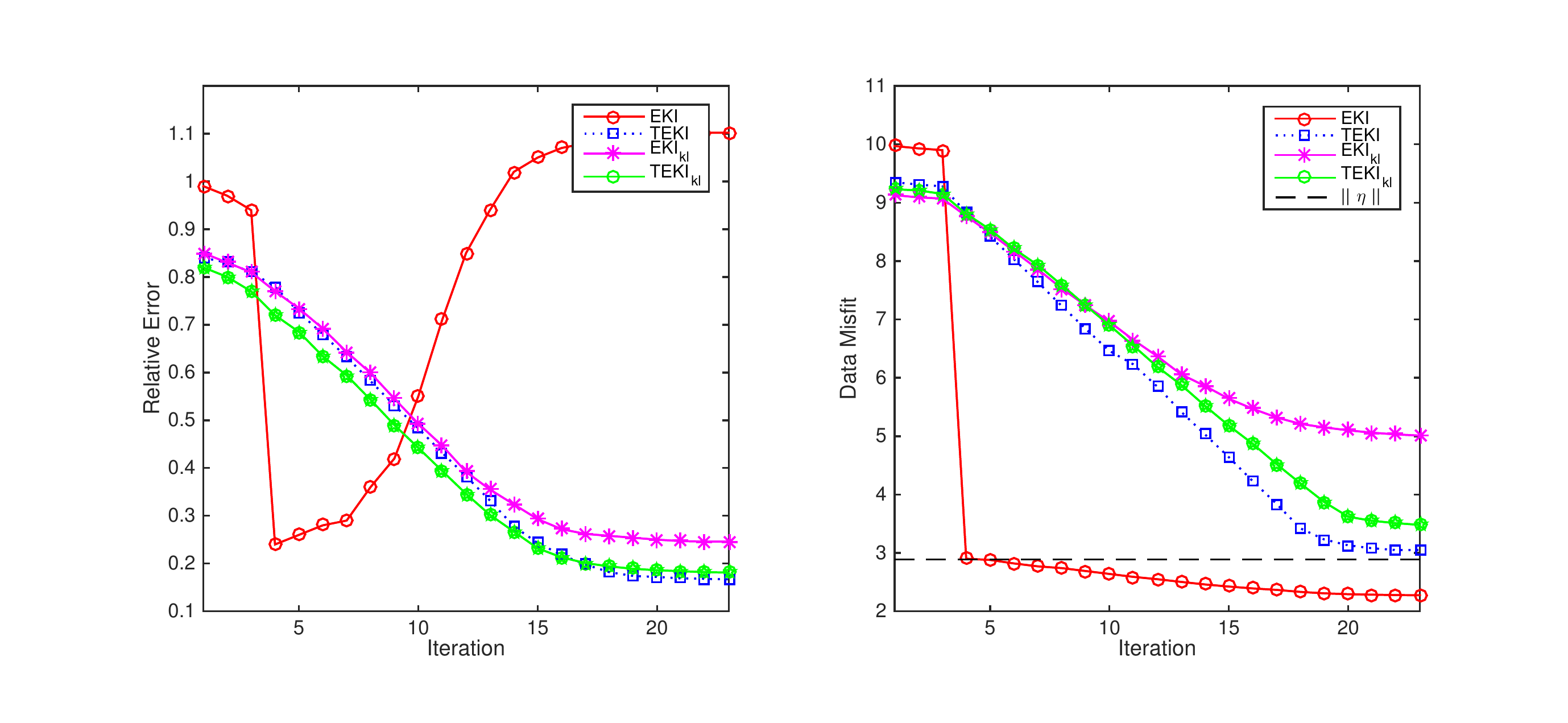}
\caption{Case 3. Relative errors and data misfits of each experiment.}
 \label{fig:RE_DM_3_df}
\end{figure}

\newpage

\begin{figure}[h!]
\centering
\includegraphics[width=\linewidth]{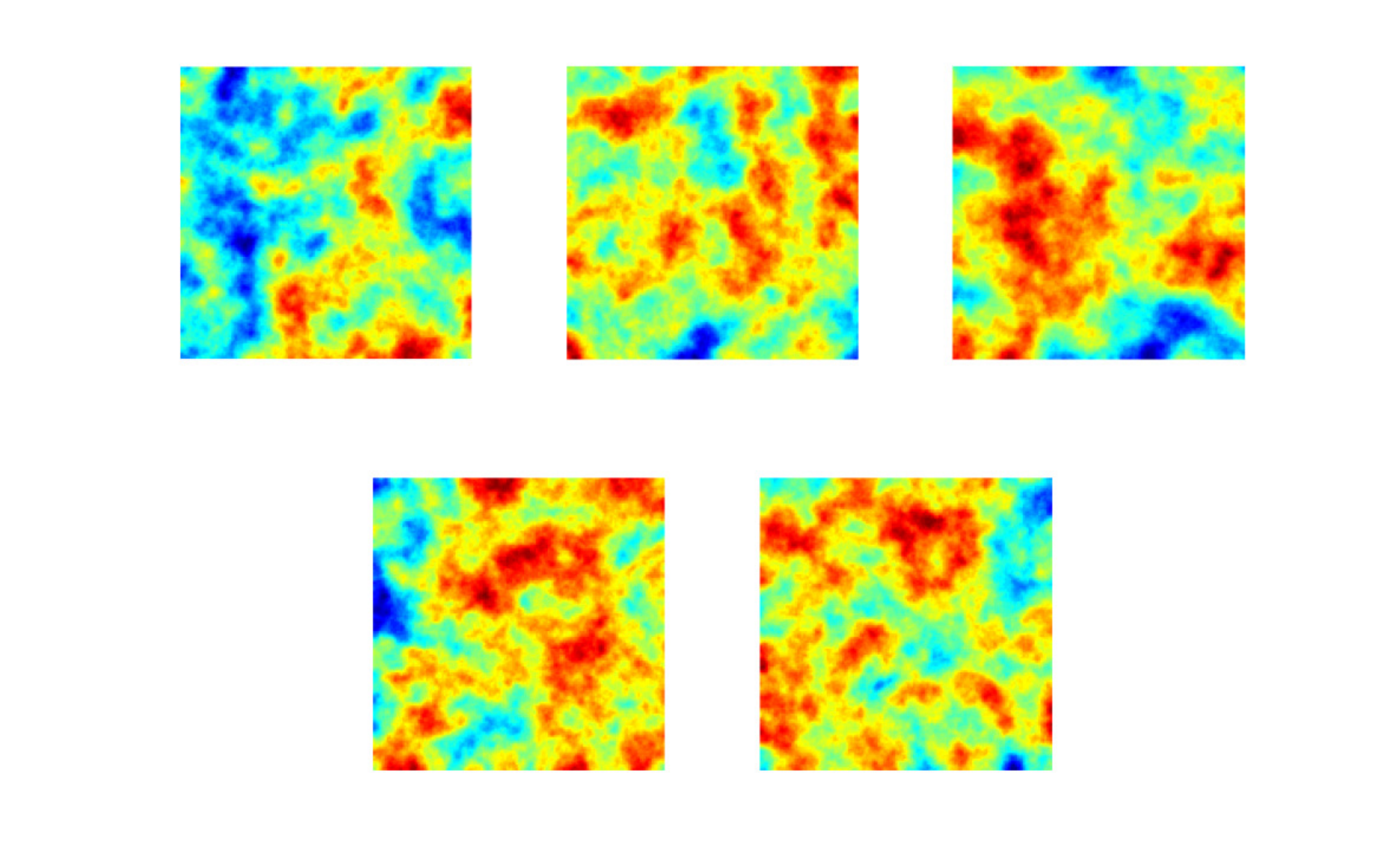}
\caption{Case 3. Progression of EKI through iteration count with prior random draws.}
 \label{fig:3EKI_RD_df}
\end{figure}

\begin{figure}[h!]
\centering
\includegraphics[width=\linewidth]{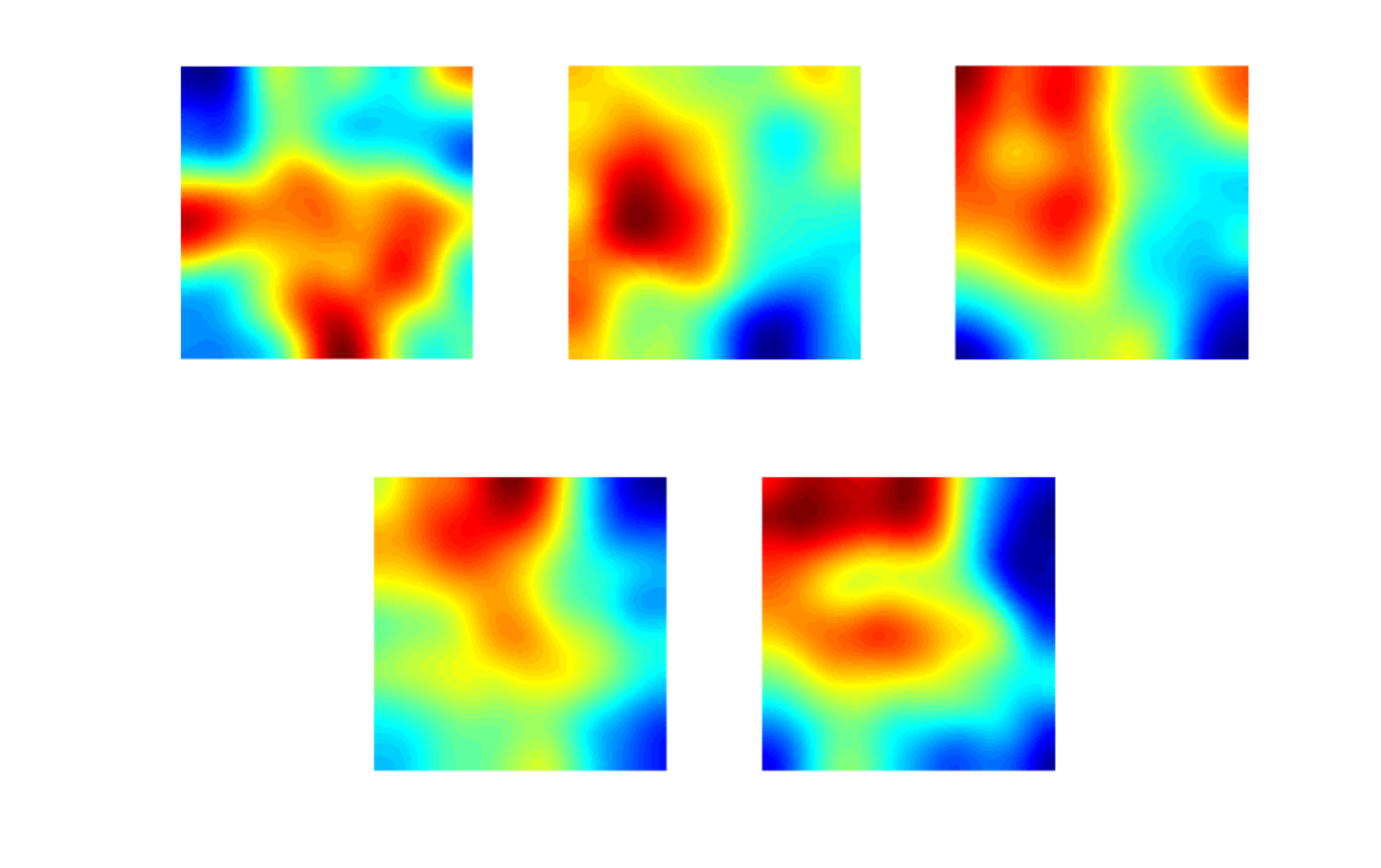}
\caption{Case 3. Progression of TEKI through iteration count with prior random draws.}
 \label{fig:3TEKI_RD_df}
\end{figure}

\newpage

\begin{figure}[h!]
\centering
\includegraphics[width=\linewidth]{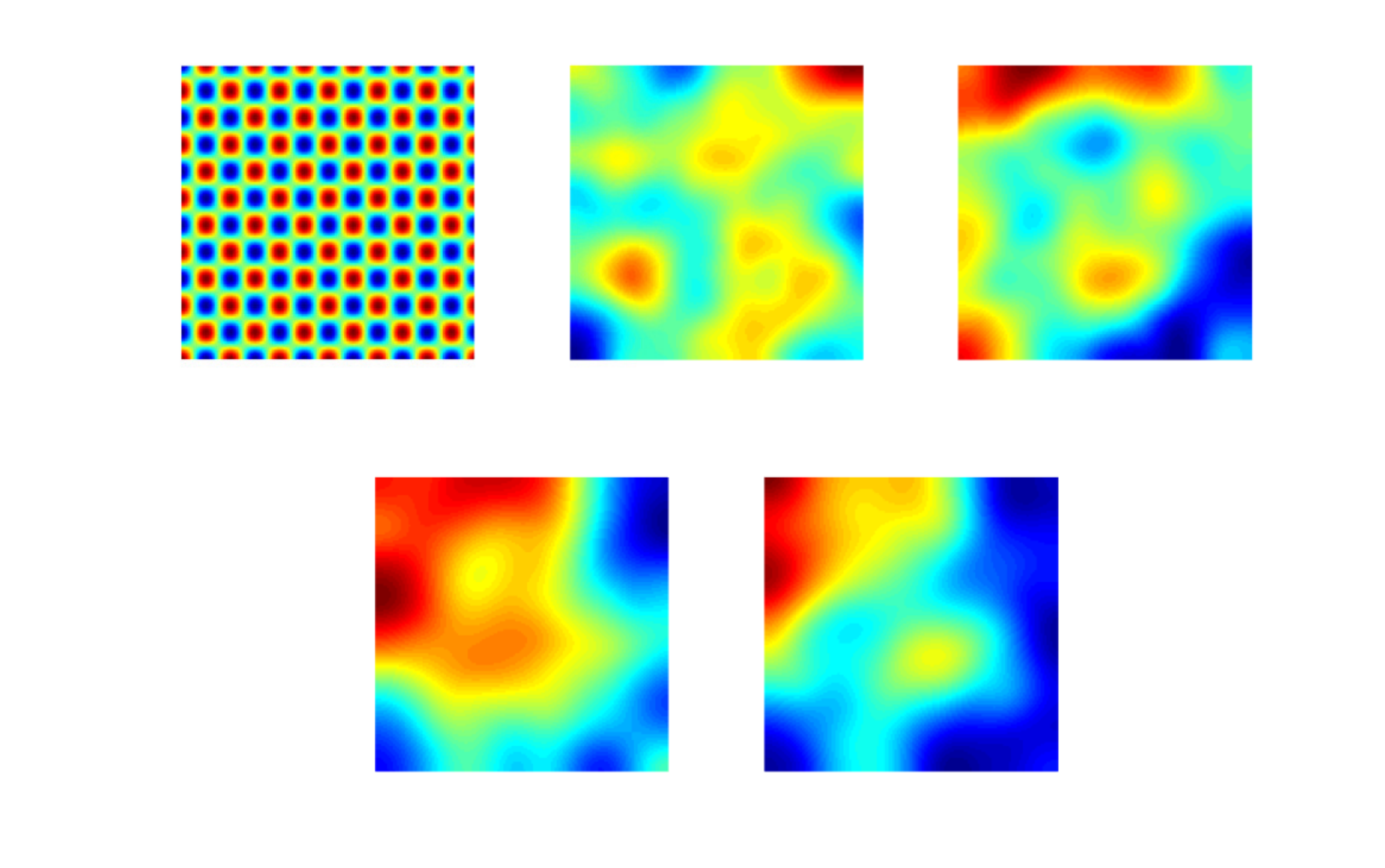}
\caption{Case 3. Progression of EKI through iteration count with prior from KL basis.}
 \label{fig:3EKI_KL_df}
\end{figure}

\begin{figure}[h!]
\centering
\includegraphics[width=\linewidth]{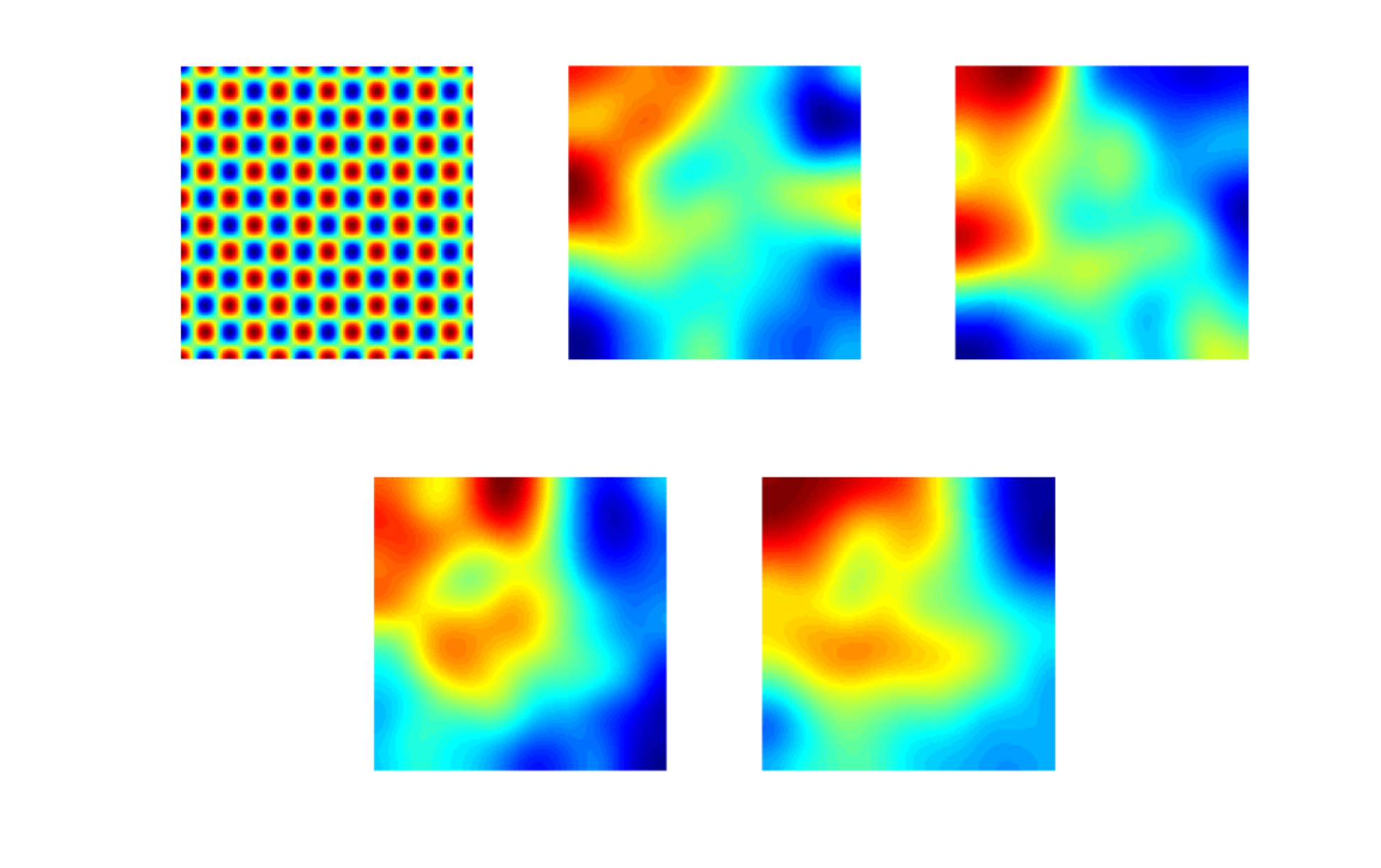}
\caption{Case 3. Progression of TEKI through iteration count with prior from KL basis.}
 \label{fig:3TEKI_KL_df}
\end{figure}

\end{document}